\def\eprint{1}  

\def\shownotes{0}

\ifnum\shownotes=1
\newcommand{\authnote}[2]{\textcolor{red}{\textsf{#1 }\textcolor{blue}{ #2}}}
\else
\newcommand{\authnote}[2]{}
\fi
\newcommand{\anote}[1]{{\authnote{Ali}{#1}}}
\newcommand{\ynote}[1]{{\authnote{Yilei}{#1}}}

\documentclass[11pt]{article}
\usepackage{amsmath,amssymb,amsfonts,amsthm}
\usepackage{url}
\usepackage{centernot}
\usepackage{courier}
\usepackage{color}
\usepackage{tikz}
\usepackage{hyperref}
\hypersetup{colorlinks=true,urlcolor=blue,citecolor=blue,linkcolor=blue}
\usepackage{listings}
\title{ Hard Isogeny Problems over RSA Moduli and \\ Groups with Infeasible Inversion}
\author{
	Salim Ali Altu\u{g}\thanks{Boston University. \texttt{saaltug@bu.edu}. Research supported by the grant DMS-1702176. } 
	\and Yilei Chen\thanks{Visa Research. \texttt{chenyilei.ra@gmail.com}. Research conducted while the author was at Boston University supported by the NSF MACS project and NSF grant CNS-1422965. } 
}
\date{\today}

\newcommand{\act}{\mathsf{act}}

\newcommand{\BE}{\mathsf{BE}}   

\newcommand{\Cert}{\mathsf{Cert}}  

\newcommand{\convert}{\mathsf{Ext}}  
\newcommand{\comp}{\mathsf{Compose}}  

\newcommand{\CRS}{\mathsf{CRS}}  
\newcommand{\CT}{\mathsf{CT}}  
\newcommand{\Dec}{\mathsf{Dec}}
\newcommand{\DS}{\mathsf{DS}}  
\newcommand{\DTS}{\mathsf{DTS}}  
\newcommand{\Enc}{\mathsf{Enc}}

\newcommand{\enc}{\mathsf{enc}}
\newcommand{\canenc}{\mathsf{enc}^*} 

\newcommand{\gcdop}{\mathsf{gcd.op}}
\newcommand{\Gen}{\mathsf{Gen}}
\newcommand{\TGII}{\mathsf{TGII}}  

\newcommand{\lad}{\mathsf{ladder}}  
\newcommand{\Lattice}{\Lambda}

\newcommand{\negl}{\mathsf{negl}}

\newcommand{\Partial}{\mathsf{Partial}}

\newcommand{\PK}{\mathsf{PK}}
\newcommand{\MPK}{\mathsf{MPK}} 
\newcommand{\PP}{\mathsf{PP}} 	

\newcommand{\Prove}{\mathsf{Prove}}

\newcommand{\Random}{\mathsf{Random}}
\newcommand{\Sam}{\mathsf{Sam}}  
\newcommand{\secp}{\lambda}  
\newcommand{\Setup}{\mathsf{Setup}} 
\newcommand{\Sig}{\mathsf{Sign}}

\newcommand{\SK}{\mathsf{SK}} 
\newcommand{\MSK}{\mathsf{MSK}} 

\newcommand{\Sym}{\mathsf{Sym}} 

\newcommand{\temp}{\mathsf{temp}}  
\newcommand{\Trap}{\mathsf{Trap}}  
\newcommand{\Ver}{\mathsf{Ver}}
\newcommand{\VK}{\mathsf{VK}}  
\newcommand{\Verify}{\mathsf{Verify}}
\newcommand{\zo}{\{0,1\}}

\newcommand{\la}{\leftarrow}

\newcommand{\ZNZ}{\Z/N\Z}

\newcommand{\Aut}{{\rm Aut}}

\newcommand{\C}{\mathbb{C}}
\newcommand{\CCC}{\mathcal{C}}
\newcommand{\charac}{{\rm char}}

\newcommand{\CRT}{\mathsf{CRT}}

\newcommand{\Ell}{{\rm Ell}}  
\newcommand{\Endo}{{\rm End}}

\newcommand{\F}{\mathbb{F}}

\newcommand{\G}{\mathbb{G}}		
\newcommand{\Gal}{{\rm Gal}}  

\newcommand{\mfk}[1]{\mathfrak{#1}}
\newcommand{\CL}{\mathcal{CL}}  

\newcommand{\kernel}{{\rm ker}}  
\newcommand{\kron}[2]{\genfrac{(}{)}{}{}{#1}{#2}}   

\newcommand{\mat}[1] { \mathbf{#1} }		

\newcommand{\ary}[1] { \mathbf{#1} }		

\newcommand{\N}{\mathbb{N}}

\newcommand{\OOO}{\mathcal{O}}
\newcommand{\poly}{\mathsf{poly}}

\newcommand{\Q}{\mathbb{Q}}
\newcommand{\QR}{\mathit{QR}}

\newcommand{\R}{\mathbb{R}}
\newcommand{\set}[1]{ \left\{ #1 \right\}  }   
\newcommand{\SL}{\mathrm{SL}}  

\newcommand{\uhp}{\mathbb{H}}  

\newcommand{\Z}{\mathbb{Z}}
\newcommand{\K}{\mathbf{k}}

\newcommand{\pmat}[1]{\begin{pmatrix}#1\end{pmatrix}}

\newtheorem{theorem}{Theorem}[section]
\newtheorem{algorithm}[theorem]{Algorithm}

\newtheorem{choice}[theorem]{Choice}

\newtheorem{construction}[theorem]{Construction}

\newtheorem{definition}[theorem]{Definition}
\newtheorem{example}[theorem]{Example}

\newtheorem{lemma}[theorem]{Lemma}

\newtheorem{proposition}[theorem]{Proposition}
\newtheorem{remark}[theorem]{Remark}

\begin{document}
\maketitle

\begin{abstract}

We initiate the study of computational problems on elliptic curve isogeny graphs defined over RSA moduli. We conjecture that several variants of the neighbor-search problem over these graphs are hard, and provide a comprehensive list of cryptanalytic attempts on these problems. Moreover, based on the hardness of these problems, we provide a construction of groups with infeasible inversion, where the underlying groups are the ideal class groups of imaginary quadratic orders. 

Recall that in a group with infeasible inversion, computing the inverse of a group element is required to be hard, while performing the group operation is easy. Motivated by the potential cryptographic application of building a directed transitive signature scheme, the search for a group with infeasible inversion was initiated in the theses of Hohenberger and Molnar (2003). Later it was also shown to provide a broadcast encryption scheme by Irrer et al. (2004). However, to date the only case of a group with infeasible inversion is implied by the much stronger primitive of self-bilinear map constructed by Yamakawa et al. (2014) based on the hardness of factoring and indistinguishability obfuscation (iO). Our construction gives a candidate without using iO. 

\end{abstract}

\newpage
\footnotesize
\tableofcontents
\normalsize
\newpage

\setcounter{page}{1}


\section{Introduction}

Let $\G$ denote a finite group written multiplicatively. The discrete-log problem asks to find the exponent $a$ given $g$ and $g^a\in\G$. In the groups traditionally used in discrete-log-based cryptosystems, such as $(\Z/q\Z)^*$ \cite{DBLP:journals/tit/DiffieH76}, groups of points on elliptic curves \cite{miller1985use,koblitz1987elliptic}, and class groups \cite{DBLP:journals/joc/BuchmannW88,mccurley1988cryptographic}, computing the inverse $x^{-1} = g^{-a}$ given $x = g^a$ is easy. We say $\G$ is a \emph{group with infeasible inversion} if computing inverses of elements is hard, while performing the group operation is easy (i.e. given $g$, $g^a$, $g^b$, computing $g^{a+b}$ is easy).

The search for a group with infeasible inversion was initiated in the theses of Hohenberger \cite{hohenberger2003cryptographic} and Molnar \cite{molnar2003homomorphic}, motivated with the potential cryptographic application of constructing a directed transitive signature. It was also shown by Irrer et al. \cite{ILOP04} to provide a broadcast encryption scheme. The only existing candidate of such a group, however, is implied by the much stronger primitive of self-bilinear maps constructed by Yamakawa et al. \cite{DBLP:conf/crypto/Yamakawa0HK14}, assuming the hardness of integer factorization and indistinguishability obfuscation (iO) \cite{DBLP:conf/crypto/BarakGIRSVY01,DBLP:conf/focs/GargGH0SW13}. 


In this paper we propose a candidate trapdoor group with infeasible inversion without using iO. The underlying group is isomorphic to the ideal class group of an imaginary quadratic order (henceforth abbreviated as ``\emph{the class group}"). In the standard representation\ifnum\eprint=1\footnote{By emphasizing the ``representation", we would like to remind the readers that the hardness of group theoretical problems (like the discrete-log problem) depends on the group representation rather than the group structure. After all, most of the cryptographically interesting finite groups are chosen to be isomorphic to the innocent looking additive group $\Z/n\Z$, $n\in\N$. However, the isomorphism is typically hard to compute.}\else\fi  ~of the class group, computing the inverse of a group element is straightforward. The representation we propose uses the volcano-like structure of the isogeny graphs of ordinary elliptic curves. In fact, the initiation of this work was driven by the desire to explore the computational problems on the isogeny graphs defined over RSA moduli.

\subsection{Elliptic curve isogenies in cryptography}

An isogeny $\varphi: E_1 \to E_2$ is a morphism of elliptic curves that preserves the identity. Given two isogenous elliptic curves $E_1$, $E_2$ over a finite field, finding an explicit rational polynomial that represents an isogeny from $E_1$ to $E_2$ is traditionally called the \emph{computational isogeny problem}. 

\ifnum\eprint=0
\else
The study of computing explicit isogenies began with the rather technical motivation of improving Schoof's polynomial time algorithm \cite{schoof1985elliptic} to compute the number of points on an elliptic curve over a finite field (the improved algorithm is usually called Schoof-Elkies-Atkin algorithm, cf. \cite{couveignes1994schoof,schoof1995counting,elkies1998elliptic} and references therein). A more straightforward use for explicit isogenies is to transfer the elliptic curve discrete-log problem from one curve to the other \cite{galbraith1999,DBLP:conf/eurocrypt/GalbraithHS02,DBLP:conf/asiacrypt/JaoMV05}. If for any two isogenous elliptic curves computing an isogeny from one to the other is efficient, then it means the discrete-log problem is equally hard among all the curves in the same isogeny class. 
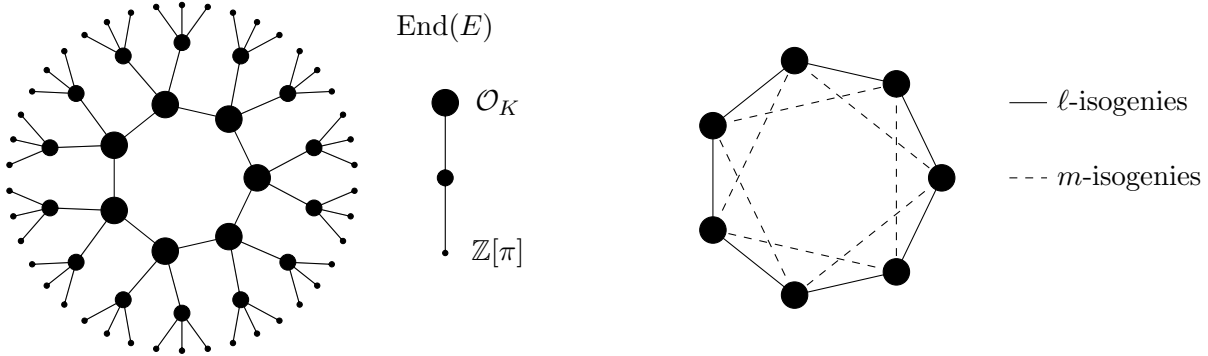
\begin{figure}
  \centering
    \begin{tikzpicture}
      \begin{scope}[xshift=0cm]
        \def\crater{7}
        \foreach \i in {1,...,\crater} {
          \draw[fill] (360/\crater*\i:1cm) circle (5pt);
          \draw (360/\crater*\i : 1cm) -- (360/\crater*\i+360/\crater : 1cm);
          \foreach \j in {-1,1} {
            \draw[fill] (360/\crater*\i : 1cm) -- (360/\crater*\i + \j*360/\crater/4 : 1.8cm) circle (3pt);
            \foreach \k in {-1,0,1} {
              \draw[fill] (360/\crater*\i + \j*360/\crater/4 : 1.8cm) --
              (360/\crater*\i + + \j*360/\crater/4 + \k*360/\crater/6 : 2.3cm) circle (1pt);
            }
          }
        }
      \end{scope}
      
      \begin{scope}[xshift=3.5cm]
        \node at (0,2) {$\Endo(E)$};
        \draw[fill] (0,1) circle(5pt) node[xshift=0.7cm]{$\OOO_K$} -- 
        (0,0) circle(3pt) --
        (0,-1) circle(1pt) node[xshift=0.7cm]{$\Z[\pi]$};
      \end{scope}

      \begin{scope}[xshift=8.5cm]
        \def\crater{7}
        \def\jumpa{-5}
        \def\diam{1.6cm}

        \foreach \i in {1,...,\crater} {
          \draw (360/\crater*\i : \diam) to (360/\crater*\i+360/\crater : \diam);
          \draw[dashed] (360/\crater*\i : \diam) to (360/\crater*\i+\jumpa*360/\crater : \diam);
        }
        \foreach \i in {1,...,\crater} {
          \pgfmathparse{int(mod(2^\i,\crater+1))}
         \let\exp\pgfmathresult
          \draw[fill] (360/\crater*\i: \diam) circle (5pt); 
        }
      \end{scope}

      \begin{scope}[xshift=11cm,yshift=1cm]
        \draw (0,0) -- (0.5,0) (0.5,0) node[black,anchor=west] {$\ell$-isogenies};
        \draw[dashed] (0,-1) -- (0.5,-1) (0.5,-1) node[black,anchor=west] {$m$-isogenies};
      \end{scope}
    \end{tikzpicture}
    \caption{Examples of isogeny graphs. Left: a connected component of $G_{\ell}(\K)$, and the corresponding tower of imaginary quadratic orders \cite{de2017mathematics}; Right: the vertex set is $\Ell_\OOO(\C)$ for an imaginary quadratic order $\OOO$, the edges represent (isomorphic classes of) isogenies of degrees $\ell$, $m$. }
  \label{fig:volcano_1}
\end{figure}

\fi

The best way of understanding the nature of the isogeny problem is to look at the \emph{isogeny graphs}. Fix a finite field $\K$ and a prime $\ell$ different than the characteristic of $\K$. Then the isogeny graph $G_{\ell}(\K)$ is defined as follows: each vertex in $G_{\ell}(\K)$ is a $j$-invariant of an isomorphism class of curves; two vertices are connected by an edge if there is an isogeny of degree $\ell$ over $\K$ that maps one curve to another. The structure of the isogeny graph is described in the PhD thesis of Kohel \cite{kohel1996endomorphism}. Roughly speaking, a connected component of an isogeny graph containing ordinary elliptic curves looks like a \emph{volcano} (termed in \cite{fouquet2002isogeny}). The connected component containing supersingular elliptic curves, on the other hand, has a different structure. In this article we will focus on the ordinary case.

\paragraph{A closer look at the algorithms of computing isogenies. }
Let $\K$ be a finite field of $q$ elements, $\ell$ be an integer such that $\gcd(\ell, q)=1$. 
Given the $j$-invariant of an elliptic curve $E$, there are at least two different ways to find all the $j$-invariants of the curves that are $\ell$-isogenous to $E$ (or to a twist of $E$) and to find the corresponding rational polynomials that represent the isogenies:
\begin{enumerate} 
\item Computing kernel subgroups of $E$ of size $\ell$, and then applying V\'{e}lu's formulae to obtain explicit isogenies and the $j$-invariants of the image curves, 
\item Calculating the $j$-invariants of the image curves by solving the $\ell^{th}$ modular polynomial $\Phi_\ell$ over $\K$, and then constructing explicit isogenies from these $j$-invariants. 
\end{enumerate}
Both methods are able to find all the $\ell$-isogenous neighbors over $\K$ in time $\poly(\ell, \log(q))$. In other words, \emph{over a finite field}, one can take a stroll around the polynomial-degree isogenous neighbors of a given elliptic curve efficiently. 

However, for two random isogenous curves over a sufficiently large field, finding an explicit isogeny between them seems to be hard, even for quantum computers. 
The conjectured hardness of computing isogenies was used in a key-exchange and a public-key cryptosystem by Couveignes \cite{cryptoeprint:2006:291} and independently by Rostovtsev and Stolbunov \cite{cryptoeprint:2006:145}. Moreover, a hash function and a key exchange scheme were proposed based on the hardness of computing isogenies over supersingular curves \cite{DBLP:journals/joc/CharlesLG09,DBLP:conf/pqcrypto/JaoF11}. Isogeny-based cryptography is attracting attention partially due to its conjectured post-quantum security.

\subsection{Isogeny graphs over RSA moduli}\label{subsec_isogenyvol}
Let $p,q$ be primes and let $N=pq$. In this work we consider computational problems related to elliptic curve isogeny graphs defined over $\Z/N\Z$, where the prime factors $p$, $q$ of $N$ are unknown. An isogeny graph over $\Z/N\Z$ is defined first by fixing the isogeny graphs over $\F_p$ and $\F_q$, then taking a graph tensor product; obtaining the $j$-invariants in the vertices of the graph over $\ZNZ$ by the Chinese remainder theorem. Working over the ring $\ZNZ$ without the factors of $N$ creates new sources of computational hardness from the isogeny problems. Of course, by assuming the hardness of factorization, we immediately lose the post-quantum privilege of the ``traditional'' isogeny problems. From now on all the discussions of hardness are with respect to the polynomial time classical algorithms. 

\paragraph{Basic neighbor search problem over $\ZNZ$.}
When the factorization of $N$ is unknown, it is not clear how to solve the basic problem of finding (even one of) the $\ell$-isogenous neighbors of a given elliptic curve. The two algorithms over finite fields we mentioned seem to fail over $\Z/N\Z$ since both of them require solving polynomials over $\Z/N\Z$, which is hard in general when the factorization of $N$ is unknown. In fact, we show that if it is feasible to find all the $\ell$-isogenous neighbors of a given elliptic curve over $\ZNZ$, then it is feasible to factorize $N$. 

\paragraph{Joint-neighbor search problem over $\ZNZ$.}
Suppose we are given several $j$-invariants over $\ZNZ$ that are connected by polynomial-degree isogenies, we ask whether it is feasible to compute their joint isogenous neighbors. For example, in the isogeny graph on the LHS of Figure~\ref{fig:joint_neighbor_intro}, suppose we are given $j_0$, $j_1$, $j_2$, and the degrees $\ell$ between $j_0$ and $j_1$, and $m$ between $j_0$ and $j_2$ such that $\gcd(\ell, m)=1$. Then we can find $j_3$ which is $m$-isogenous to $j_1$ and $\ell$-isogenous to $j_2$, by computing the polynomial $f(x) = \gcd( \Phi_{m}(j_1, x), \Phi_{\ell}(j_2, x) )$ over $\ZNZ$.  When $\gcd(\ell,m)=1$ the polynomial $f(x)$ turns out to be linear with its only root being $j_3$, hence computing the $(\ell,m)$ neighbor in this case is feasible.

\begin{figure}
  \centering
    \begin{tikzpicture}
      \begin{scope}[xshift=3cm, yshift = 0cm]
        \node at (0,0) {$\ZNZ$};
        \def\n{3}
        \def\angleshift{90}
        \foreach \i in {0,...,\n} {
          \pgfmathparse{360/(\n+3.5)*(3/2-\i) + \angleshift}
          \let\angle\pgfmathresult
          \draw (\angle:2.5) node (E\i) {$j_{\i}$};
        }
        \foreach \i in {0,2} {
          \pgfmathparse{int(\i+1)}
          \let\j\pgfmathresult
          \draw (E\i) -- node[auto,swap] {\scriptsize$\ell$} (E\j);
        }
        \foreach \i in {0,...,1} {
          \pgfmathparse{int(\i+2)}
          \let\j\pgfmathresult
          \draw[dashed] (E\i) -- node[auto,swap] {\scriptsize$m$} (E\j);
        }
      \end{scope}

      \begin{scope}[xshift=10cm, yshift = 0cm]
        \node at (0,0) {$\ZNZ$};
        \def\n{1}
        \def\angleshift{40}
          \pgfmathparse{360/(\n+5)*(3/2+1) + \angleshift}
          \let\angle\pgfmathresult
          \draw (\angle:2) node (E-1) {$j_{-1} = ~ ?$};
          \pgfmathparse{360/(\n+5)*(3/2-2) + \angleshift}
          \let\angle\pgfmathresult
          \draw (\angle:2) node (E2) {$ ... $};
        \foreach \i in {0,...,\n} {
          \pgfmathparse{360/(\n+5)*(3/2-\i) + \angleshift}
          \let\angle\pgfmathresult
          \draw (\angle:2) node (E\i) {$j_{\i}$};
        }
        \foreach \i in {1,...,2} {
          \pgfmathparse{360/(\n+12)*(3/2-\i) + \angleshift + 90}
          \let\angle\pgfmathresult
          \draw (\angle:3.5) node (H\i) {$j_{0, \i} = ~ ?$};
        }
        \foreach \i in {1,...,2} {
          \pgfmathparse{360/(\n+12)*(3/2-\i) + \angleshift + 30}
          \let\angle\pgfmathresult
          \draw (\angle:3.5) node (G\i) {$ ... $};
        }
        \foreach \i in {-1,...,1} {
          \pgfmathparse{int(\i+1)}
          \let\j\pgfmathresult
          \draw (E\i) -- node[auto,swap] {\scriptsize$\ell$} (E\j);
        }
          \draw (E0) -- node[auto,swap] {\scriptsize$\ell$} (H1);
          \draw (E0) -- node[auto,swap] {\scriptsize$\ell$} (H2);
          \draw (E1) -- node[auto,swap] {\scriptsize$\ell$} (G1);
          \draw (E1) -- node[auto,swap] {\scriptsize$\ell$} (G2);
      \end{scope}
    \end{tikzpicture}
  
    \caption{Left: the $(\ell, m)$-isogenous neighbor problem where $\gcd(\ell, m) = 1$. Right: the $(\ell, \ell^2)$-isogenous neighbor problem. }
  \label{fig:joint_neighbor_intro}
\end{figure}
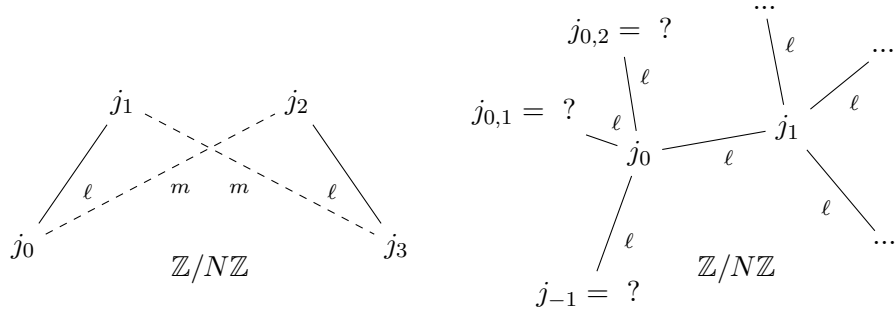

However, not all the joint-isogenous neighbors are easy to find. As an example, consider the following \emph{$(\ell, \ell^2)$-joint neighbor problem} illustrated on the RHS of Figure~\ref{fig:joint_neighbor_intro}. Suppose we are given $j_0$ and $j_1$ that are $\ell$-isogenous, and asked to find another $j$-invariant $j_{-1}$ which is $\ell$-isogenous to $j_0$ and $\ell^2$-isogenous to $j_1$. The natural way is to take the gcd of $\Phi_{\ell}(j_0, x)$ and $\Phi_{\ell^2}(j_1, x)$, but in this case the resulting polynomial is of degree $\ell>1$ and we are left with the problem of finding a root of it over $\ZNZ$, which is believed to be computationally hard without knowing the factors of $N$. 

Currently we do not know if solving this problem is as hard as factoring $N$. Neither do we know of an efficient algorithm of solving the $(\ell, \ell^2)$-joint neighbor problem. We will list our attempts in solving the $(\ell, \ell^2)$-joint neighbor problem in Section~\ref{sec:analysis_ell_ell2}.

The conjectured computational hardness of the $(\ell, \ell^2)$-joint neighbor problem is fundamental to the infeasibility of inversion in the group we construct.

\subsection{Constructing a trapdoor group with infeasible inversion}\label{sec:intro:tgii}

To explain the construction of the trapdoor group with infeasible inversion (TGII), it is necessary to recall the connection of the ideal class groups and elliptic curve isogenies. Let $\K$ be a finite field as before and let $E$ be an elliptic curve over $\K$ whose endomorphism ring is isomorphic to an imaginary quadratic order $\OOO$. The group of invertible $\OOO$-ideals acts on the set of elliptic curves with endomorphism ring $\OOO$. The ideal class group $\CL(\OOO)$ acts faithfully and transitively on the set 
\[ \Ell_\OOO(\K) = \set{ j(E) : E\text{ with } \Endo(E) \simeq \OOO }.\]
In other words, there is a map
\[
    \CL(\OOO)\times \Ell_\OOO(\K) \to \Ell_\OOO(\K), ~~~ (\mfk{a}, j) \mapsto \mfk{a} * j
\]
such that $\mfk{a}*(\mfk{b} * j)=(\mfk{a}\mfk{b}) * j$ for all $\mfk{a},\mfk{b}\in\CL(\OOO)$ and $j\in\Ell_\OOO(\K)$; for any $j, j'\in \Ell_\OOO(\K)$, there is a unique $\mfk{a}\in\CL(\OOO)$ such that $j' = \mfk{a} * j$. The cardinality of $\Ell_\OOO(\K)$ equals to the class number $h(\OOO)$.

We are now ready to provide an overview of the TGII construction with a toy example in Figure~\ref{fig:volcano_CRT}. 

{\bf Parameter generation.} 
To simplify this overview let us assume that the group $\CL(\OOO)$ is cyclic, in which case the group $\G$ with infeasible inversion is exactly $\CL(\OOO)$ (in the detailed construction we usually choose a cyclic subgroup of $\CL(\OOO)$). 
To generate the public parameter for the group $\CL(\OOO)$, we choose two primes $p$, $q$ and curves $E_{0,\F_p}$ over $\F_p$ and $E_{0,\F_q}$ over $\F_q$ such that the endomorphism rings of $E_{0,\F_p}$ and $E_{0,\F_q}$ are both isomorphic to $\OOO$. Let $N = p\cdot q$. Let $E_0$ be an elliptic curve over $\ZNZ$ as the CRT composition of $E_{0,\F_p}$ and $E_{0,\F_q}$. The $j$-invariant of $E_0$, denoted as $j_0$, equals to the $\CRT$ composition of the $j$-invariants of $E_{0,\F_p} $ and $E_{0,\F_q}$. 
The identity of $\CL(\OOO)$ is represented by $j_0$. The public parameter of the group is $(N, j_0)$. 

In the example of Figure~\ref{fig:volcano_CRT}, we set the discriminant $D$ of the imaginary quadratic order $\OOO$ to be $-251$. The group order is then the class number $h(\OOO) = 7$. Choose $p = 83$, $q = 173$, $N = pq = 14359$. Fix a curve $E_0$ so that $j(E_{0,\F_p}) = 15$, $j(E_{0,\F_q}) = 2$, then $j_0 = \CRT(83,173;15,2) = 12631$. The public parameter is $(14359, 12631)$.

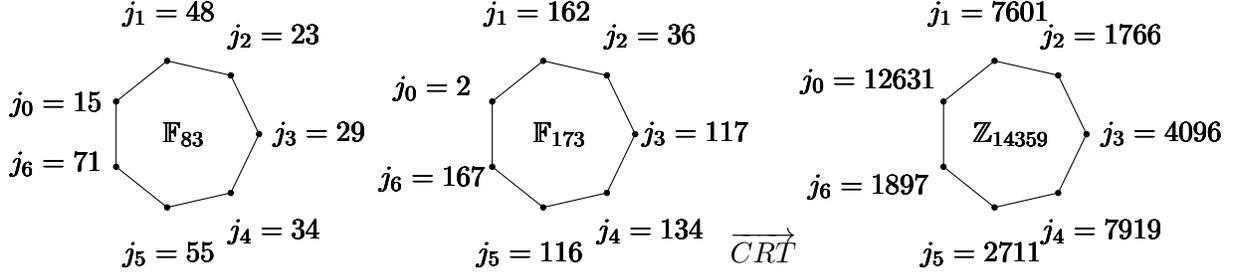
\begin{figure}[t]
  \centering
    \begin{tikzpicture}
      \def \diameter{1cm}
      \begin{scope}[xshift= 0 cm]
        \def\crater{7}
        \foreach \i in {1,...,\crater} {
          \draw[fill] (360/\crater*\i:\diameter) circle (1pt);
          \draw (360/\crater*\i : \diameter) -- (360/\crater*\i+360/\crater : \diameter);
                    \node at (0,0) {$\F_{83}$};
                    \node at (-1.7, 0.4) {$j_0 = 15$};
                    \node at (-0.2, 1.6) {$j_1 = 48$};
                    \node at ( 1.2, 1.3) {$j_2 = 23$};
                    \node at ( 1.8, 0.0) {$j_3 = 29$};
                    \node at ( 1.2,-1.3) {$j_4 = 34$};
                    \node at (-0.2,-1.6) {$j_5 = 55$};
                    \node at (-1.7,-0.4) {$j_6 = 71$};
        }
      \end{scope}
      \begin{scope}[xshift= 5 cm]
        \def\crater{7}
        \foreach \i in {1,...,\crater} {
          \draw[fill] (360/\crater*\i:\diameter) circle (1pt);
          \draw (360/\crater*\i : \diameter) -- (360/\crater*\i+360/\crater : \diameter);
                    \node at (0,0) {$\F_{173}$};
                    \node at (-1.7, 0.6) {$j_0 = 2$};
                    \node at (-0.3, 1.6) {$j_1 = 162$};
                    \node at ( 1.2, 1.3) {$j_2 = 36$};
                    \node at ( 1.8, 0.0) {$j_3 = 117$};
                    \node at ( 1.2,-1.3) {$j_4 = 134$};
                    \node at (-0.4,-1.6) {$j_5 = 116$};
                    \node at (-1.7,-0.6) {$j_6 = 167$};
        }
      \end{scope}
      \begin{scope}[xshift= 7.7 cm]
      \node at (0, -1.5) {$\overrightarrow{CRT}$};
      \end{scope}
      \begin{scope}[xshift= 11 cm]
        \def\crater{7}
        \foreach \i in {1,...,\crater} {
          \draw[fill] (360/\crater*\i:\diameter) circle (1pt);
          \draw (360/\crater*\i : \diameter) -- (360/\crater*\i+360/\crater : \diameter);
                    \node at (0,0) {$\Z_{14359}$};
                    \node at (-1.9, 0.7) {$j_0 = 12631$};
                    \node at (-0.3, 1.6) {$j_1 = 7601$};
                    \node at ( 1.2, 1.3) {$j_2 = 1766$};
                    \node at ( 2.0, 0.0) {$j_3 = 4096$};
                    \node at ( 1.2,-1.3) {$j_4 = 7919$};
                    \node at (-0.4,-1.6) {$j_5 = 2711$};
                    \node at (-1.9,-0.7) {$j_6 = 1897$};
        }
      \end{scope}
    \end{tikzpicture}
    \caption{A representation of $\CL(-251)$ by a $3$-isogeny volcano over $\Z_{14359}$ of size $h(-251)=7$. The $\F_{83}$ part is taken from \cite{cryptoeprint:2006:145}. 
    }\label{fig:volcano_CRT}
\end{figure}

{\bf The encodings. }
We provide two types of encodings for each group element: the canonical and composable embeddings. 
The \emph{canonical encoding} of an element is uniquely determined once the public parameter is fixed and it can be directly used in the equivalence test. It, however, does not support efficient group operations. The \emph{composable encoding} of an element, on the other hand, supports efficient group operations with the other composable encodings. Moreover, a composable encoding can be converted to a canonical encoding by an efficient, public extraction algorithm.

An element $x \in\CL(\OOO)$ is canonically represented by the $j$-invariant of the elliptic curve $ x * E_0$ (once again, obtained over $\F_p$ and $\F_q$ then composed by $\CRT$), and we call $j(x * E_0)$ the canonical encoding of $x$. Note that the canonical encodings of all the elements are fixed once $j_0$ and $N$ are fixed. 

To make things concrete, let $a = \sqrt{-251}$ and consider the toy example above. The ideal class $x = [(3, \frac{a+1}{2})]$ acting on $E_0$ over $\F_p$ gives $j(x * E_{0,\F_p}) = 48$, over $\F_q$ gives $j(x * E_{0,\F_q}) = 162$. The canonical encoding of $x$ is then $j_1 = \CRT(83,173;48,162) = 7601$. 
Similarly, the canonical encodings of the ideal classes $[(7, \frac{a-1}{2})]$, $[(5, \frac{a+7}{2})]$, $[(5, \frac{a+3}{2})]$, $[(7, \frac{a+1}{2})]$, $[(3, \frac{a-1}{2})]$ are $1766, 4096, 7919, 2711, 1897$.

{\bf The composable encodings and the composition law.}
To generate a composable encoding of $x \in\CL(\OOO)$, we factorize $x$ as $x = \prod_{x_i\in S} x_i^{e_i}$, where $S$ denotes a generating set, 
and both the norms $N(x_i)$ and the exponents $e_i$ being polynomial in size. The composable encoding of $x$ then consists of the norms $N(x_i)$ and the $j$-invariants of $x_i^{k} * E_0$, for $k\in[e_i]$, for $i\in [|S|]$. The \emph{degree} of a composable encoding is defined to be the product of the norms of the ideals $\prod_{x_i\in S}N(x_i)^{e_i}$. Note that the degree depends on the choice of $S$ and the factorization of $x$, which is not unique. 

As an example let us consider the simplest situation, where the composable encodings are just the canonical encodings themselves together with the norms of the ideals (i.e. the degrees of the isogenies). Set the composable encoding of $x = [(3, \frac{a + 1}{2})]$ be $(3, 7601)$, the composable encoding of $y = [(7, \frac{a-1}{2})]$ be $(7, 1766)$.

Let us remark an intricacy of the construction of composable encodings. When the degrees of the composable encodings of $x$ and $y$ are coprime and polynomially large, the composition of $x$ and $y$ can be done simply by concatenating the corresponding encodings. To extract the canonical encoding of $x\circ y$, we take the gcd of the modular polynomials. In the example above, the canonical encoding of $x\circ y$ can be obtained by taking the gcd of $\Phi_7(7601, x)$ and $\Phi_3(1766, x)$ over $\ZNZ$. Since the degrees are coprime, the resulting polynomial is linear, with the only root being $4096$, which is the canonical encoding of $[(5, \frac{a+7}{2})]$.

Note, however, that if the degrees share prime factors, then the gcd algorithm does not yield a linear polynomial, so the the above algorithm for composition does not go through. To give a concrete example to what this means let us go back to our example: if we represent $y=[(7, \frac{a-1}{2})]$ by first factorizing $y$ as $[(3, \frac{a + 1}{2})]^2$ we then get the composable encoding of $y$ as $(3, (7601,1766))$. In this case the gcd of $\Phi_{3^2}(7601, x)$ and $\Phi_3(1766, x)$ over $\ZNZ$ yields a degree $3$ polynomial, where it is unclear how to extract the roots. Hence, in this case we cannot calculate the canonical embedding of $x\circ y$ simply by looking at the gcd.

Therefore, to facilitate the efficient compositions of the encodings of group elements, we will need to represent them as the product of \emph{pairwise co-prime} ideals with polynomially large norms. This, in particular, means the encoding algorithm will need to keep track on the primes used in the degrees of the composable encodings in the system. In other words, the encoding algorithm is \emph{stateful}.

{\bf The infeasibility of inversion.} The infeasibility of inversion amounts to the hardness of the computation of the canonical embedding of an element $x^{-1}\in\G$ from a composable encoding of $x$, and it is based on the hardness of the $(\ell, \ell^2)$-isogenous neighbors problem for each ideal of a composable encoding.

Going back to our example, given the composable encoding $(3, 7601)$ of $x = [(3, \frac{a + 1}{2})]$, the canonical encoding of $x^{-1} = [(3, \frac{a - 1}{2})]$ is a root of $f(x) = \gcd( \Phi_{3^2}(7601, x), \Phi_3(12631, x))$. The degree of $f$, however, is $3$, so that it is not clear how to extract the root efficiently over an RSA modulus.

\paragraph{The difficulty of sampling the class group invariants and its implications.}
Let us remark that the actual instantiation of TGII is more involved. A number of challenges arise solely from working with the ideal class groups of imaginary quadratic orders. To give a simple example of the challenges we face, efficiently generating a class group with a known large prime class number is a well-known open problem. Additionally, our construction requires more than the class number (namely, a short basis of the relation lattice of the class group) to support an efficient encoding algorithm. 

In our solution, we choose the discriminant $D$ to be of size roughly $\secp^{O(\log \secp)}$ and polynomially smooth, so as to make the parameter generation algorithm and the encoding algorithm run in polynomial time. The discriminant $D$ (i.e. the description of the class group $\CL(D)$) has to be hidden to preserve the plausible $\secp^{O(\log \secp)}$-security of the TGII. Furthermore, even if $D$ is hidden, there is an $\secp^{O(\log \secp)}$ attack by first guessing $D$ or the group order, then solving the discrete-log problem given the polynomially-smooth group order. Extending the working parameters regime seems to require the solutions of several open problems concerning ideal class groups of imaginary quadratic orders. 


\paragraph{Summary of the TGII construction.}
To summarize, our construction of TGII chooses two sets of $j$-invariants that correspond to elliptic curves with the same imaginary quadratic order $\OOO$ over $\F_p$ and $\F_q$, and glues the $j$-invariants via the CRT composition as the canonical encodings of the group elements in $\CL(\OOO)$. 
The composable encoding of a group element $x$ is given as several $j$-invariants that represent the smooth ideals in a factorization of $x$. The efficiency of solving the $(\ell, m)$-joint-neighbor problem over $\ZNZ$ facilitates the efficient group operation over coprime degree encodings. The conjectured hardness of the $(\ell, \ell^2)$-joint-neighbor problem over $\ZNZ$ is the main reason behind the hardness of inversion, but it also stops us from composing encodings that share prime factors.

The drawbacks of our construction of TGII are as follows.
\begin{enumerate}
\item Composition is only feasible for coprime-degree encodings, which means in order to publish arbitrarily polynomially many encodings, the encoding algorithm has to be stateful in order to choose different polynomially large prime factors for the degrees of the encoding (we cannot choose polynomially large prime degrees and hope they are all different). 
\item In the definition from \cite{hohenberger2003cryptographic,molnar2003homomorphic}, the composable encodings obtained during the composition are required to be indistinguishable to a freshly sampled encoding. In our construction the encodings keep growing during compositions, until they are extracted to the canonical encoding which is a single $j$-invariant.
\item In addition to the $(\ell, \ell^2)$-joint-neighbor problem, the security of the TGII construction relies on several other heuristic assumptions. We will list our cryptanalytic attempts in \S\ref{sec:attack_concrete_TGII}. Moreover, even if we have not missed any attacks, the current best attack only requires $\secp^{O(\log \secp)}$-time, by first guessing the discriminant or the group order.
\end{enumerate}

\paragraph{The two applications of TGII.}
Let us briefly mention the impact of the limitation of our TGII on the applications of directed transitive signature (DTS) \cite{hohenberger2003cryptographic,molnar2003homomorphic} and broadcast encryption \cite{ILOP04}. 
For the broadcast encryption from TGII \cite{ILOP04}, the growth of the composable encodings do not cause a problem. 
For DTS, in the direct instantiation of DTS from TGII \cite{hohenberger2003cryptographic,molnar2003homomorphic}, the signature is a composable encoding, so the length of the signature keeps growing during the composition, which is an undesirable feature for a non-trivial DTS. So on top of the basic instantiation, we provide an additional compression technique to shrink the composed signature.

Let us also remark that in the directed transitive signature \cite{hohenberger2003cryptographic,molnar2003homomorphic}, encodings are sampled by the master signer; in the broadcast encryption scheme \cite{ILOP04}, encodings are sampled by the master encrypter. At least for these two applications, having the master signer/encrypter being stateful is not ideal but acceptable.

\ifnum\eprint=1
\subsection{Related works }
Note that given $g$ and $g^a$ over the ring $(\Z/n\Z)^*$, computing $g^{-a}$ is feasible for any $n$. On the other hand, computing $g^{1/a}$ is infeasible for suitable subgroup $\G$ of $(\Z/n\Z)^*$. However, in general, it is not clear how to efficiently perform the multiplicative operation ``in the exponent".

The only existing candidate of (T)GII that supports a large number of group operations is implied by the self-bilinear maps constructed by Yamakawa et al. \cite{DBLP:conf/crypto/Yamakawa0HK14} using general purpose indistinguishability obfuscation \cite{DBLP:conf/crypto/BarakGIRSVY01}. The existence of iO is currently considered as a strong assumption in the cryptography community. Over the past five years many candidates (since \cite{DBLP:conf/focs/GargGH0SW13}) and attacks (since \cite{DBLP:conf/eurocrypt/CheonHLRS15}) were proposed for iO. Basing iO on a clearly stated hard mathematical problem is still an open research area. 

Nevertheless, the self-bilinear maps construction from iO is conceptually simple. Here we sketch the idea. Given an integer $N$ with unknown factorization, a group element $a\in (\Z/\frac{\phi(N)}{4}\Z)^*$ is represented by $g^a\in \QR^+(N)$ ($\QR^+$ denotes the signed group of quadratic residues), together with an obfuscation of the circuit $C_{2a, N}$:
\[ C_{2a, N}: \QR^+(N) \to \QR^+(N),  ~~ x\mapsto x^{2a}.   \]
Given $g^a$, $\text{Obf}( C_{2a, N} )$,  $g^b$, $\text{Obf}( C_{2b, N} )$, everyone is able to compute $g^{2ab}$. \cite{DBLP:conf/crypto/Yamakawa0HK14} proves that under the hardness of factoring and assuming that the obfuscator satisfy the security of indistinguishable obfuscation, it is infeasible for the adversary to compute $g^{ab}$. Such a result implies that under the same assumption, it is infeasible to compute $g^{1/x}$ given $g^{x}$ and $\text{Obf}( C_{2x, N} )$. 

A downside of the self-bilinear maps of \cite{DBLP:conf/crypto/Yamakawa0HK14} is that the obfuscated circuits,  referred to as ``auxiliary inputs", keep growing after the compositions. Self-bilinear maps without auxiliary input is recently investigated by \cite{DBLP:journals/iacr/YamakawaYHK18} in the context of rings with infeasible inversion, but constructing such a primitive is open even assuming iO.

\else
\paragraph{Organization.}
The rest of the paper is organized as follows. 
Section~\ref{sec:prelim} provides the background of imaginary quadratic fields, elliptic curves and isogenies. 
Section~\ref{sec:neighbor_vol} defines the computational problems for isogeny graphs over composite moduli.
Section~\ref{sec:construction_GII} provides our basic construction of a trapdoor group with infeasible inversion. 
Section~\ref{sec:cryptanalysis} provides a highlight of our cryptanalysis attempts.

Due to the page limitation, we refer the readers to the full version in the supplementary material for the general construction of TGII, the details of our cryptanalysis attempts, and the constructions of the directed transitive signature and broadcast encryption schemes.

\fi

\section{Preliminaries}\label{sec:prelim}


\paragraph{Notation and terminology. }

Let $\C, \R, \Q, \Z, \N$ denote the set of complex numbers, reals, rationals, integers, and positive integers respectively. For any field $K$ we fix an algebraic closure and denote it by $\bar{K}$. 
For $n\in\N$, let $[n] := \set{1, ..., n}$. 
For $B\in \R$, an integer $n$ is called $B$-smooth if all the prime factors of $n$ are less than or equal to $B$.
An $n$-dimensional vector is written as a bold lower-case letter, e.g. $\ary{v}: = (v_1, ..., v_n)$. For an index $k\in \N$, distinct prime numbers $p_i$ for $i\in[k]$, and $c_i\in\Z/p_i\Z$ we will let $\CRT(p_1, ..., p_k; c_1, ..., c_k)$ to denote the unique $y\in \Z/(\prod_i^k p_i)\Z$ such that $y\equiv c_i \pmod{p_i}$, for $i\in[k]$.

In cryptography, the security parameter (denoted by $\secp$) is a variable that is used to parameterize the computational complexity of the cryptographic algorithm or protocol, and the adversary's probability of breaking security. In theory and by default, an algorithm is called ``efficient" if it runs in probabilistic polynomial time over $\secp$. Exceptions may occur in reality and we will explicitly discuss them when they come up in our applications.

An $n$-dimensional lattice $\Lattice$ is a discrete additive subgroup of $\R^n$ that generate it as a vector space over $\R$. Given $n$ linearly independent vectors $\mat{B} =\set{\ary{b}_1, ..., \ary{b}_n\in \R^n}$, the lattice generated by $\mat{B}$ is 
\[ \Lattice(\mat{B}) = \Lattice(\ary{b}_1, ..., \ary{b}_n) = \set{  \sum\limits_{i=1}^{n} x_i\cdot \ary{b}_i , x_i\in \Z }. \] 
We will denote the Gram-Schmidt orthogonalization of $\mat{B}$ by $\mat{\tilde{B}}$.


Let $\G$ denote a finite abelian group. We will denote the prime factorization of its order $|\G|$ by $|\G| = \prod_{i\in [k]}p_i^{w(p_i)}$. For each $p_i$, we set $H(p_i) := |\G|/p_i^{w(p_i)}$, and $G(p_i):= \set{g^{H(p_i)}, g\in\G}$. Note that with this notation we have the isomorphism
\[ \G\to \G(p_1)\times...\times \G(p_k), ~~g\mapsto (g^{H(p_1)}, ..., g^{H(p_k)}).   \] 

For a cyclic group $\G$, the \emph{discrete-log problem} asks to find the exponent $a\in [|\G|]$ given a generator $g$ and a group element $x = g^a\in\G$. The Pohlig-Hellman algorithm \cite{DBLP:journals/tit/PohligH78} solves the discrete-log problem in time $O\left( \sum_{i} w(p_i)(\log |\G| + \sqrt{p_i}) \right)$ if the factorization of $|\G|$ is known. 

Over a possibly non-cyclic group $\G$, the discrete-log problem is defined as follows: given a set of elements $g_1, ..., g_k$ and a group element $x\in\G$, output a vector $\ary{e}\in\Z^k$ such that $x = \prod_{i = 1}^{k} g_i^{e_i}$, or decide that $x$ is not in the subgroup generated by $\{g_1, ..., g_k\}$. A generalization of the Pohlig-Hellman algorithm works for non-cyclic groups with essentially the same cost plus an $O(\log |\G|)$ factor (the algorithm is folklore \cite{DBLP:journals/tit/PohligH78} and is explicitly given in \cite{teske1999pohlig}). 
\ifnum\eprint=1
A further improvement removing the $\log|\G|$ factor is given by Sutherland \cite{sutherland2011structure}.
\else
\fi


\subsection{Ideal class groups of imaginary quadratic orders}\label{sec:ICG}

There are two equivalent ways of describing ideal class groups of imaginary quadratic orders: via the theory of ideals or quadratic forms. We will be using these two view points interchangeably. The main references for these are \cite{mccurley1988cryptographic,cohen1995course,cox2011primes}.

Let $K$ be an imaginary quadratic field. An \emph{order} $\OOO$ in $K$ is a subset of $K$ such that
\begin{enumerate}
\item $\OOO$ is a subring of $K$ containing 1,
\item $\OOO$ is a finitely generated $\Z$-module,
\item $\OOO$ contains a $\Q$-basis of $K$.
\end{enumerate}
The ring $\OOO_K$ of integers of $K$ is always an order. For any order $\OOO$, we have $\OOO\subseteq \OOO_K$, in other words $\OOO_K$ is the maximal order of $K$ with respect to inclusion.

The ideal class group (or class group) of $\OOO$ is the quotient group $\CL(\OOO) = I(\OOO)/P(\OOO)$ where $I(\OOO)$ denotes the group of proper (i.e. invertible) fractional $\OOO$-ideals of, and $P(\OOO)$ is its subgroup of principal $\OOO$-ideals. 
Let $D$ be the discriminant of $\OOO$. Note that since $\OOO$ is quadratic imaginary we have $D<0$.
Sometimes we will denote the class group $\CL(\OOO)$ as $\CL(D)$, and the class number (the group order of $\CL(\OOO)$) as $h(\OOO)$ or $h(D)$. 

Let $D = D_0\cdot f^2$, where $D_0$ is the \emph{fundamental discriminant} and $f$ is the \emph{conductor} of $\OOO$ (or $D$). The following well-known formula relates the class number of an non-maximal order to that of the maximal one:
\begin{equation}\label{eqn:classnumbernonmaximal}
\frac{h(D)}{w(D)} = \frac{h(D_0)}{w(D_0)} \cdot f \prod_{p\mid f}\left( 1 - \frac{\kron{D_0}{p} }{p} \right),
\end{equation}
where $w(D) = 6$ if $D = -3$, $w(D) = 4$ if $D = -4$, and $w(D) = 2$ if $D < -4$. Let us also remark that the Brauer-Siegel theorem implies that $\ln(h(D))\sim \ln(\sqrt{|D|})$ as $D\to -\infty$.




\paragraph{Representations.}
An $\OOO$-ideal of discriminant $D$ can be represented by its generators, or by its binary quadratic forms. A binary quadratic form of discriminant $D$ is a polynomial $ax^2 + bxy + cy^2$ with $b^2 - 4ac = D$. We denote a binary quadratic form by $(a,b,c)$. The group $SL_2(\Z)$ acts on the set of binary quadratic forms and preserves the discriminant. We shall always be assuming that our forms are positive definite, i.e. $a>0$. Recall that a form $(a, b, c)$ is called \emph{primitive} if $\gcd(a,b,c)=1$, and a primitive form is called \emph{reduced} if $-a<b\leq a < c$ or $0\leq b \leq a = c$. Reduced forms satisfy $a\leq \sqrt{|D|/3}$. 

A fundamental fact, which goes back to Gauss, is that in each equivalence class, there is a unique reduced form (see Corollary 5.2.6 of \cite{cohen1995course}). Given a form $(a, b, c)$, denote $[ (a, b, c) ]$ as its equivalence class. Note that when $D$ is fixed, we can denote a class simply by $[(a,b,\cdot)]$. Efficient algorithms of composing forms and computing the reduced form can be found in \cite[Page~9]{mccurley1988cryptographic}.

\subsection{Elliptic curves and their isogenies}\label{sec:prelim:isogeny}

In this section we will recall some background on elliptic curves and isogenies. All of this material is well-known and the main references for this section are \cite{kohel1996endomorphism,silverman2009arithmetic,silverman2013advanced,sutherland2013isogeny,de2017mathematics}. 

Let $E$ be an elliptic curve defined over a finite field $\K$ of characteristic $\neq 2, 3$ with $q$ elements, given by its Weierstrass form $y^2 = x^3 + ax +b $ where $a, b\in \K$. 
By the Hasse bound we know that the order of the $\K$-rational points $E(\K)$ satisfies 
\[  - 2\sqrt{q} \leq \#E(\K) - (q+1) \leq 2\sqrt{q}.\]
Here, $t=q+1-\#E(\K)$ is the trace of Frobenius endomorphism $\pi: (x, y)\mapsto (x^q, y^q)$. Let us also recall that Schoof's algorithm \cite{schoof1985elliptic} takes as inputs $E$ and $q$, and computes $t$, and hence $\#E(\K)$, in time $\poly(\log q)$.

The \emph{$j$-invariant} of $E$ is defined as
\[ j(E) = 1728\cdot\frac{4a^3}{4a^3+27b^2}.\] 
The values $ j = 0$ or $1728$ are special and we will choose to avoid these two values throughout the paper.
Two elliptic curves are isomorphic over the algebraic closure $\bar{\K}$ if and only if their $j$-invariants are the same. Note that this isomorphism may not be defined over the base field $\K$, in which case the curves are called twists of each other. It will be convenient for us to use $j$-invariants to represent isomorphism classes of elliptic curves (including their twists). In many cases, with abuse of notation, a $j$-invariant will be treated as the same to an elliptic curve over $\K$ in the corresponding isomorphism class.

\paragraph{Isogenies.}
An \emph{isogeny} $\varphi: E_1 \to E_2$ is a morphism of elliptic curves that preserves the identity. 
Every nonzero isogeny induces a surjective group homomorphism from $E_1(\bar{\K})$ to $E_2(\bar{\K})$ with a finite kernel. Elliptic curves related by a nonzero isogeny are said to be isogenous. By the Tate isogeny theorem \cite[pg.139]{tate1966endomorphisms} two elliptic curves $E_1$ and $E_2$ are isogenous over $\K$ if and only if $\#E_1(\K) = \#E_2(\K)$. 

The degree of an isogeny is its degree as a rational map. An isogeny of degree $\ell$ is called an $\ell$-isogeny. When $\charac (\K)\nmid\ell$, the kernel of an $\ell$-isogeny has cardinality $\ell$. Two isogenies $\phi$ and $\varphi$ are considered equivalent if $\phi = \iota_1 \circ \varphi \circ \iota_2$ for isomorphisms $\iota_1$ and $\iota_2$. 
Every $\ell$-isogeny $\varphi: E_1 \to E_2$ has a unique dual isogeny $\hat{\varphi}: E_2 \to E_1$ of the same degree such that $\varphi \circ \hat{\varphi} = \hat{\varphi} \circ \varphi = [\ell]$, 
where $[\ell]$ is the multiplication by $\ell$ map. 
The kernel of the multiplication-by-$\ell$ map is the \emph{$\ell$-torsion subgroup} 
\[ E[\ell] = \set{ P \in E(\bar{\K}) : \ell P = 0 }.\] 
When $\ell\nmid\charac(\K)$ we have $ E[\ell] \simeq \Z/\ell\Z \times \Z/\ell\Z$. 
For a prime $\ell \neq \charac(\K)$, there are $\ell + 1$ cyclic subgroups in $E[\ell]$ of order $\ell$, each corresponding to the kernel of an $\ell$-isogeny $\varphi$ from $E$. An isogeny from $E$ is defined over $\K$ if and only if its kernel subgroup $G$ is defined over $\K$ (namely, for $P\in G$ and $\sigma\in\Gal(\bar{\K}/\K)$, $\sigma(P)\in G$; note that this does not imply $G\subseteq E(\K)$). If $\ell \nmid \charac(\K)$ and $j(E)\neq 0$ or $1728$, then up to isomorphism the number of $\ell$-isogenies from $E$ defined over $\K$ is $0, 1, 2$, or $\ell + 1$. 



\paragraph{Modular polynomials.}

Let $\ell\in\Z$, let $\uhp$ denote the upper half plane $\uhp:= \set{ \tau\in\C: \mathrm{im}~\tau>0 }$ and $\uhp^* = \uhp \cup \Q\cup \set{\infty}$. Let $j(\tau)$ be the classical modular function defined on $\uhp$. For any $\tau \in \uhp$, the complex numbers $j(\tau)$ and $j(\ell \tau)$ are the $j$-invariants of elliptic curves defined over $\C$ that are related by an isogeny whose kernel is a cyclic group of order $\ell$. The minimal polynomial $\Phi_\ell(y)$ of the function $j(\ell z)$ over the field $\C(j(z))$ has coefficients that are polynomials in $j(z)$ with inter coefficients. Replacing $j(z)$ with a variable $x$ gives the \emph{modular polynomial} $\Phi_\ell(x,y) \in \Z[x,y]$, which is symmetric in $x$ and $y$. It parameterizes pairs of elliptic curves over $\C$ related by a cyclic $\ell$-isogeny (an isogeny is said to be cyclic if its kernel is a cyclic group; when $\ell$ is a prime every $\ell$-isogeny is cyclic). 
The modular equation $\Phi_\ell(x,y) = 0$ is a canonical equation for the modular curve $Y_0(\ell) = \uhp/\Gamma_0(\ell)$, where $\Gamma_0(\ell)$ is the congruence subgroup of $\SL_2(\Z)$ defined by 
\[ \Gamma_0(\ell)=\set{ \pmat{ a & b \\ c & d  }\in\SL_2(\Z) \bigg\vert \pmat{ a & b \\ c & d  } \equiv \pmat{ * & * \\ 0 & *  }\pmod \ell  }. \]

The time and space required for computing the modular polynomial $\Phi_\ell$ are polynomial in $\ell$, cf. \cite[\S~3]{elkies1998elliptic} or \cite[Page~386]{cohen1995course}. In this article we will only use $\set{ \Phi_\ell \in\Z[x, y] }_{\ell\in\poly(\secp)}$, so we might as well assume that the modular polynomials are computed ahead of time\footnote{The modular polynomials $\Phi_\ell$ for $1\leq\ell\leq 300$ are available at \url{https://math.mit.edu/~drew/ClassicalModPolys.html}. }. 
In reality the coefficients of $\Phi_\ell$ over $\Z[x, y]$ grow significantly with $\ell$, so computing $\Phi_\ell$ over $\K[x,y]$ directly is preferable using the improved algorithms of \cite{charles2005computing,broker2012modular}, or even $\Phi_\ell(j_1, y)$ over $\K[y]$ using \cite{sutherland2013evaluation}.


\subsection{Isogeny volcanoes and the class groups}\label{sec:isogenyandclassgroup}

An isogeny from an elliptic curve $E$ to itself is called an \emph{endomorphism}. 
Over a finite field $\K$, $\Endo(E)$ is isomorphic to an imaginary quadratic order when $E$ is ordinary, or an order in a definite quaternion algebra when $E$ is supersingular. In this paper we will be focusing on the ordinary case.

\paragraph{Isogeny graphs.}
These are graphs capturing the relation of being $\ell$-isogenous among elliptic curves over a finite field $\K$.

\begin{definition}[$\ell$-isogeny graph]\label{def:isoggraph}
Fix a prime $\ell$ and a finite field $\K$ such that $ \charac(\K)\neq \ell$. The $\ell$-isogeny graph $G_\ell(\K)$ has vertex set $\K$. Two vertices $(j_1, j_2)$ have a directed edge (from $j_1$ to $j_2$) with multiplicity equal to the multiplicity of $j_2$ as a root of $\Phi_\ell(j_1, Y)$. The vertices of $G_\ell(\K)$ are $j$-invariants and each edge corresponds to an (isomorphism class of an) $\ell$-isogeny. 
\end{definition}
For $j_1,j_2\notin\{0,1728\}$, an edge $(j_1, j_2)$ occurs with the same multiplicity as $(j_2, j_1)$ and thus the subgraph of $G_\ell(\K)$ on $\K\setminus\set{0, 1728}$ can be viewed as an undirected graph. 
Every curve in the isogeny class of a supersingular curve is supersingular. Accordingly, $G_{\ell}(\K)$ has supersingular and ordinary components. The ordinary components of $G_\ell(\K)$ look like $\ell$-\emph{volcanoes}:
\begin{definition}[$\ell$-volcano]
Fix a prime $\ell$. An $\ell$-volcano $V$ is a connected undirected graph whose vertices are partitioned into one or more levels $V_0$, ..., $V_d$ such that the following hold:
\begin{enumerate}
\item The subgraph on $V_0$ (the surface, or the crater) is a regular graph of degree at most $2$.
\item For $i > 0$, each vertex in $V_i$ has exactly one neighbor in level $V_{i-1}$.
\item For $i < d$, each vertex in $V_i$ has degree $\ell + 1$.
\end{enumerate}
\end{definition}

Let $\phi: E_1 \to E_2$ by an $\ell$-isogeny of elliptic curves with endomorphism rings $\OOO_1 = \Endo(E_1)$ and $\OOO_2 = \Endo(E_2)$ respectively. Then, there are three possibilities for $\OOO_1$ and $\OOO_2$:
\begin{itemize} 
\item If $\OOO_1 = \OOO_2$, then $\phi$ is called horizontal,
\item If $[\OOO_1 : \OOO_2] = \ell$, then $\phi$ is called descending,
\item If $[\OOO_2 : \OOO_1] = \ell$, then $\phi$ is called ascending.
\end{itemize}

Let $E$ be an elliptic curve over $\K$ whose endomorphism ring is isomorphic to an imaginary quadratic order $\OOO$. Then, the set
\[ \Ell_\OOO(\K) = \set{ j(E)\in \K\mid \text{ with } \Endo(E) \simeq \OOO }\]
is naturally a $\CL(\OOO)$-torsor as follows:
For an invertible $\OOO$-ideal $\mfk{a}$ the $\mfk{a}$-torsion subgroup
\[ E[\mfk{a}] = \set{ P \in E(\bar{\K}): \alpha(P) = 0, \forall \alpha \in \mfk{a}} \]
is the kernel of a separable isogeny $\phi_\mfk{a}: E \to E'$. If the norm $N(\mfk{a})=[\OOO:a]$ is not divisible by $\charac(\K)$, then the degree of $\phi_{\mfk{a}}$ is $N(\mfk{a})$. 
Moreover, if $\mfk{a}$ and $\mfk{b}$ are two invertible $\OOO$-ideals, then $\phi_{\mfk{a}\mfk{b}} = \phi_{\mfk{a}}\phi_{\mfk{b}}$, and if $\mfk{a}$ is principal then $\phi_{\mfk{a}}$ is an isomorphism. This gives a faithful and transitive action of $\CL(\OOO)$ on $\Ell_{\OOO}(\K)$.

Every horizontal $\ell$-isogeny arises this way from the action of an invertible $\OOO$-ideal $\mfk{l}$ of norm $\ell$.
Let $K$ denote the fraction field of $\OOO$ and $\OOO_K$ be its ring of integers. If $\ell\mid [\OOO_K:\OOO]$ then no such ideal exists. Otherwise, $\OOO$ is said to be maximal at $\ell$ and there are $1 + \kron{ D }{\ell}$ horizontal $\ell$-isogenies. 

\begin{remark}[Linking ideals and horizontal isogenies]\label{remark:idealandisogeny}
When $\ell$ splits in $\OOO$ we have $(\ell) = \mfk{l}\cdot \bar{\mfk{l}}$. 
Fix an elliptic curve $E(\K)$ with $\Endo(E) \simeq \OOO$, the two horizontal isogenies $\phi_1: E\to E_1$ and $\phi_2: E\to E_2$ can be efficiently associated with the two ideals $\mfk{l}$ and $\bar{\mfk{l}}$ when $\ell\in\poly(\secp)$ (cf. \cite{schoof1995counting}). 
To do so, factorize the characteristic polynomial of Frobenius $\pi$ as $ (x-\mu)(x-\nu) \pmod \ell$, where $\mu, \nu\in \Z/\ell\Z$. Given an $\ell$-isogeny $\phi$ from $E$ to $ E/G$, the eigenvalue (say $\mu$) corresponding to the eigenspace $G$ can be verified by picking a point $P\in G$, then check whether $\pi(P) = [\mu]P$ modulo $G$. If so then $\mu$ corresponds to $\phi$. 
\end{remark}

The following fundamental result of Kohel summarizes the above discussion and more.
\begin{lemma}[\cite{kohel1996endomorphism}]\label{lemma:isograph}
Let $\ell$ be a prime. Let $V$ be an ordinary component of $G_\ell(\F_q)$ that does not contain $0$ or $1728$. Then $V$ is an $\ell$-volcano for which the following hold:
\begin{enumerate} 
\item The vertices in level $V_i$ all have the same endomorphism ring $\OOO_i$.
\item The subgraph on $V_0$ has degree $1 + \kron{D_0}{\ell}$, where $D_0 = \rm{disc}(\OOO_0)$.
\item If $\kron{D_0}{\ell} \geq 0$, then $|V_0|$ is the order of $[\mfk{l}]$ in $\CL(\OOO_0)$; otherwise $|V_0| = 1$.
\item The depth of $V$ is $d$, where $2d$ is the largest power of $\ell$ dividing $ (t^2-4q)/D_0 $, and $t^2 = \rm{tr}(\pi_E)^2$ for $j(E)\in V$. 
\item $\ell\nmid [\OOO_K:\OOO_0]$ and $[\OOO_i:\OOO_{i+1}] = \ell$ for $0\leq i<d$. 
\end{enumerate}
\end{lemma}


Let $\mathcal{G}_{\OOO, m}(\K)$ be the regular graph whose vertices are the elements of $\Ell_\OOO(\K)$, and whose edges are the equivalence classes of horizontal isogenies defined over $\K$ of prime degrees $\leq m$. The following result states that assuming GRH $\mathcal{G}_{\OOO, m}(\K)$ is an expander graph.

\begin{lemma}[\cite{DBLP:conf/asiacrypt/JaoMV05,DBLP:journals/iacr/BonehGKLSSTZ18}]\label{lemma:expander}
Let $q=\#\K$, $\OOO$ be an imaginary quadratic order of discriminant $D$, and $\epsilon$ be a fixed constant. 
Let $m$ be such that $m\geq (\log q)^{2+\epsilon}$. 
Assuming GRH, a random walk on $\mathcal{G}_{\OOO, m}(\K)$ will reach a subset of size $S$ with probability at least $\frac{S}{2|\mathcal{G}_{\OOO, m}(\K)|}$ after $\text{polylog}(q)$ many steps.
Furthermore, for a suitable constant $C$ and any $\delta>0$, assuming GRH, the distribution of a vertice obtained from a random walk on $\mathcal{G}_{\OOO, m}(\K)$ of length $C\cdot \frac{\delta + \log h(D)}{\epsilon \cdot \log\log |D|}$ is $e^{-\delta}$-statistically close to uniform.
\end{lemma}

\paragraph{More about the endomorphism ring from a computational perspective. }
Given an ordinary curve $E$ over $\K$, its endomorphism ring $\OOO$ can be determined by first computing the trace $t$ of Frobenius endomorphism $\pi$, then computing $t^2 - 4q = v^2 D_0$, where $v^2 D_0$ is the discriminant of $\Z[\pi]$, $\Z[\pi]\subseteq \OOO\subseteq \OOO_K$, and $K = \Q(\sqrt{D_0})$. The discriminant of $\OOO$ is then $u^2 D_0$ for some $u\mid v$. When $v$ has only few small factors, determining the endomorphism ring can be done in time polynomial in $\log(q)$ \cite{kohel1996endomorphism}.
In general it can take up to subexponential time in $\log(q)$ under GRH \cite{bissonS2011computing,bisson2011computing}.

Let $\OOO$ be an imaginary quadratic order of discriminant $D$. Let $H_D(x)$ be the Hilbert class polynomial defined by
\[ H_D(x)  = \prod_{j(E)\in\Ell_\OOO( \C )}(x - j(E)) . \]
$H_D$ has integer coefficients and is of degree $h(D)$. Furthermore, it takes $O(|D|^{1+\epsilon} )$ bits of storage. 
Under GRH, computing $H_D$ mod $q$ takes $O(|D|^{1+\epsilon})$ time and $O(|D|^{1/2+\epsilon} \log q)$ space \cite{sutherland2011computing}. In reality $H_D$ is only feasible for small $|D|$ since it takes a solid amount of space to store $H_D$. Over $\Z[x]$, \cite{sutherland2011computing} is able to compute $H_D$ for $|D|\approx 10^{13}$ and $h(D) \approx 10^6$. Over $\F_q[x]$, \cite{sutherland2012accelerating} is able to compute $H_D$ for $|D| \approx 10^{16}$ with $q \approx 2^{256}$.


\section{Isogeny graphs over composite moduli}\label{sec:neighbor_vol}

Let $p, q$ be distinct primes and set $N=pq$. We will be using elliptic curves over the ring $\ZNZ$. We will not be needing a formal treatment of elliptic curves over rings as such a discussion would take us too far afield. Instead, we will be defining objects and quantities over $\ZNZ$ by taking the $\CRT$ of the corresponding ones over $\F_p$ and $\F_q$, which will suffice for our purposes. This follows the treatment given in \cite{lenstra1987factoring}.

Since the underlying rings will matter, we will denote an elliptic curve over a ring $R$ by $E(R)$. If $R$ is clear from the context we shall omit it from the notation. To begin, let us remark that the number of points $\#(E(\ZNZ))$ is equal to $\#(E(\F_p)) \cdot \#(E(\F_q))$, and the $j$-invariant of $E(\ZNZ)$ is $\CRT(p, q; j(E(\F_p)), j(E(\F_q)))$.

\subsection{Isogeny graphs over $\ZNZ$}\label{sec:neighbor_vol_def}

Let $N$ be as above.
For every prime $\ell\nmid N$ the isogeny graph $G_\ell(\Z/N\Z)$ can be defined naturally as the graph tensor product of $G_\ell(\F_p)$ and $G_\ell(\F_q)$. 


\begin{definition}[$\ell$-isogeny graph over $\Z/N\Z$]\label{def:isoggraphN}
Let $\ell$, $p$, and $q$ be distinct primes and let $N=pq$. The $\ell$-isogeny graph $G_\ell(\Z/N\Z)$ has 

\begin{itemize}
\item The vertex set of $G_\ell(\ZNZ)$ is $\ZNZ$, identified with  $\Z/p\Z\times\Z/q\Z$ by $\CRT$,
\item  Two vertices $v_1=(v_{1,p},v_{1,q})$ and $v_2=(v_{2,p},v_{2,q})$ are connected  if and only if $v_{1,p}$ is connected to $v_{2,p}$ in $G_{\ell}(\F_p)$ and $v_{1,q}$ is connected to $v_{2,q}$ in $G_\ell(\F_q)$.
\end{itemize}
\end{definition}

Let us make a remark for future consideration. In the construction of groups with infeasible inversion, we will be working with special subgraphs of $G_\ell(\Z/ N\Z)$, where the vertices over $\F_p$ and $\F_q$ correspond to $j$-invariants of curves whose endomorphism rings are the same imaginary quadratic order $\OOO$. Nevertheless, this is a choice we made for convenience, and it does not hurt to define the computational problems over the largest possible graph and to study them first.

\subsection{The $\ell$-isogenous neighbors problem over $\ZNZ$}\label{sec:neighbor_vol_l}

\begin{definition}[The $\ell$-isogenous neighbors problem]\label{def:problem:basic}
Let $p,q$ be two distinct primes and let $N = p q$. Let $\ell$ be a polynomially large prime s.t. $\gcd(\ell, N)=1$. 
The input of the $\ell$-isogenous neighbor problem is $N$ and an integer $j\in \Z/N\Z$ such that there exists (possibly more than) one integer $j'$ that $\Phi_\ell( j, j' ) = 0$ over $\Z/N\Z$. 
The problem asks to find such integer(s) $j'$.
\end{definition}

The following theorem shows that the problem of finding \emph{all} of the $\ell$-isogenous neighbors is at least as hard as factoring $N$.

\begin{theorem}\label{thm:basicequalfactoring}
If there is a probabilistic polynomial time algorithm that finds all the $\ell$-isogenous neighbors in Problem~\ref{def:problem:basic}, then there is a probabilistic polynomial time algorithm that solves the integer factorization problem. 
\end{theorem}
The idea behind the reduction is as follows. Suppose it is efficient to pick an integer $j$ over $\ZNZ$, let $j_p = j \pmod p$ and $j_q = j \pmod q$, such that $j_p$ has at least two distinct neighbors in $G_{\ell}(\F_p)$, and $j_q$ has at least two distinct neighbors in $G_{\ell}(\F_q)$. In this case if we are able to find \emph{all} the integer solutions $j'\in \Z/N\Z$ such that $\Phi_\ell( j, j' ) = 0$ over $\Z/N\Z$, then there exist two distinct integers $j'_1$ and $j'_2$ among the solutions such that $N>\gcd(j'_1-j'_2, N)>1$. One can also show that finding \emph{one} of the integer solutions is hard using a probabilistic argument, assuming the underlying algorithm outputs a random solution when there are multiple ones.


In the reduction we pick the elliptic curve $E$ randomly, so we have to make sure that for a non-negligible fraction of the elliptic curves $E$ over $\F_p$, $j(E)\in G_{\ell}(\F_p)$ has at least two neighbors. The estimate for this relies on the following lemma:
\begin{lemma}[\cite{lenstra1987factoring} (1.9)]\label{lemma:num_of_points_dist}
There exists an efficiently computable positive constant $c $ such that for each prime number $p>3$, for a set of integers $S\subseteq \set{s\in \Z\,\mid\, | p+1-s |<\sqrt{p} }$, we have 
\[ \# ' \set{ E \mid E \text{ is an elliptic curve over } \F_p,\,\, \# E(\F_p)\in S }_{/ \simeq_{\F_p}}
    \geq c\, (\# S-2) \frac{ \sqrt{p}}{\log p} . \]
where $\# '\set{ E }_{/ \simeq_{\F_p}}$ denotes the number of isomorphism classes of elliptic curves over $\F_p$, each counted with weight $(\# \Aut E)^{-1}$.
\end{lemma}

\begin{theorem}\label{thm:ratioD}
Let $p,\ell$ be primes such that $6\ell<\sqrt{p}$. The probability that for a random elliptic curve $E$ over $\F_p$ (i.e. a random pair $(a,b)\in\F_p\times \F_p$ such that $4a^3+27b^2\neq 0$) $j(E)\in G_{\ell}(\F_p)$ having at least two neighbors is $\Omega(\frac{1}{\log p})$.
\end{theorem}

\ifnum\eprint=0
Due to the page limitation we refer the readers to the full version for the proof of Theorem~\ref{thm:ratioD}.

\else

\begin{proof}[Proof of Theorem~\ref{thm:ratioD}]

We first give a lower bound on the number of ordinary elliptic curves over $\F_p$ whose endomorphism ring has discriminant $D$ such that $\kron{D}{\ell} = 1$. If for some pair of $(\ell, p)$ there are not enough elliptic curves over $\F_p$ with two horizontal $\ell$-isogenies then we count the elliptic curves with vertical $\ell$-isogenies.

We start with estimating the portion of $t\in [\ell]$ that satisfies $\kron{t^2 - 4p}{\ell} = 1$: 
\begin{equation}  
\Pr_{t\in[\ell]}\left[ \kron{t^2 - 4p}{\ell} = 1 \right] = 
\begin{cases}
	0 							& \ell = 2 \\
	\frac12-\frac{3}{2\ell}	& \ell > 2 \text{ and } \kron{4p}{\ell} = 1 \\
	\frac12- \frac{ 1}{2\ell}		& \ell > 2 \text{ and } \kron{4p}{\ell} = -1
\end{cases}
\end{equation} 
where the last two equations follows the identity\footnote{Derived from Theorem~19 in \url{http://www.imomath.com/index.php?options=328&lmm=0}.} $\sum_{t = 1}^{\ell} \kron{t^2 - 4p}{\ell} = -1$.
Hence for $\ell\geq 5$ or $\ell = 3$ and $\kron{4p}{3} = \kron{p}{3} = -1$, no less than $\frac12-\frac{ 3}{2\ell}$ of the $t\in[\ell]$ satisfy $\kron{t^2 - 4p}{\ell} = 1$. 

We now estimate the number of elliptic curves over $\F_p$ whose discriminant of the endomorphism ring $D$ satisfies $\kron{D}{\ell} = 1$. To do so we set $ r = \lfloor \sqrt{p} / \ell \rfloor $, and use Lemma~\ref{lemma:num_of_points_dist} by choosing the set $S$ as 
\[ S = \set{ s ~\bigg\vert \kron{(p+1-s)^2 - 4p}{\ell} = 1, s\in\set{ (p+1)-r\cdot\ell, ..., (p+1)+r\cdot\ell}\setminus\set{p+1}  }.  \]
Note that $\# S\geq r(\ell-3)$. 
Therefore, Lemma~\ref{lemma:num_of_points_dist}, there exists an effectively computable constant $c$ such that the number of isomorphism classes of elliptic curves over $\F_p$ with the number of points in the set $S$ is greater or equal to  
\begin{equation}\label{eqn:lenstra_bound}
 		 c\cdot (r(\ell-3)-2)\cdot \frac{\sqrt{p}}{\log p} 
	\geq c\cdot ((\frac{\sqrt{p}}{\ell}-1)(\ell-3)-2)\cdot \frac{\sqrt{p}}{\log p}
	> 	 c\cdot (\frac{\sqrt{p}}{2\ell} \cdot \frac{\ell}{3} - 2) \cdot \frac{\sqrt{p}}{\log p} 
	\geq c\cdot \frac{p}{ 18 \log p} . 
\end{equation}
Since the total number of elliptic curves over $\F_p$ is $p^2 - p$; the number of elliptic curves isomorphic to a given elliptic curve $E$ is $\frac{(p-1)}{\# \Aut E}$ \cite[(1.4)]{lenstra1987factoring}. 
So for $\ell\geq 5$ or $\ell = 3$ and $\kron{4p}{3} = -1$, the ratio of elliptic curves over $\F_p$ with discriminant $D$ such that $\kron{D}{\ell} = 1$ is $\Omega(\frac{1}{\log p})$.

To finish the treatment of the case, where $\ell\geq 5$, or $\ell = 3$ and $\kron{4p}{3} = \kron{p}{3} = -1$, we will show that among such curves the proportion of the $j$-invariants $j(E)$ on the crater of the volcano having one or two neighbors is $o(\frac{1}{\log p})$. Recall that we are in the case $\ell\nmid D$ and $\ell = \mathfrak{l}_1\mathfrak{l_2}$ in $\Q(\sqrt{D})$, and that the crater has size equals to the order of $\mathfrak{l}_1$ (which is the same as the order of $\mathfrak{l}_2$) in $\CL(\OOO)$. 

If the crater has $1$ or $2$ vertices, $\mathfrak{l}_1$ must have order dividing $2$ in $\CL(\OOO)$. If $\mathfrak{l}_1$ has order $1$ in $\CL(\OOO)$ then we have $x^2-Dy^2=\ell$ for some $x,y\in \Z$. Since $\ell$ is prime we necessarily have $y\neq0$. Moreover, since $6\ell<\sqrt{p}$ we have $-D<\sqrt{p}$ and therefore $4p -\sqrt{p}<t^2$. On the other hand, by the Hasse bound we have $t^2\leq 4p$, hence $4p-\sqrt{p}<t^2\leq4p$. Therefore, there are at most $O(p^{\frac14})$ $j$-invariants for which the the top of the volcano consists of a single vertex.  This handles the case of $\mathfrak{l}_1$ having order $1$.\\
The remaining case of $\mathfrak{l}_1$ having order $2$, on the other hand, cannot happen because of genus theory. More precisely, let $D_0$ be the discriminant of $\Q(\sqrt{D})$.  Since $\mathfrak{l}_1$ has order $2$ in $\CL(D)$ and $\mathfrak{l}_1\nmid D$ it has order $2$ in the class group $\CL(D_0)$ of $\Q(\sqrt{D})$. Now, recall that, by genus theory, the $2$-torsion in $\CL(D_0)$ is generated by primes dividing $D_0$. Therefore, if $\mathfrak{l}_1$ has order $2$, then $\mathfrak{l}_1\mid D_0$, which gives a contradiction.

Therefore, in the case of $\kron{D}{\ell}=1$ the probability that a random $j(E)\in \F_p$  having $\leq 2$ neighbors on the crater of the volcano is $O(p^{-\frac34})=o(\frac{1}{\log p})$, which finishes the treatment of the case $\ell\geq 5$ or $\ell = 3$ and $\kron{4p}{3} = \kron{p}{3} = -1$.

For the remaining cases, where $\ell = 2$, or $\ell = 3$ and $\kron{p}{3} = 1$, we count the number of vertical isogenies. Following the formula for the depth of an isogeny volcano in Lemma~\ref{lemma:isograph}, for an elliptic curve $E(\F_p)$ of trace $t$ with $\ell^2\mid t^2 - 4p$, the curve lives on the part of the volcano of depth $\geq$ 1. In this case we only need to make sure that the curve does not live at the bottom of the volcano because otherwise it will necessarily have at least $\ell$-neighbors (if it is at the bottom it has only one neighbor). 



When $\ell = 2$, every $t=2t_1$ satisfies $4\mid t^2 - 4p$. So every $t\in [ - 2\sqrt{p}, 2\sqrt{p}]\cap 2\Z$ corresponds to trace of an elliptic curve over $\F_p$ with at least two neighbors.

When $\ell = 3$ and $\kron{p}{3} = 1$, Hensel's lemma implies that $\frac29$ of the $t\in [-2\sqrt{p},2\sqrt{p}]\cap\Z$  satisfy $t^2\equiv 4p \bmod 9$. These all correspond to traces of an elliptic curves over $\F_p$ with at least two neighbors.

Finally, in both cases using Lemma~\ref{lemma:num_of_points_dist} with a set $S$ that takes a constant fraction from $[ p+1-\sqrt{p}, p+1+\sqrt{p}]\cap\Z$, we see that the $O(\frac{1}{\log p})$ lower bound also applies for $\ell = 2$ or $\ell = 3$ and $\kron{p}{3} = 1$. 

\end{proof}

\fi

\begin{proof}[Proof of Theorem~\ref{thm:basicequalfactoring}]
Suppose that there is a probabilistic polynomial time algorithm $A$ that finds all the $\ell$-isogenous neighbors in Problem~\ref{def:problem:basic} with non-negligible probability $\eta$. We will build a probabilistic polynomial time algorithm $A'$ that solves factoring. Given an integer $N$, $A'$ samples two random integers $a, b\in\Z/N\Z$ such that $4a^3+27b^2\neq 0$, and computes $j = 1728\cdot\frac{4a^3}{4a^3+27b^2}$. With all but negligible probability $\gcd(j,N) = 1$ and $j\neq 0, 1728$; if $j$ happens to satisfy $1<\gcd(j,N) <N$, then $A'$ outputs $\gcd(j,N)$.

$A'$ then sends $N, j_0$ to the solver $A$ for Problem~\ref{def:problem:basic} for a fixed polynomially large prime $\ell$, gets back a set of solutions $\mathcal{J} = \set{j_i}_{i\in[k]}$, where $0\leq k\leq (\ell+1)^2$ denotes the number of solutions. 
With probability $\Omega(\frac{1}{\log^2 N}) $, the curve $E: y^2 = x^3 + a x + b$ has at least two $\ell$-isogenies over both $\F_p$ and $\F_q$ due to Theorem~\ref{thm:ratioD}. In that case there exists $j, j'\in \mathcal{J}$ such that $1 < \gcd(j - j', N) < N$, which gives a prime factor of $N$.
\end{proof}

\subsection{The $(\ell, m)$-isogenous neighbors problem over $\ZNZ$}\label{sec:neighbor_vol_lm}

\begin{definition}[The $(\ell, m)$-isogenous neighbors problem]\label{def:problem:joint}
Let $p$ and $q$ be two distinct primes. Let $N := p\cdot q$.  
Let $\ell$, $m$ be two polynomially large integers s.t. $\gcd(\ell m, N)=1$. 
The input of the $(\ell, m)$-isogenous neighbor problem is the $j$-invariants $j_1$, $j_2$ of two elliptic curves $E_1$, $E_2$ defined over $\Z/N\Z$. The problem asks to find all the integers $j'$ such that $\Phi_\ell( j(E_1), j' ) = 0$, and $\Phi_m( j(E_2), j' ) = 0$ over $\Z/N\Z$.
\end{definition}

When $\gcd(\ell, m)=1$, applying the Euclidean algorithm on $\Phi_{\ell}(j_1, x)$ and $\Phi_{m}(j_2, x)$ gives a linear polynomial over $x$.
\begin{lemma}[\cite{enge2010class}]\label{lemma:gcd_modularpoly}
Let $j_1, j_2 \in \Ell_\OOO(\F_p)$, and let $\ell, m \neq p$ be distinct primes with $4\ell^2 m^2 < |D|$. 
Then the degree of $f(x): = \gcd( \Phi_{\ell}(j_1, x), \Phi_{m}(j_2, x) )$ is less than or equal to $1$.
\end{lemma}

When $\gcd(\ell, m) = d > 1$, applying the Euclidean algorithm on $\Phi_{\ell}(j_1, x)$ and $\Phi_{m}(j_2, x)$ gives a polynomial of degree at least $d$. We present a proof in the the case where $m = \ell^2$, which has the general idea. 

\begin{lemma}
Let $p\neq2,3$ and $\ell\neq p$ be primes, and let $j_0,j_1$ be such that $\Phi_{\ell}(j_0,j_1)=0 \bmod p$. Let $\Phi_{\ell}(X, j_0)$ and $\Phi_{\ell^2}(X,j_1)$ be the modular polynomials of levels $\ell$ and $\ell^2$ respectively. Then,
$$(X-j_1) \cdot \gcd(\Phi_{\ell}(X,j_0), \Phi_{\ell^2}(X,j_1)) = \Phi_{\ell}(X,j_0)$$
in $\F_p[X]$. In particular, 
$$\deg(\gcd(\Phi_{\ell}(X,j_0), \Phi_{\ell^2}(X,j_1))) = \ell$$
\end{lemma}

\begin{proof} Without loss of generality we can, and we do, assume that $\Phi_{\ell}(X,j_0)$, $\Phi_{\ell}(X,j_1)$, and $\Phi_{\ell^2}(X,j_1)$ split over $\F_p$ (otherwise we can base change to an extension $k'/\F_p$, where the full $\ell^2$-torsion is defined, this does not affect the degree of the $\gcd$). 

Assume that the degree of the $\gcd$ is $N_{\gcd}$. We have,
\begin{equation}\label{mod_deg}
\deg(\Phi_\ell(X,j_0))=\ell+1,\hspace{0.3in}\deg(\Phi_{\ell^2}(X,j_1))=\ell(\ell+1).
\end{equation}
Let $E_0,E_1$ denote the (isomorphism classes of) elliptic curves with $j$-invariants $j_0$ and $j_1$ respectively, and $\varphi_{\ell}:E_0\rightarrow E_1$ be the corresponding isogeny. We count the number $N_{\ell^2}$ of cyclic $\ell^2$-isogenies from $E_1$ two ways. First, $N_{\ell^2}$ is the number of roots of $\Phi_{\ell^2}(X,j_1)$, which, by \eqref{mod_deg} and the assumption that $\ell^2+\ell<p$, is $\ell^2+\ell$. 

Next, recall (cf. Corollary 6.11 of \cite{DrewLecture6}) that every isogeny of degree $\ell^2$ can be decomposed as a composition of two degree $\ell$ isogenies (which are necessarily cyclic). Using this $N_{\ell^2}$ is bounded above by $N_{\gcd}+\ell^2$, where the first factor counts the number of $\ell^2$-isogenies $E_1\rightarrow E$ that are compositions $E_1\xrightarrow{\hat{\varphi}_{\ell}} E_0 \rightarrow E$, and the second factor counts the isogenies that are compositions $E_1\rightarrow E'\rightarrow E$, where $E'\ncong E_1$. Note that we are not counting compositions $E_1\xrightarrow{\phi}\tilde{E}\xrightarrow{\hat{\phi}}E_1$ since these do not give rise to cyclic isogenies.

This shows that $\ell^2+\ell\leq \ell^2+N_{\ell^2}\Rightarrow N_{\gcd}\geq \ell$. On the other hand, by \eqref{mod_deg} $N_{\gcd}\leq \ell$ since $\Phi_{\ell}(X,j_0)/(X-j_0)$ has degree $\ell$ and each root except for $j_1$ gives a (possibly cyclic) $\ell^2$-isogeny by composition with $\hat{\varphi}_{\ell}$. This implies that $N_{\gcd}=\ell$ and that all the $\ell^2$-isogenies obtained this way are cyclic. In particular, we get that the $\gcd$ is $\Phi_{\ell}(X,j_0)/(X-j_1)$.
\end{proof}

\paragraph{Discussions.}
Let us remark that we do not know if solving the $(\ell, \ell^2)$-isogenous neighbors problem is as hard as factoring. To adapt the same reduction in the proof of Theorem~\ref{thm:basicequalfactoring}, we need the feasibility of sampling two integers $j_1$, $j_2$ such that $\Phi_{\ell}(j_1, j_2) = 0 \pmod N$, and $j_1$ or $j_2$ has to have another isogenous neighbor over $\F_p$ or $\F_q$. However the feasibility is unclear to us in general. 

From the cryptanalytic point of view, a significant difference of the $(\ell, \ell^2)$-isogenous neighbors problem and the $\ell$-isogenous neighbors problem is the following.  
Let $\ell$ be an odd prime. 
Recall that an isogeny $\phi: E_1 \to E_2$ of degree $\ell$ can be represented by a rational polynomial 
\[ \phi: E_1 \to E_2, ~~~ (x, y) \mapsto \left( \frac{f(x)}{h(x)^2}, \frac{g(x,y)}{h(x)^3} \right), \]
where $h(x)$ is its \emph{kernel polynomial} of degree $\frac{\ell-1}{2}$. The roots of $h(x)$ are the $x$-coordinates of the kernel subgroup $G\subset E_1[\ell]$ such that $\phi: E_1 \to E_1/G$. 


Given a single $j$-invariant $j'$ over $\ZNZ$, it is infeasible to find a rational polynomial $\phi$ of degree $\ell$ that maps from a curve $E$ with $j$-invariant $j'$ to another curve $j''$, since otherwise $j''$ is a solution to the $\ell$-isogenous neighbors problem.
However, if we are given two $j$-invariants $j_1, j_2\in\Z/N\Z$ such that $\Phi_\ell(j_1, j_2) = 0 \pmod N$, as in the $(\ell, \ell^2)$-isogenous neighbors problem; then it is feasible to compute a pair of curves $E_1$, $E_2$ such that $j(E_1) = j_1$, $j(E_2) = j_2$, together with an explicit rational polynomial of an $\ell$-isogeny from $E_1$ to $E_2$. This is because the arithmetic operations involved in computing the kernel polynomial $h(x)$ mentioned in \cite{couveignes1994schoof,schoof1995counting,elkies1998elliptic} works over $\ZNZ$ by reduction mod $N$, and does not require the factorization of $N$.

\begin{proposition}\label{claim:kernelpoly}
Given $\ell, N\in\Z$ such that $\gcd(\ell, N)=1$, and two integers $j_1, j_2\in\ZNZ$ such that $\Phi_{\ell}(j_1, j_2) = 0$ over $\ZNZ$, the elliptic curves $E_1$, $E_2$, and the kernel polynomial $h(x)$ of an isogeny $\phi$ from $E_1$, $E_2$ can be computed in time polynomial in $\ell, \log(N)$. From the kernel polynomial $h(x)$ of an isogeny $\phi$, computing $f(x)$, $g(x,y)$, hence the entire rational polynomial of $\phi$, is feasible over $\ZNZ$ via V\'{e}lu's formulae \cite{velu1971}.
\end{proposition}
However, it is unclear how to utilize the rational polynomial to solve the $(\ell, \ell^2)$-joint neighbors problem. We postpone further discussions on the hardness and cryptanalysis to Section~\ref{sec:cryptanalysis}.

%
%

\section{Trapdoor group with infeasible inversion}\label{sec:construction_GII}

In this section we present the construction of the trapdoor group with infeasible inversion. 
As the general construction is somewhat technical we will present it in two steps: first we will go over the basic algorithms that feature a simple encoding and composition rule, which suffices for the instantiations of the applications; we will then move to the general algorithms that offer potential optimization and flexibility.
\ifnum\eprint=0
Due to the page limitation, we leave the general setting to the full version in the supplementary material.
\else
\fi

\subsection{Definitions}

Let us first provide the definition of a TGII, adapted from the original definition in \cite{hohenberger2003cryptographic,molnar2003homomorphic} to match our construction. 
The main differences are:
\begin{enumerate}
\item The trapdoor in the definition of \cite{hohenberger2003cryptographic,molnar2003homomorphic} is only used to invert an encoded group element, whereas we assume the trapdoor can be use to encode and decode (which implies the ability of inverting).
\item We classify the encodings of the group elements as \emph{canonical encodings} and \emph{composable encodings}, whereas the definition from \cite{hohenberger2003cryptographic,molnar2003homomorphic} does not. 
In our definition, the canonical encoding of an element is uniquely determined once the public parameter is fixed. It can be directly used in the equivalence test, but it does not support efficient group operations. Composable encodings of group elements support efficient group operations. A composable encoding, moreover, can be converted into a canonical encoding by an efficient, public extraction algorithm.
\end{enumerate}

\begin{definition}
Let $\G = ( \circ, 1_\G )$ be a finite multiplicative group where $\circ$ denotes the group operator, and $1_\G$ denotes the identity. 
For $x\in\G$, denote its inverse by $x^{-1}$. $\G$ is associated with the following efficient algorithms:
\begin{description}
\item[Parameter generation. ] $\Gen(1^\secp)$ takes as input the security parameter $1^\secp$, outputs the public parameter $\PP$ and the trapdoor $\tau$.

\item[Private sampling. ] $\Trap\Sam(\PP, \tau, x)$ takes as inputs the public parameter $\PP$, the trapdoor $\tau$, and a plaintext group element $x\in\G$, outputs a composable encoding $\enc(x)$.

\item[Composition. ] $\comp(\PP, \enc(x), \enc(y) )$ takes as inputs the public parameter $\PP$, two composable encodings $\enc(x), \enc(y)$, outputs $\enc(x\circ y)$. We often use the notation $\enc(x)\circ\enc(y)$ for $\comp(\PP, \enc(x), \enc(y) )$.

\item[Extraction. ] $\convert(\PP, \enc(x))$ takes as inputs the public parameter $\PP$, a composable encoding $\enc(x)$ of $x$, outputs the canonical encoding of $x$ as $\canenc(x)$.



\end{description}

The hardness of inversion requires that it is infeasible for any efficient algorithm to produce the canonical encoding of $x^{-1}$ given a composable encoding of $x\in\G$.
\begin{description}
\item[Hardness of inversion. ] For any p.p.t. algorithm $A$,
\[ \Pr [  z = \enc^*(x^{-1}) \mid z\la A(\PP, \enc(x)) ] < \negl(\secp),  \]
where the probability is taken over the randomness in the generation of $\PP$, $x$, $\enc(x)$, and the adversary $A$.
\end{description}
\end{definition}


\ifnum\eprint=1
\subsection{Construction - 1: Basic setting}\label{sec:TGIIfromisogeny}
\else
\subsection{Construction details: basic}\label{sec:TGIIfromisogeny}
\fi

In this section we provide the formal construction of the TGII with the basic setting of algorithms. 
The basic setting assumes that in the application of TGII, the encoding sampling algorithm can be stateful, and it is easy to determine which encodings have to be pairwise composable, and which are not. 
Under these assumptions, we show that we can always sample composable encodings so that the composition always succeeds. 
That is, the degrees of the any two encodings are chosen to be coprime if they will be composed in the application, and not coprime if they will not be composed. The reader may be wondering why we are distinguishing pairs that are composable and those that are not, as opposed to simply assuming that every pairs of encoding are composable. The reason is for security, meanly due to the parallelogram attack in \S\ref{sec:parallelogram_attack}.

The basic setting suffices for instantiating the directed transitive signature \cite{hohenberger2003cryptographic,molnar2003homomorphic} and the broadcast encryption schemes \cite{ILOP04}, where the master signer and the master  encrypter are stateful. 
We will explain how to determine which encodings are pairwise composable in these two applications, so as to determine the prime degrees of the encodings (the rest of the parameters are not application-specific and follow the universal solution from this section).

For convenience of the reader and for further reference, we provide in Figure~\ref{fig:parameter} a summary of the parameters, with the basic constraints they should satisfy, and whether they are public or hidden. The correctness and efficiency reasons behind these constraints will be detailed in the coming paragraphs, whereas the security reasons will be explained in \S\ref{sec:cryptanalysis}.

\paragraph{Parameter generation. }
The parameter generation algorithm $\Gen(1^\secp)$ takes the security parameter $1^\secp$ as input, first chooses a non-maximal order $\OOO$ of an imaginary quadratic field as follows:
\begin{enumerate}
	\item Select a square-free negative integer $D_0\equiv 1 \bmod 4$ as the fundamental discriminant, such that $D_0$ is polynomially large and $h(D_0)$ is a prime.	
	\item Choose $k = O(\log(\secp))$, and a set of distinct polynomially large prime numbers $\set{ f_i }_{i\in[k]}$ such that the odd-part of $\left(f_i - \kron{D_0}{f_i}\right) $ is square-free and not divisible by $h(D_0)$. Let $f = \prod_{i\in[k]} f_i$. 
	\item Set $D = f^2 D_0$. Recall from Eqn.~\eqref{eqn:classnumbernonmaximal} that
		\begin{equation}\label{eqn:hD}
			h(D) = 2\cdot \frac{h(D_0)}{w(D_0)}\prod_{i\in [k]}\left( f_i - \kron{D_0}{f_i} \right)
		\end{equation}
	\end{enumerate}
Let $\CL(\OOO)_{\text{odd}}$ be the odd part of $\CL(\OOO)$, $h(D)_{\text{odd}}$ be largest odd factor of $h(D)$. Note that due to the choices of $D_0$ and $\set{f_i}$, $\CL(\OOO)_{\text{odd}}$ is cyclic, and we have $|D|, h(D)_{\text{odd}} \in \secp^{O(\log\secp)}$. The group with infeasible inversion $\G$ is then $\CL(\OOO)_{\text{odd}}$ with group order $h(D)_{\text{odd}}$.

We then sample the public parameters as follows:
\begin{enumerate}
	\item Choose two primes $p$, $q$, and elliptic curves $E_{0,\F_p}$, $E_{0,\F_q}$ with discriminant $D$, using the CM method (cf. \cite{lay1994constructing} and more). 
	\item Check whether $p$ and $q$ are safe RSA primes (if not, then back to the previous step and restart). 
	Also, check whether the number of points $\#(E_0(\F_p))$, $\#(E_0(\F_q))$, $\#(\tilde{E}_0(\F_p))$, $\#(\tilde{E}_0(\F_q))$ (where $\tilde{E}$ denotes the quadratic twist of $E$) are polynomially smooth (if yes, then back to the previous step and restart). $p$, $q$ and the number of points should be hidden for security.
	\item Set the modulus $N$ as $N := p\cdot q$ and let $j_0=\CRT(p,q;j(E_{0,\F_p}),j(E_{0,\F_q}))$. Let $j_0$ represent the identity of $\G$.  
\end{enumerate}
Output $(N, j_0)$ as the public parameter $\PP$.
Keep $(D, p, q)$ as the trapdoor $\tau$ ($D$ and the group order of $\G$ should be hidden for security).

\begin{figure}[t]
\begin{center}
\begin{tabular}{lll}\hline
Parameters 			      & Basic constraints   			                        & Public?  	\\ \hline
The modulus $N$           & $N = pq$, $p, q$ are primes, $|p|, |q|\in\poly(\secp)$  &  Yes      \\
The identity $j(E_0)$     & $\Endo(E_0(\F_p))\simeq \Endo(E_0(\F_q)) \simeq \OOO$   &  Yes      \\
$\#(E_0(\F_p))$, $\#(E_0(\F_q))$  & not polynomially smooth  	                    &  No    \\
The discriminant $D$ of $\OOO$    & $D = D_0\cdot f^2$, $D\approx \secp^{O(\log\secp)}$, $D$ is polynomially smooth & No \\
The class number $h(D)$   & follows the choice of $D$                               & No \\
A set $S$ in an encoding: & $S=\set{ C_{i} = [(p_{i}, b_{i}, \cdot)] }_{i\in[ w ]}$ generates $\CL(D)_{\text{odd}}$ & See below    \\
~~The number $w$ of ideals  & $w\in O(\log\secp)$  	                                & Yes   \\
~~The degree $p_i$ of isogenies & $p_i\in\poly(\secp)$  	                        & Yes   \\
~~The basis $\mat{B}$ of $\Lattice_S$ & $\| \tilde{\mat{B}} \|\in\poly(\secp)$  	& No  
\end{tabular}
\caption{Summary of the choices of parameters in the basic setting.  }\label{fig:parameter}
\end{center}
\end{figure}

\paragraph{The sampling algorithm and the group operation of the composable encodings.}
Next we provide the definitions and the algorithms for the composable encoding. 

\begin{definition}[Composable encoding]\label{def:represent_enc}
Given a factorization of $x$ as $\prod_{i=1}^{w} C_i^{e_i}$, where $w\in O(\log\secp)$; $C_i = [(p_i, b_i, \cdot)]\in\G$, $e_i\in\N$, for $i\in[w]$. 
A composable encoding of $x \in\G$ is represented by 
\[ \enc(x) = (L; T_{1}, ..., T_{w}) = ((p_{1}, ..., p_{w}); (j_{1,1}, ..., j_{1,e_{1}}), ..., (j_{w,1}, ..., j_{w,e_{w}}) ), \]
where all the primes in the list $L = (p_{1}, ..., p_{w})$ are distinct; for each $i\in[w]$, $T_i\in(\ZNZ)^{e_i}$ is a list of the $j$-invariants such that $j_{ i, k} = C_i^k * j_0$, for $k\in[e_i]$. 

The degree of an encoding $\enc(x)$ is defined to be $d(\enc(x)) := \prod_{i=1}^{w} p_{i}^{e_{i}} $.
\end{definition}

Notice that the factorization of $x = \prod_{i=1}^{w} C_i^{e_i}$ has to satisfy $e_i\in\poly(\secp)$, for all $i\in[w]$, so as to ensure the length of $\enc(x)$ is polynomial.
Looking ahead, we also require each $p_i$, the degree of the isogeny that represents the $C_{i}$-action, to be polynomially large so as to ensure Algorithm~\ref{alg:classgroupaction} in the encoding sampling algorithm and Algorithm~\ref{alg:gcdop} in the extraction algorithm run in polynomial time.

The composable encoding sampling algorithm requires the following subroutine:
\begin{algorithm}\label{alg:classgroupaction}
$\act(\tau, j, C)$ takes as input the trapdoor $\tau = (D, p, q)$, a $j$-invariant $j\in\ZNZ$, and an ideal class $C\in\CL(\OOO)$, proceeds as follows:
\begin{enumerate}
	\item Let $j_p = j \mod p$, $j_q = j \mod q$. 
    \item Compute $j'_p := C * j_p  \in \F_p$, $j'_q := C * j_q \in \F_q$. 
    \item Output $j': = \CRT(p, q; j'_p, j'_q)$.
\end{enumerate}
\end{algorithm}

\begin{algorithm}[Sample a composable encoding]\label{alg:represent_enc}
Given as input the public parameter $\PP = (N, j_0)$, the trapdoor $\tau = (D, p, q)$, and $x\in\G$, 
$\Trap\Sam(\PP, \tau, x)$ produces a composable encoding of $x$ is sampled as follows:
\begin{enumerate}
\item Choose $w\in O(\log\secp)$ and a generation set $ S = \set{ C_{i} = [(p_{i}, b_{i}, \cdot)] }_{i\in[ w ]}\subset \G$.
\item Sample a short basis $\mat{B}$ (in the sense that $\|\tilde{\mat{B}}\|\in\poly(\secp)$) for the relation lattice $\Lattice_S$: 
\begin{equation}\label{eqn:relationLattice}
	\Lattice_S: = \set{ \ary{y} \mid \ary{y}\in\Z^{w}, \prod_{i\in[w]} C_i^{y_i} = 1_\G  }. 
\end{equation}  
\item Given $x$, $S$, $\mat{B}$, sample a short vector $\ary{e}\in\set{\poly(\secp)\cap\N}^{w}$ such that $x = \prod_{i\in [w]} C_{i}^{e_{i}}$. 
\item For all $i\in[w]$: 
	\begin{enumerate}
	\item Let $j_{i,0}:=j_0$.
	\item For $k = 1$ to $e_{i}$: compute $j_{i,k} := \act( \tau, j_{i,k-1}, C_{i} )$.
	\item Let $T_{i} := (j_{i,1}, ..., j_{i,e_{i}})$.
	\end{enumerate}
\item Let $L\in\N^{w}$ be a list where the $i^{th}$ entry of $L$ is $p_{i}$.
\item Output the composable encoding of $x$ as
\[ \enc(x) = (L; T_{1}, ..., T_{w}) = ((p_{1}, ..., p_{w}); (j_{1,1}, ..., j_{1,e_{1}}), ..., (j_{w,1}, ..., j_{w,e_{w}}) ).\]
\end{enumerate}
\end{algorithm}

\begin{remark}[Thinking of each adjacent pair of $j$-invariants as an isogeny]\label{remark:isogenyoverZNZ}
In each $T_i$, each adjacent pair of the $j$-invariants can be thought of representing an isogeny $\phi$ that corresponds to the ideal class $C_{i} = [(p_{i}, b_{i}, \cdot)]$. Over the finite field, $C_i$ can be explicitly recovered from an adjacent pair of the $j$-invariants and $p_i$ (cf. Remark~\ref{remark:idealandisogeny}). 
Over $\ZNZ$, the rational polynomial of the isogeny $\phi$ can be recovered from the adjacent pair of the $j$-invariants and $p_i$ (cf. Proposition~\ref{claim:kernelpoly}), but it is not clear how to recover $b_i$ in the binary quadratic form representation of $C_{i}$. 
\end{remark}

\begin{remark}[The only stateful step in the sampling algorithm]
Recall that the basic setting assumes the encoding algorithm is stateful, where the state records the prime factors of the degrees used in the existing composable encodings. The state is only used in the first step to choose the $\set{p_i}$ of the ideals in the generation set $ S = \set{ C_{i} = [(p_{i}, b_{i}, \cdot)] }_{i\in[ w ]}$.
\end{remark}

\paragraph{Group operations.} 
Given two composable encodings, the group operation is done by simply concatenating the encodings if their degrees are coprime, or otherwise outputting ``failure".

\begin{algorithm}\label{alg:compose_ct}
The encoding composition algorithm $\comp(\PP, \enc(x), \enc(y))$ parses 
$\enc(x) = (L_x; T_{x,1}, ..., T_{x,w_x})$, $\enc(y) = (L_y; T_{y,1}, ..., T_{y,w_y})$, 
produces the composable encoding of $ x\circ y$ as follows:
\begin{itemize}
	\item If $\gcd(d(\enc(x)), d(\enc(y)))=1$, then output the composable encoding of $ x\circ y$ as 
	\[ \enc( x\circ y ) = ( L_x || L_y; T_{x,1}, ..., T_{x,w_x}, T_{y,1}, ..., T_{y,w_y} ).\] 
	\item If $\gcd(d(\enc(x)), d(\enc(y)))>1$, output ``failure". 
\end{itemize}
\end{algorithm}

\paragraph{The canonical encoding and the extraction algorithm.}

\begin{definition}[Canonical encoding]\label{def:represent_enc_un}
The canonical encoding of $x\in\G$ is $x * j_0 \in \ZNZ$. 
\end{definition}

The canonical encoding of $x$ can be computed by first obtaining a composable encoding of $x$, and then converting the composable encoding into the canonical encoding using the extraction algorithm. The extraction algorithm requires the following subroutine.

\begin{algorithm}[The ``gcd" operation]\label{alg:gcdop}
The algorithm $\gcdop(\PP, \ell_1, \ell_2; j_1, j_2)$ takes as input the public parameter $\PP$, 
two degrees $\ell_1, \ell_2$ and two $j$-invariants $j_1, j_2$, proceeds as follows:
	\begin{itemize}
	\item If $\gcd(\ell_1, \ell_2) = 1$, then it computes the linear function $f(x) = \gcd( \Phi_{\ell_2}(j_1, x), \Phi_{\ell_1}(j_2, x) )$ over $\ZNZ$, and outputs the only root of $f(x)$;
	\item If $\gcd(\ell_1, \ell_2) > 1$, it outputs $\bot$. 
	\end{itemize}
\end{algorithm}

\begin{algorithm}\label{alg:convert}
$\convert(\PP, \enc(x))$ converts the composable encoding $\enc(x)$ into the canonical encoding $\canenc(x)$. 
The algorithm maintains a pair of lists $(U, V)$, 
where $U$ stores a list of $j$-invariants $(j_1, ..., j_{|U|})$, 
$V$ stores a list of degrees where the $i^{th}$ entry of $V$ is the degree of isogeny between $j_{i}$ and $j_{i-1}$ (when $i=1$, $j_{i-1}$ is the $j_0$ in the public parameter). 
The lengths of $U$ and $V$ are always equal during the execution of the algorithm.  

The algorithm parses $\enc(x) = (L; T_{1}, ..., T_{w})$, proceeds as follows:
\begin{enumerate}
\item Initialization: Let $U := T_{1}$, $V := ( L_{1}, ..., L_{1} )$ of length $|T_{ 1}|$ (i.e. copy $L_{1}$ for $|T_{1}|$ times ).  
\item For $i = 2$ to $w$:
	\begin{enumerate}
	\item Set $u_\temp:= | U |$.
	\item For $k = 1$ to $|T_{i}|$:
		\begin{enumerate}
		\item Let $t_{i,k,0}$ be the $k^{th}$ $j$-invariant in $T_{i}$, i.e. $j_{i,k}$;
		\item For $h = 1$ to $u_\temp$: 
			\begin{itemize}
			\item If $k=1$, compute $t_{i,k,h}: = \gcdop(\PP, L_{i}, V_h; t_{i, k,h-1}, U_h)$;
			\item If $k>1$, compute $t_{i,k,h}: = \gcdop(\PP, L_{i}, V_h; t_{i, k,h-1}, t_{i, k-1,h})$;
			\end{itemize}
		\item Append $t_{i, k,u_\temp}$ to the list $U$, append $L_{i}$ to the list $V$. 
		\end{enumerate}
	\end{enumerate}
\item Return the last entry of $U$.
\end{enumerate}
\end{algorithm}
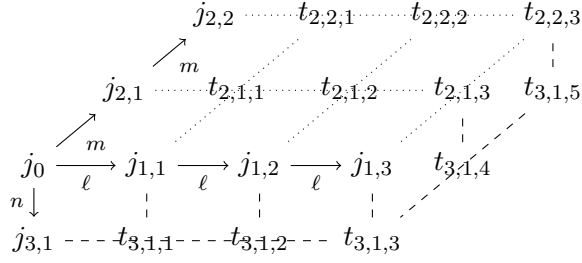
\begin{figure}
  \centering
    \begin{tikzpicture}
      \begin{scope}[xshift=13cm, yshift = 0.7cm]
        \draw node at (0   ,0) (O) {$j_0$};
        \draw node at (1.5 ,0) (l1) {$j_{1,1}$};
        \draw node at (3   ,0) (l2) {$j_{1,2}$};
        \draw node at (4.5 ,0) (l3) {$j_{1,3}$};
        \draw node at (1.2 ,1) (m1) {$j_{2,1}$};
        \draw node at (2.4 ,2) (m2) {$j_{2,2}$};
        \draw node at (0  ,-1) (n1) {$j_{3,1}$};

        \draw node at (1.2+1.5 ,1) (t211) {$t_{2,1,1}$};
        \draw node at (1.2+3.  ,1) (t212) {$t_{2,1,2}$};
        \draw node at (1.2+4.5 ,1) (t213) {$t_{2,1,3}$};
        \draw node at (2.4+1.5 ,2) (t221) {$t_{2,2,1}$};
        \draw node at (2.4+3.  ,2) (t222) {$t_{2,2,2}$};
        \draw node at (2.4+4.5 ,2) (t223) {$t_{2,2,3}$};

        \draw node at (1.5 ,-1) (t311) {$t_{3,1,1}$};
        \draw node at (3   ,-1) (t312) {$t_{3,1,2}$};
        \draw node at (4.5 ,-1) (t313) {$t_{3,1,3}$};
        \draw node at (1.2+4.5 ,0) (t314) {$t_{3,1,4}$};
        \draw node at (2.4+4.5 ,1) (t315) {$t_{3,1,5}$};

        \draw [->] (O) -- node[auto,swap]  {\scriptsize$\ell$} (l1);
        \draw [->] (l1) -- node[auto,swap] {\scriptsize$\ell$} (l2);
        \draw [->] (l2) -- node[auto,swap] {\scriptsize$\ell$} (l3);
        \draw [->] (O) -- node[auto,swap]  {\scriptsize$m$} (m1);
        \draw [->] (m1) -- node[auto,swap] {\scriptsize$m$} (m2);
        \draw [->] (O) -- node[auto,swap] {\scriptsize$n$} (n1);
        \draw (m1) edge[dotted] (t213);         
        \draw (m2) edge[dotted] (t223);        
        \draw (l1) edge[dotted] (t221);        
        \draw (l2) edge[dotted] (t222);        
        \draw (l3) edge[dotted] (t223);        
        \draw (t315) edge[dashed] (t223); 
        \draw (t314) edge[dashed] (t213); 
        \draw (t313) edge[dashed] (l3); 
        \draw (t312) edge[dashed] (l2); 
        \draw (t311) edge[dashed] (l1);    
        \draw (t315) edge[dashed] (t313);  
        \draw (n1) edge[dashed] (t313);     
      \end{scope}
    \end{tikzpicture}
    \caption{ An example for the composable encoding and the extraction algorithm.  }\label{fig:composition}
\end{figure}

\begin{example}
Let us give a simple example for the composition and the extraction algorithms.
Let $\ell, m, n$ be three distinct polynomially large primes.
Let the composable encoding of an element $y$ be $\enc(y) = ((\ell); (j_{1,1}, j_{1,2}, j_{1,3}))$, based on the factorization of $y = C_1^{e_1} = [(\ell, b_\ell, \cdot)]^3$.
Let the composable encoding of an element $z$ be $\enc(z) = ((m,n); (j_{2,1}, j_{2,2}), (j_{3,1}) )$, based on the factorization of $z = C_2^{e_2}\cdot C_3^{e_3} = [(m, b_m, \cdot)]^2 \cdot [(n, b_n, \cdot)]^1$. 
Then the composable encoding of $x = y\circ z$ obtained from Algorithm~\ref{alg:compose_ct} is  
$\enc(x) = ((\ell, m, n); (j_{1,1}, j_{1,2}, j_{1,3}), (j_{2,1}, j_{2,2}), (j_{3,1}))$. 

Next we explain how to extract the canonical encoding of $x$ from $\enc(x)$.
In Figure~\ref{fig:composition}, the $j$-invariants in $\enc(x)$ are placed on the solid arrows (their positions do not follow the relative positions on the volcano). We can think of each gcd operation in Algorithm~\ref{alg:gcdop} as fulfilling a missing vertex of a parallelogram defined by three existing vertices.

When running $\convert(\PP, \enc(x))$, the list $U$ is initialized as $(j_{1,1}, j_{1,2}, j_{1,3})$, the list $V$ is initialized as $(\ell, \ell, \ell)$. 
Let us go through the algorithm for $i = 2$ and $i = 3$ in the second step. 
\begin{itemize}
\item When $i = 2$, $u_\temp$ equals to $|U| = 3$. The intermediate $j$-invariants $\set{t_{2, k, h}}_{k\in[|T_{2}|], h\in[u_\temp]}$ are placed on the dotted lines, computed in the order of $t_{2, 1, 1}$, $t_{2, 1, 2}$, $t_{2, 1, 3}$, $t_{2, 2, 1}$, $t_{2, 2, 2}$, $t_{2, 2, 3}$. The list $U$ is updated to $(j_{1,1}, j_{1,2}, j_{1,3}, t_{2, 1, 3}, t_{2, 2, 3})$, the list $V$ is updated to $(\ell, \ell, \ell, m, m)$
\item When $i = 3$, $u_\temp$ equals to $|U| = 5$. The intermediate $j$-invariants $\set{t_{3, 1, h}}_{ h\in[u_\temp]}$ are placed on the dashed lines, computed in the order of $t_{3, 1, 1}$, ..., $t_{3, 1, 5}$. In the end, $t_{3, 1, 5}$ is appended to $U$, $n$ is appended to $V$.
\end{itemize}
The canonical encoding of $x$ is then $t_{3, 1, 5}$.
\end{example}

\paragraph{On correctness and efficiency.}
We now verify the correctness and efficiency of the parameter generation, encoding sampling, composition, and the extraction algorithms.

To begin with, we verify that the canonical encoding correctly and uniquely determines the group element in $\CL(\OOO)$. It follows from the choices of the elliptic curves $E_0(\F_p)$ and $E_0(\F_q)$ with 
$\Endo(E_0(\F_p)) \simeq \Endo( E_0(\F_q) )\simeq \OOO$, and the following bijection once we fix $E_0$: 
\[    \CL(\OOO) \to \Ell_\OOO(\K), ~~  x \mapsto x * j(E_0(\K)), \text{   for } \K\in\set{\F_p, \F_q}\]

Next, we will show that generating the parameters, i.e. the curves $E_{0,\F_p}$, $E_{0,\F_q}$ with a given fundamental discriminant $D_0$ and a conductor $f = \prod_{i}^{k}f_i$, is efficient when $|D_0|$ and all the factors of $f$ are of polynomial size. Let $u$ be an integer such that $f\mid u$. Choose a $p$ and $t_p$ such that $t_p^2 - 4p = u^2D_0$. Then, compute the Hilbert class polynomial $H_{D_0}$ over $\F_p$ and find one of its roots $j$. From $j$, descending on the volcanoes $G_{f_i}(\F_p)$ for every $f_i$ gives the $j$-invariant for the curve with desired discriminant. The same construction works verbatim for $q$.

We then show that sampling the composable encodings can be done in heuristic polynomial time: 
\begin{enumerate}
	\item Given a logarithmically large generation set $S = \set{ C_i = [(\ell_i, b_i, \cdot)]\in\CL(\OOO) }_{i\in[w]}$, a possibly big basis of the relation lattice $\Lattice_S$ can be obtained by solving the discrete-log problem over $\CL(\OOO)$, which can be done in polynomial time since the group order is polynomially smooth. 
	\item Suppose that the lattice $\Lattice_S$ satisfies the Gaussian heuristic (this is the only heuristic we assume).  That is, for all $1\leq i\leq w$, the $i^{th}$ successive minimum of $\Lattice_S$, denoted as $\lambda_i$, satisfies $\lambda_i \approx \sqrt{w}\cdot h(\OOO)^{1/w} \in\poly(\secp) $. 
	Since $w = O(\log(\secp))$, the short basis $\mat{B}$ of $\Lattice_S$, produced by the LLL algorithm, satisfies $\| \mat{B} \|\leq 2^\frac{w}{2}\cdot \lambda_w \in \poly(\secp)$. 
	\item Given a target group element $x\in\CL(\OOO)$, the polynomially short basis $\mat{B}$, we can sample a vector $\ary{e}\in\N^w$ such that $\prod_{i = 1}^m C_i^{e_i} = x$ and $\|\ary{e}\|_1\in\poly(\secp)$ in polynomial time using e.g. Babai's algorithm \cite{DBLP:journals/combinatorica/Babai86}. (In \S\ref{sec:attack:basis}, we will explain that the GPV sampler \cite{DBLP:conf/stoc/GentryPV08} is preferred for the security purpose.)
	\item The unit operation $\act(\tau, j, C)$ is efficient when the ideal class $C$ corresponds to a polynomial degree isogeny, since it is efficient to compute polynomial degree isogenies over the finite fields.
	\item The length of the final output $\enc(x)$ is $(w + \|\ary{e}\|_1)\cdot \poly(\secp)\in\poly(\secp)$.
\end{enumerate}

The algorithm $\comp(\PP, \enc(x), \enc(y) )$ simply concatenates $\enc(x)$, $\enc(y)$, so it is efficient as long as $\enc(x)$, $\enc(y)$ are of polynomial size. 

The correctness of the unit operation $\gcdop$ follows the commutativity of the endomorphism ring $\OOO$. The operation $\gcdop(\PP, \ell_1, \ell_2; j_1, j_2)$ is efficient when $\gcd(\ell_1,\ell_2)=1$, $\ell_1, \ell_2\in\poly(\secp)$, given that solving the $(\ell_1,\ell_2)$ isogenous neighbor problem over $\ZNZ$ is efficient under these conditions.

When applying $\convert()$ (Algorithm~\ref{alg:convert}) on a composable encoding $\enc(x) = (L_x; T_{x,1}, ..., T_{x,w_x})$, it runs $\gcdop$ for $\max_{i = 1}^{w_x}|T_{x_i}|\cdot\left(\sum_{i = 1}^{w_x}|T_{x_i}|\right)$ times. So obtaining the canonical encoding is efficient as long as all the primes in $L_x$ are polynomially large, and $|T_{x,i}|\in\poly(\secp)$ for all $i\in[w_x]$.

\ifnum\eprint=1
\begin{remark}[The parameters in practice]
Let us mention the two extreme sides on the deviations of the parameters between theory and practice that we expect. 
The discriminant $D = D_0\cdot f^2$ is bounded by $\secp^{O(\log\secp)}$ asymptotically, which leads to a $\secp^{O(\log\secp)}$-time attack by first guessing $D$, then solving the discrete-log problem over $\CL(D)$. 
The bottleneck of the bound of $D$ is at the dimension $w$ of the lattice $\Lattice_S$, which is set to be $O(\log\secp)$ so that the lattice reduction algorithm runs in polynomial time. In practice the lattice reduction algorithms are known to outperform the asymptotic bound. So we expect the discriminant $D$ can be set larger than a direct estimation from $\secp^{O(\log\secp)}$.

On the other hand, the modular polynomials take a solid amount of space to store in practice, so we expect the gcd operation in Algorithm~\ref{alg:gcdop} to be slow for large polynomial degrees of $\ell_1$ and $\ell_2$. 
\end{remark}
\else
\fi

\subsection{Construction - 2: General case}\label{sec:TGIIconstruction_general}

We describe the generalizations of the basic algorithms that offer potential optimization and flexibility.

\paragraph{Composing non-coprime degree encodings using ``ladders''.} 
In the basic setting of the composable encodings, the group operation is feasible only for encodings with relatively prime degrees. To support the compositions of non-relatively prime degree encodings, say the shared prime degree is $\ell$, we can include in the public parameter the encodings of $\set{x^i}_{i\in[k]}$, where $x = [(\ell, b, \cdot)]$, $k$ is a polynomial. This then supports the composition of two encodings whose sum of exponents on the degree $\ell$ is $\leq k$. 

Sampling ladders has the benefit of supporting bounded number of shared-degree compositions. In some settings of the applications (e.g. in the broadcast encryption where the number of users is bounded a priori), ladders enable stateless encoding sampling algorithms. 

\begin{algorithm}[Sampling a ladder]\label{alg:ladder}
Given a polynomially large prime $\ell$, a polynomial $k$, and an element $x = [(\ell, b, \cdot)]\in\CL(D)$, sample a ladder for degree $\ell$ of length $k$ as follows:
\begin{enumerate}
	\item For $i = 1$ to $k$: compute $j_{i} := \act( \tau, j_{i-1}, x )$. (Recall $j_0$ is the identity from the public parameter.)
	\item Let $\lad(\ell) := ( j_1, ..., j_k )$. Include $\lad(\ell)$ in the public parameter.
\end{enumerate}
\end{algorithm}

For the convenience of the description of the group operation with ladders, let $\iota(\ell,L)$ take as input an integer $\ell$, a list $L$, and output the index of $\ell$ in $L$. 

\begin{algorithm}[Composition with ladders]
The algorithm $\comp(\PP, \enc(x), \enc(y))$ parses $\enc(x) = (L_x; T_{x,1}, ..., T_{x,w_x})$, $\enc(y) = (L_y; T_{y,1}, ..., T_{y,w_y})$, produces the composable encoding of $z = x\circ y$ as follows:
\begin{enumerate} 
\item Let $L_z = L_x\cup L_y$. 
\item For all $\ell \in L_x \setminus L_{x}\cap L_{y}$, let $T_{z,\iota(\ell,L_z)} = T_{x,\iota(\ell,L_x)}$; for all $\ell \in L_y \setminus L_{x}\cap L_{y}$, let $T_{z,\iota(\ell,L_z)} = T_{y,\iota(\ell,L_y)}$. 
\item For all $\ell\in L_{x}\cap L_{y}$:
	\begin{itemize}
	\item If $|\lad(\ell)|\geq |T_{x,\iota(\ell,L_x)}|+|T_{y,\iota(\ell,L_y)}|$, then let $T_{z,\iota(\ell,L_z)}$ be the list of the first $|T_{x,\iota(\ell,L_x)}|+|T_{y,\iota(\ell,L_y)}|$ elements in $\lad(\ell)$. 
	\item If $|\lad(\ell)| < |T_{x,\iota(\ell,L_x)}|+|T_{y,\iota(\ell,L_y)}|$, then the composition is infeasible. Return ``failure". 
	\end{itemize}
\item Output the composable encoding of $z$ as $\enc(z) = ( L_z; T_{z,1}, ..., T_{z,|L_z|} )$.
\end{enumerate} 
\end{algorithm}

\paragraph{Sample the composable encoding of a random element.}  
In the applications we are often required to sample (using the trapdoor) the composable encoding of a random element in $\G$ under the generation set $S = \set{ C_i := [(\ell_i, b_i, \cdot)] }_{i\in[w]}$ and. 
Let us call such an algorithm $\Random\Sam(\tau, S)$. 

An obvious instantiation of $\Random\Sam(\tau, S)$ is to simply choose a random element $x$ from $\G$ first, then run $\Trap\Sam(\PP, \tau, x)$ in Algorithm~\ref{alg:represent_enc}.

Alternatively, we can pick a random exponent vector $\ary{e}\in [-B,B]^w$, and let $x = \prod_{i\in[w]}{C_i}^{e_i}$. For proper choices of $B$, $w$, and the set of ideals, $x$ is with high min-entropy heuristically, but figuring out the exact distribution of $x$ is difficult in general. If $B$, $w$, and the set of ideals are chosen according to Lemma~\ref{lemma:expander}, then under GRH, $x$ is statistically close to uniform over $\CL(\OOO)$.

\paragraph{Compressing the composable encoding using the partial extraction algorithm.} 
In the basic setting, the composition of composable encodings keeps growing. 
In some application, it is tempting to compress the composable encodings as much as possible. 

We provide a partial extraction algorithm, which takes two composable encodings $\enc(x), \enc(y)$ and the public parameter as inputs. It outputs a $j$-invariant $j_x$ as the canonical encoding of $x$, and a sequences of isogenies $\phi_y = (\phi_1, ..., \phi_n)$ such that $\phi_n\circ ...\circ\phi_1(j_x) = j_{x\circ y}$ (we abuse the notation by using the $j$-invariant to represent an elliptic curve with such a $j$-invariant), where $j_{x\circ y}$ is the canonical encoding of ${x\circ y}$, $n$ is the number of the $j$-invariants in $\enc(y)$. 
In other words, $\phi_y$ represents the ideal class $C_y\in\G$ defined by $\enc(y)$. 
The one-side partial extraction algorithm is useful in the situation where the encoding $\enc(x)$ is relatively long and needs to be compressed, the encoding $\enc(y)$ is relatively short and doesn't have to be compressed. 

Now we present the partial extraction algorithm. The algorithm first runs the basic extraction algorithm $\convert(\PP, \cdot)$ on $\enc(x)$ (cf. Algorithm~\ref{alg:convert}), then produces $\phi_y$ using $\enc(y)$ and the intermediate information from $\convert(\PP, \enc(x))$.
\begin{algorithm}\label{alg:partialconvert}
$\Partial.\convert(\PP, \enc(x), \enc(y))$ maintains a pair of lists $(U, V)$, 
where $U$ stores a list of $j$-invariants $(j_1, ..., j_{|U|})$, 
$V$ stores a list of degrees where the $i^{th}$ entry of $V$ is the degree of isogeny between $j_{i}$ and $j_{i-1}$ (when $i=1$, $j_{i-1}$ is the $j_0$ in the public parameter). 
The lengths of $U$ and $V$ are always equal during the execution of the algorithm.  

The algorithm parses $\enc(x) = (L_x; T_{x,1}, ..., T_{x,w_x})$, $\enc(y) = (L_y; T_{y,1}, ..., T_{y,w_y})$, proceeds as follows:
\begin{enumerate}
\item Initialization: Let $U := T_{x,1}$, $V := ( L_{x,1}, ..., L_{x,1} )$ of length $|T_{x, 1}|$.  
\item For $i = 2$ to $w_x$:
	\begin{enumerate}
	\item Set $u_\temp:= | U |$.
	\item For $k = 1$ to $|T_{x,i}|$:
		\begin{enumerate}
		\item Let $t_{x,i,k,0}$ be the $k^{th}$ $j$-invariant in $T_{x,i}$, i.e. $j_{x,i,k}$;
		\item For $h = 1$ to $u_\temp$: 
			\begin{itemize}
			\item If $k=1$, compute $t_{x,i,k,h}: = \gcdop(\PP, L_{x,i}, V_h; t_{x,i,k,h-1}, U_h)$;
			\item If $k>1$, compute $t_{x,i,k,h}: = \gcdop(\PP, L_{x,i}, V_h; t_{x,i,k,h-1}, t_{x,i, k-1,h})$;
			\end{itemize}
		\item Append $t_{x,i, k,u_\temp}$ to the list $U$, append $L_{x,i}$ to the list $V$. 
		\end{enumerate}
	\end{enumerate}
\item Let $j_x$ be the last entry of $U$.
\item Initialize a counter $z := 1$.
\item For $i = 1$ to $w_y$:
	\begin{enumerate}
	\item Set $u_\temp:= | U |$.
	\item For $k = 1$ to $|T_{y,i}|$:
		\begin{enumerate}
		\item Let $t_{y,i,k,0}$ be the $k^{th}$ $j$-invariant in $T_{y,i}$, i.e. $j_{y,i,k}$;
		\item For $h = 1$ to $u_\temp$: 
			\begin{itemize}
			\item If $k=1$, compute $t_{y,i,k,h}: = \gcdop(\PP, L_{y,i}, V_h; t_{y,i, k,h-1}, U_h)$;
			\item If $k>1$, compute $t_{y,i,k,h}: = \gcdop(\PP, L_{y,i}, V_h; t_{y,i, k,h-1}, t_{y,i, k-1,h})$;
			\end{itemize}
		\item Let $j_\temp$ be the last entry in $U$. 
		\item $(*)$ Produce an isogeny $\phi_z$ of degree $L_{y,i}$ such that $\phi_z(j_\temp) = t_{y,i, k,u_\temp}$.
		\item Increase the counter $z := z+1$.
		\item Append $t_{y,i, k,u_\temp}$ to the list $U$, append $L_{y,i}$ to the list $V$. 
		\end{enumerate}
	\end{enumerate}
\item Let $\phi_y := (\phi_1, ..., \phi_{n})$, where $n = z-1$.
\item Output $j_x$ and $\phi_y$. 
\end{enumerate}
\end{algorithm}
Note that the operation $(*)$ can be performed efficiently due to Proposition~\ref{claim:kernelpoly}.

The partial extraction algorithm is used in the directed transitive signature scheme to compress a composed signature. In fact, in the DTS scheme, we need to efficiently verify that $\phi_y$ does represent the ideal class $C_y\in\G$ defined by $\enc(y)$. We don't know how to verify that solely from $\enc(y)$ and the rational polynomial of $\phi_y$ over $\ZNZ$. (As mentioned in Remark~\ref{remark:isogenyoverZNZ}, it is not clear how to recover the explicit ideal class $C$ given an isogeny $\phi$ over $\ZNZ$. Of course we can recover the norm of the ideal from the degree of $\phi$, but normally there are two ideals of the same prime norm.) 
Instead we provide a succinct proof on the correctness of the execution of Algorithm~\ref{alg:partialconvert}.

Let us recall the definition of a succinct non-interactive argument (SNARG).
\begin{definition}
Let $\CCC = \set{ C_\secp: \zo^{g(\secp)}\times \zo^{h(\secp)}\to \zo }_{\secp\in\N}$ be a family of Boolean circuits. A SNARG for the instance-witness relation defined by $\CCC$ is a tuple of efficient algorithms $(\Gen, \Prove, \Verify)$ defined as:
\begin{itemize}
	\item $\Gen(1^\secp)$ take the security parameter $\secp$, outputs a common reference string $\CRS$.
	\item $\Prove(\CRS,x,w)$ takes the $\CRS$, the instance $x\in \zo^{g(\secp)}$, and a witness $w\in\zo^{h(\secp)}$, outputs a proof $\pi$.
	\item $\Verify(\CRS,x,\pi)$ takes as input an instance $x$, and a proof $\pi$, outputs 1 if it accepts the proof, and 0 otherwise.
\end{itemize}
It satisfies the following properties:
\begin{itemize}
	\item Completeness: For all $x\in\zo^{g(\secp)}, w\in\zo^{h(\secp)}$ such that $C(x,w)=1$:
	\[ \Pr[ \Verify(\CRS,x,\Prove(\CRS,x,w)) = 1   ]=1   . \]
	\item Soundness: For all $x\in\zo^{g(\secp)}$ such that $C(x,w)=0$ for all $w\in\zo^{h(\secp)}$, for all polynomially bounded cheating prover $P^*$:
	\[ \Pr[ \Verify(\CRS,x, P^*(\CRS,x)) = 1   ]\leq \negl(\secp)   . \]
	\item Succinctness: There exists a universal polynomial $Q$ (independent of $\CCC$) such that $\Gen$ runs in time $Q(\secp + \log |\CCC_\secp|)$, the length of the proof output by $\Prove$ is bounded by $Q(\secp + \log |\CCC_\secp|)$, and $\Verify$ runs in time $Q(\secp + g(\secp) + \log |\CCC_\secp|)$.
\end{itemize}
\end{definition}
A construction of SNARG in the random oracle model is given by \cite{DBLP:journals/siamcomp/Micali00}.

With a SNARG in hand, we can add in Algorithm~\ref{alg:partialconvert} a proof $\pi$ for the instance $(\enc(y), \phi_y)$ and statement ``there exists an encoding $\enc(x)$ such that $\phi_y$ is computed from running a circuit that instantiates Algorithm~\ref{alg:partialconvert} on inputs $\enc(x), \enc(y)$''. 
Here the public parameter $\PP$ is hardcoded in the circuit. The encoding $\enc(x)$ is the witness in the relation. 
The soundness of SNARG guarantees that $\phi_y$ is an isogeny that corresponds to an $\OOO$-ideal represented by $\enc(y)$. The succinctness of SNARG guarantees that the length of the proof $\pi$ is $\poly\log(|\enc(x)|)$ and the time to verify the proof is $\poly(|\enc(y)|, \log(|\enc(x)|))$.

\paragraph{Alternative choices for the class group. }
Ideally we would like to efficiently sample an imaginary quadratic order $\OOO$ of large discriminant $D$, together with the class number $h(D)$, a generation set $\set{ C_i }$ and a short basis $\mat{B}$ for $\Lattice_\OOO$, and two large primes $p$, $q$ as well as curves $E_{0,\F_p}, E_{0,\F_q}$ such that $\Endo(E_{0,\F_p})\simeq\Endo(E_{0,\F_q})\simeq \OOO$. 
In Section~\ref{sec:cryptanalysis} we will explain that if $|D|$ is polynomial then computing the group inversion takes polynomial time, so we are forced to choose a super-polynomially large $|D|$. 
On the other hand, there is no polynomial time solution for the task of choosing a square-free discriminant $D$ with a super-polynomially large $h(D)$ (see, for instance, \cite{DBLP:conf/asiacrypt/HamdyM00}).


In the basic setting of the parameter, we have described a solution where the resulting $D$ is not square-free, of size $\approx\secp^{O(\log\secp)}$, and polynomially smooth. 
Here we provide the background if one would like to work with a square-free discriminant $D$. In this case, $\OOO$, the ring of integers of $\Q(\sqrt{D})$, is the maximal order of an imaginary quadratic field $K$. 

According to the Cohen-Lenstra heuristics \cite{cohenheuristics}, about $97.7575\%$ of the imaginary quadratic fields $K$ have the odd part of $\CL(\OOO_K)$ cyclic. If we choose $D$ such that $|D|\equiv 3\bmod 4$ and is a prime, then $h(D)$ is odd (by genus theory). So we might as well assume that $\CL(\OOO_K)$ is cyclic with odd order. 

For a fixed discriminant $D$, heuristically about half of the primes $\ell$ satisfies $\kron{D}{\ell} = 1$, so there are polynomially many ideals of polynomially large norm that can be used in the generation set.

However, it is not clear how to efficiently choose $p$, $q$ and curves $E_{0,\F_p}$, $E_{0,\F_q}$ such that the endomorphism rings of $E_{0,\F_p}$ and $E_{0,\F_q}$ have the given discriminant $D$ of super-polynomial size. 
The classical CM method (cf. \cite{lay1994constructing} and more) requires computing the Hilbert class polynomial $H_D$, whose cost grows proportional to $|D|$. 
Let us remark that the CM method might be an overkill, since we do not need to specify the number of points $\#(E_{0,\F_p}(\F_p))$ and $\#(E_{0,\F_q}(\F_q))$. However, we do not know any other better methods.

\section{Cryptanalysis}\label{sec:cryptanalysis}

In this section we will discuss our cryptanalysis attempts, and the countermeasures.

Central to the security of our cryptosystem is the conjectured hardness of solving various problems over $\ZNZ$ without knowing the factors of $N$. So we start from the feasibility of performing several individual computational tasks over $\ZNZ$; then focus on the $(\ell, \ell^2)$-isogenous neighbor problem over $\ZNZ$, whose hardness is necessary for the security of our candidate TGII; finally address all the other attacks in the TGII construction.




\subsection{The (in)feasibility of performing computations over $\ZNZ$}

The task of finding roots of polynomials of degree $d \geq2$ over $\ZNZ$ sits in the subroutines of many potential algorithms we need to consider, so let us begin with a clarification on the status of this problem. Currently, no polynomial time algorithm is known for solving this problem in general. The hardness of a few special instances have been extensively studied. They fall into the following three categories:

\begin{enumerate}
\item For certain families of polynomials, it is known that finding a root of them over $\ZNZ$ is as hard as factorizing $N$. For example, the family of polynomials $\set{f_a(x)=x^2-a}_{a\in(\ZNZ)^{\times}}$
whose potential roots are the solutions for the quadratic residue problem \cite{rabin1979digitalized}. 
\item There are families of polynomials, where finding at least one root is feasible. For example, if a root of a polynomial over $\ZNZ$ is known to be the same as the root over Q, then we can use LLL \cite{lenstra1982factoring}; or if a root is known to be smaller than roughly $O(N^{1/d})$, then Coppersmith-type algorithms can be used to find such a root \cite{DBLP:journals/joc/Coppersmith97}. But these families of polynomials only form a negligible portion of all the polynomials with polynomially bounded degrees.
\item The majority of the polynomials seem to live in the third case, where finding a root is conjectured to be hard, however the hardness is not known to be based on integer factorization. Among them, some specific families are even conjectured to be unlikely to have a reduction from integer factorization. For example, when $\gcd(3,\phi(N)) = 1$, the family of polynomials $\set{f_a(x)=x^3-a}_{a\in (\ZNZ)^{\times}}$ is conjectured to be hard to solve, and unlikely to be as hard as integer factorization \cite{boneh1998breaking}. 
\end{enumerate}

\subsubsection{Feasible information from a single $j$-invariant }
Let $N = pq$, as before, where $p$ and $q$ are large primes. From any $j \in \ZNZ, j\neq 0,1728$, we can easily find the coefficients $a$ and $b$ of the Weierstrass form of an elliptic curve $E(\ZNZ)$ with $j(E) = j$ by computing $a=3j(1728-j), b = 2j(1728-j)^2$. However, this method does not guarantee that the curve belongs to a specific isomorphism class (there are four of them). By choosing a value $u \in (\ZNZ)^{\times}$ and let $a^* = u^4  a, b^* = u^6  b$ one gets the coefficient of another curve with the same $j$-invariant, each belonging to one of the four isomorphism classes.


On the other hand, choosing a curve over $\ZNZ$ with a given $j$-invariant together with a point on the curve seems tricky. Note that there are methods of sampling a curve with a known $\ell$-torsion point over $\ZNZ$ (cf. \cite{kubert1976universal,sutherland2012constructing}). But with the additional constraint on the $j$-invariant, these methods seem to require solving a non-linear polynomial over $\ZNZ$.

Nevertheless, it is always feasible to choose a curve together with the $x$-coordinate of a point on it, since a random $x \in \ZNZ$ is the $x$-coordinate of some point on the curve with probability roughly $\frac12$. It is also known that computing the multiples of a point $P$ over $E(\ZNZ)$ is feasible solely using the $x$-coordinate of $P$ (cf. \cite{DBLP:conf/eurocrypt/Demytko93}). The implication of this is that we should at the very least not give out the group orders of the curves involved in the scheme. More precisely, we should avoid the $j$-invariants corresponding to curves (or their twists) with polynomially smooth cardinalities over either $\F_p$ or $\F_q$. Otherwise Lenstra's algorithm \cite{lenstra1987factoring} can be used to factorize $N$. 

In our application we also assume that the endomorphism rings of $E(\F_p)$ and $E(\F_q)$ are isomorphic and not given out to begin with. Computing the discriminant of $\OOO \simeq \Endo(E(\F_p)) \simeq \Endo(E(\F_q))$ or the number of points of $E$ over $\Z/N\Z$ seems to be hard given only $N$ and a $j$-invariant. In fact Kunihiro and Koyama (and others) have reduced factorizing $N$ to computing the number of points of general elliptic curves over $\Z/N\Z$ \cite{DBLP:conf/eurocrypt/KunihiroK98}. However, these reductions are not efficient in the special case, where the endomorphism rings of $E(\F_p)$ and $E(\F_q)$ are required to be isomorphic. So, the result of \cite{DBLP:conf/eurocrypt/KunihiroK98} can be viewed as evidence that the polynomial time algorithms for counting points on elliptic curves over finite fields may fail over $\Z/N\Z$ without making use of the fact that the endomorphism rings of $E(\F_p)$ and $E(\F_q)$ are isomorphic.



\subsubsection{Computing explicit isogenies over $\ZNZ$ given more than one $j$-invariant}\label{sec:explicit_iso_znz}

Let $\ell$ be a prime. We will be concerned with degree $\ell$ isogenies. If we are only given a single $j$-invariant $j_1 \in\Z/N\Z$, then finding an integer $j_2$ such that $\Phi_\ell(j_1, j_2) = 0 \pmod N$ seems hard. We remark that Theorem~\ref{thm:basicequalfactoring} does not guarantee that finding $j_2$ is as hard as factoring when the endomorphism rings of $E(\F_p)$ and $E(\F_q)$ are isomorphic. However, as of now, we do not know how to make use of the condition that the endomorphism rings are isomorphic to mount an attack on the problem.

Of course, in the construction of a TGII, we are not only given a single $j$-invariant, but many $j$-invariants with each neighboring pair of them satisfying the $\ell^{th}$ modular polynomial, a polynomial degree $\ell + 1$. We will study what other information can be extracted from these neighboring $j$-invariants.

In Proposition~\ref{claim:kernelpoly}, we have explained that given two integers $j_1, j_2\in\ZNZ$ such that $\Phi_{\ell}(j_1, j_2) = 0$ over $\ZNZ$, the elliptic curves $E_1$, $E_2$, and the kernel polynomial $h(x)$ of an isogeny $\phi$ from $E_1$, $E_2$ can be computed in time polynomial in $\ell, \log(N)$. However, it is not clear how to use the explicit expression of $\phi$ to break factoring or solve the inversion problem. 

A natural next step is to recover a point in the kernel of $\phi$, but it is also not clear how to recover even the $x$-coordinate of a point in the kernel when $\ell \geq 5$. For $\ell=3$, on the other hand, the kernel polynomial does reveal the $x$-coordinate of a point $P$ in the kernel $G\subset E_1[3]$ (note that $h(\cdot)$ is of degree $1$ in this particular case). But revealing the $x$-coordinate of a point $P\in E_1[3]$ does not immediately break factoring, since $3P$ is $O$ over both $\F_p$ and $\F_q$. At this moment we do not know of a full attack from a point in $\ker(\phi)$. Nevertheless, we still choose to take an additional safeguard by avoiding the use of 3-isogenies since it reveals the $x$-coordinate of a point in $E_1[3]$, and many operations on elliptic curves are feasible given the $x$-coordinate of a point.

\subsection{Tackling the $(\ell, \ell^2)$-isogenous neighbor problem}\label{sec:analysis_ell_ell2}

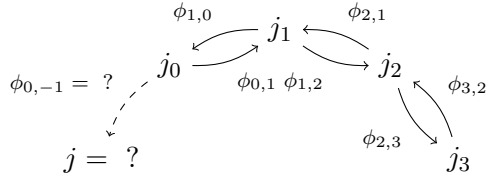
\begin{figure}
  \centering
    \begin{tikzpicture}
      \def\angleshift{72}
      \def\n{3}
      \foreach \i in {0,...,\n} {
        \pgfmathparse{360/(\n+7)*(3/2-\i) + \angleshift}
        \let\angle\pgfmathresult
        \draw (\angle:2.5) node (E\i) {$j_\i$};
      }
        \pgfmathparse{360/(\n+7)*(3/2+1) + \angleshift}
        \let\angle\pgfmathresult
        \draw (\angle:2.5) node (E-1) {$j = ~ ?$};
      \foreach \i in {0,...,2} {
        \pgfmathparse{int(mod(\i+1, \n+1))}
        \let\j\pgfmathresult
        \draw (E\i) edge[->,bend right=18] node[auto,swap] {\scriptsize$\phi_{\i,\j}$} (E\j);
      }
      \foreach \i in {0,...,2} {
        \pgfmathparse{int(mod(\i+1, \n+1))}
        \let\j\pgfmathresult
        \draw (E\j) edge[->,bend right=18] node[auto,swap] {\scriptsize$\phi_{\j,\i}$} (E\i);
      }
        \draw (E0) edge[dashed,->,bend right=18] node[auto,swap] {\scriptsize$\phi_{0,-1} = ~?$} (E-1);
    \end{tikzpicture}
   \caption{ A pictorial description of the $(\ell, \ell^2)$-isogenous neighbor problem. }
  \label{fig:volcano_subgroup}
\end{figure}

The hardness of infeasible inversion of our group representation in fact relies on the hardness of the following generalization of the $(\ell, \ell^2)$-isogenous neighbor problem (cf. Definition~\ref{def:problem:joint}). Let the modulus $N=p q$, where the large prime factors $p$ and $q$ are hidden. Let $\ell$ be the degree of the isogenies. For a polynomially large $k\in\N$ and a given a sequence of integers $j_0$, $j_1$, $j_2$, ... $j_k$ such that $\Phi_{\ell}( j_{i-1}, j_i ) = 0 \pmod N$ for all $i\in[k]$, the problem asks to find an integer $j_{-1}$ such that $\Phi_{\ell^{i+1}}( j_{-1}, j_i ) = 0 \pmod N$ for all $i\in[k]$. In addition, for all $i\in\set{0, 1, ..., k}$ the endomorphism rings of $E_{i,\F_p}$, $E_{i,\F_q}$ are isomorphic to an imaginary quadratic order $\OOO$ ($\OOO$ is supposed to be hidden for reasons to be explained later). See Figure~\ref{fig:volcano_subgroup} for the pictorial description of the problem.

From the discussion in \S~\ref{sec:explicit_iso_znz} we know that it is feasible to compute the kernel polynomials of the isogenies $\phi_{i, i+1}$ and their duals $\phi_{i+1, i}$, for $i = 0, ..., k-1$. Denote these kernel polynomials by $h_{i, i+1}$ and $h_{i+1, i}$ respectively. If one can compute the kernel polynomial $h_{0, -1}$ of the isogeny $\phi_{0, -1}$ that maps from $j_0$ to $j$ one can then solve the $(\ell,\ell^2)$-problem. Since $h_{0, 1}$ can be recovered, there is a chance of obtaining $h_{0, -1}$ from $h_{0, 1}$. As mentioned in Remark~\ref{remark:idealandisogeny}, over a finite field, the two kernel polynomials $h_{0, -1}$ and $h_{0, 1}$ of the two horizontal isogenies can be explicitly related via the Frobenius endomorphism. However it is not clear how to use the relation over $\Z/N\Z$. 

Another attempt is to see if $\phi_{1,0}(\kernel(\phi_{1,2}))$ gives $\kernel(\phi_{0,-1})$ (the map $\phi_{1,0}(\kernel(\phi_{1,2}))$ can be computed using Newton's identities). However, applying $\phi_{1,0}$ on the kernel subgroup $G_{1,2}$ of $\phi_{1,2}$ gives $\kernel(\phi_{0,1})$, since $\phi_{0,1}\circ \phi_{1,0}(G_{1,2}) = [\ell]G_{1,2} = O$, so it does not give the desired kernel subgroup.

\subsubsection{Hilbert class polynomial attack}\label{sec:attack_hilbert}

Let $D$ be the discriminant of the imaginary quadratic order $\OOO$ that we are working with. If computing the Hilbert class polynomial $H_D$ is feasible, then we can solve the $(\ell, \ell^2)$-isogenous neighbor problem. Given $j_0, j_1$ such that $\Phi_\ell(j_0, j_1) = 0$, compute the polynomial $\gamma(x)$,
\[ \gamma(x) := \gcd( \Phi_{\ell}(j_0, x), \Phi_{\ell^2}(j_1, x), H_{D}(x) ) \in(\ZNZ)[x] .\] 
The gcd of $\Phi_{\ell}(j_0, x)$ and $\Phi_{\ell^2}(j_1, x)$ gives a polynomial of degree $\ell$. The potential root they share with $H_{D}(x)$ is the only one with the same endomorphism ring with $j_0$ and $j_1$, which is $j_{-1}$. So $\gamma(x)$ is a linear function. 

This attack is ineffective when $D$ is super-polynomial.

\subsubsection{Deciding the direction of an isogeny on the volcano}

Fix an isogeny volcano $G_{\ell}(\F_q)$ over a finite field $\F_q$. As we mentioned, given a $j$-invariant $j$ on $G_\ell(\F_q)$ there are efficient algorithms that output all the $\ell$-isogenous neighbors of $j$. When there are $1$ or $2$ neighbors, we know $j$ is at the bottom of the volcano (the case of $2$ neighbors corresponding to when the volcano consists just of the surface). When there are more than $2$ (i.e. $\ell+1$) neighbors, can we decide which isogeny is ascending, horizontal, or descending? The method mentioned in \cite{kohel1996endomorphism,fouquet2002isogeny,sutherland2013isogeny} takes a trial and error approach. It picks a random neighbor and goes forward, until the path reaches one of the terminating conditions. 
\ifnum\eprint=0
The only algorithm that is able to produce an isogeny with a designated direction is given by Ionica and Joux \cite{DBLP:journals/moc/IonicaJ13}. They recognize an invariant related to the group structure of the curves lying on the same level of the volcano.

In our application we are mostly interested in finding the unvisited neighbor over $\ZNZ$ that lies on the same level of the volcano, which equals to finding the horizontal isogeny of a given curve. As we conjectured, the algorithm in \cite{DBLP:journals/moc/IonicaJ13} do not extend to $\ZNZ$. So if the task of finding an isogeny with a designated direction is feasible over $\Z/N\Z$, then it is likely to imply a new algorithm of the same task over the finite field, which seems to be challenging on its own.
\else
For instance, if it reaches a point where there is only one neighbor, then that means the path is descending; if the path forms a loop, or takes longer than the estimated maximum depth of the volcano, then the initial step is ascending or horizontal. 

The only algorithm that is able to produce an isogeny with a designated direction is given by Ionica and Joux \cite{DBLP:journals/moc/IonicaJ13}. They recognize an invariant related to the group structure of the curves lying on the same level of the volcano. The highlevel structure of the algorithm is as follows:
\begin{enumerate}
\item First decide the group structure via a pairing (\cite{DBLP:journals/moc/IonicaJ13} uses reduced Tate pairing).
\item Then find the kernel subgroup with a property that depends on whether the desired isogeny is ascending, descending, or horizontal.
\end{enumerate}
Either steps seem to carry out over $\ZNZ$. One of the main barriers is to compute (even the $x$-coordinate of) a point that sits in the specific subgroup of the curve.

The precise descriptions of the invariant and the algorithm are rather technical. So we only sketch the main theorem from \cite{DBLP:journals/moc/IonicaJ13}, skipping many technical details. 
Let $n\geq 0$, let $E[\ell^n](\F_{q^k})$ be the subgroup of points of order $\ell^n$ defined over an extension field over $\F_q$. Let $E[\ell^\infty](\F_{q})$ be the $\ell$-Sylow subgroup of $E(\F_q)$. 

Let $m$ be an integer such that $m\mid \# E(\F_q)$. Let $k$ be the embedding degree, i.e. the smallest integer that $m\mid q^k-1$. 
Let $T_m: E[m](\F_{q^k})\times E(\F_{q^k})/m E(\F_{q^k}) \to \mu_m$ be the reduced Tate pairing, where $\mu_m$ denotes the $m^{th}$ roots of unity. Furthermore, define the symmetric pairing 
\[ S(P, Q) = (T_{\ell^n}(P, Q)\, T_{\ell^n}(Q, P))^{1/2},  \]
and call $S(P, P) = T_{\ell^n}(P, P)$ the self-pairing of $P$.

\begin{theorem}[The main theorem of \cite{DBLP:journals/moc/IonicaJ13}, informally]
Let $E$ be an elliptic curve defined a finite field $\F_q$ and let $E[\ell^\infty](\F_q)$ be isomorphic to $\Z/\ell^{n_1} \Z \times \Z/\ell^{n_2}\Z$ with $n_1 \geq n_2 \geq 1$. Then 
\begin{itemize}
\item Suppose $P$ is a $\ell^{n_2}$-torsion point such that $T_{\ell^{n_2}}(P,P)$ is a primitive ${\ell^{n_2}}$th root of unity. Then the $\ell$-isogeny whose kernel is generated by $\ell^{n_2-1} P$ is descending. 
\item 
Suppose certain condition holds, and let $P$ be a $\ell^{n_2}$-torsion point with degenerate self-pairing. Then the $\ell$-isogeny whose kernel is generated by $\ell^{n_2-1} P$ is either ascending or horizontal. Moreover, for any $\ell^{n_2}$-torsion point $Q$ whose self-pairing is non-degenerate, the isogeny with kernel spanned by $\ell^{n_2-1} Q$ is descending.
\end{itemize}
\end{theorem}

Carrying out the algorithm of finding an isogeny with a specific direction (say horizontal) requires finding (even merely the $x$-coordinate of) a point in the designated subgroup. Over a finite field, such points can be found efficiently by picking a random point $R$ and compute $P = \frac{\#(E(\F_q))}{r} R$, where $r$ is the order of the designated subgroup; then test if the candidate point $P$ lies in the correct subgroup by taking pairing. Over $\ZNZ$ it seems hard even to find a point with a specific order since we do not know $\# E(\ZNZ)$. 
Even if we are able to find points in desired subgroups, there are additional technicalities such as computing pairing over $\ZNZ$ (cf. \cite{galbraith2005pairings}).

In our application we are mostly interested in finding the unvisited neighbor over $\ZNZ$ that lies on the same level of the volcano, which equals to finding the horizontal isogeny of a given curve. As we conjectured, the algorithm in \cite{DBLP:journals/moc/IonicaJ13} do not extend to $\ZNZ$. So if the task of finding an isogeny with a designated direction is feasible over $\Z/N\Z$, then it is likely to imply a new algorithm of the same task over the finite field, which seems to be challenging on its own.
\fi


\subsubsection{More about modular curves and characteristic zero attacks}
Given $j$, solving $\Phi_{\ell}(j,x)$ is not the only way to find the $j$-invariants of the $\ell$-isogenous curves. Alternative complex analytic (i.e. characteristic zero) methods have been discussed, for instance, in \cite[Section~3]{elkies1998elliptic}. However, these methods all involve solving polynomials of degree $\geq2$ to get started.

As mentioned in Section~\ref{sec:prelim:isogeny}, the curve $\uhp/\Gamma_0(\ell)$ parameterizes pairs of elliptic curves over $\C$ related by a cyclic $\ell$-isogeny. The $(\ell,\ell^2)$-isogenous neighbor problem, on the other hand, concerns curves that are horizontally $\ell$-isogenous, i.e. $\ell$-isogenous and have the same endomorphism ring. To avoid an attack through characteristic zero techniques, we make sure that there is no immediate quotient of $\uhp$ that parametrizes curves which are related with an $\ell$-isogeny and have the same endomorphism ring. Below, we first go over the well-known moduli description of modular curves\footnote{Recall that a modular curve is a quotient of the extended upper half plane by a congruence subgroup, and a congruence subgroup is a subgroup of $\SL_2(\Z)$, which contains $\Gamma(\ell)$.} to make sure that they don't lead to an immediate attack, and then show that there is indeed no quotient of $\uhp$ between $\uhp/\SL_2(\Z)$ and $\uhp/\Gamma_0(\ell)$, so we don't have to worry about possible attacks on that end.

Let $\Gamma: = \SL_2(\Z)$, and let $\Gamma(\ell)$ and $\Gamma_1(\ell)$ denote the congruence subgroups,

\[ \Gamma(\ell):=\set{ \pmat{ a & b \\ c & d  }\in\SL_2(\Z) \bigg\vert \pmat{ a & b \\ c & d  } \equiv \pmat{ 1 & 0 \\ 0 & 1  }\pmod \ell  }, \]
\[ \Gamma_1(\ell):=\set{ \pmat{ a & b \\ c & d  }\in\SL_2(\Z) \bigg\vert \pmat{ a & b \\ c & d  } \equiv \pmat{ 1 & * \\ 0 & 1  }\pmod \ell  }. \]

It is well-known that the curves $\uhp/\Gamma_1(\ell)$ and $\uhp/\Gamma(\ell)$ parametrize elliptic curves with extra data on their $\ell$-torsion (cf. \cite{kohel1996endomorphism}). $\uhp/\Gamma_1(\ell)$ parametrizes $(E, P)$, where $P$ is a point on $E$ having order exactly $\ell$, and $\uhp/\Gamma(\ell)$ parametrizes triples $(E,P,Q)$, where $E[\ell]=\langle P,Q\rangle$ and they have a fixed Weil pairing. These curves carry more information than the $\ell$-isogenous relation and they are not immediately helpful for solving the $(\ell, \ell^2)$-isogenous neighbor problem.

As for the quotients between $\uhp/\SL_2(\Z)$ and $\uhp/\Gamma_0(\ell)$, the following lemma shows that there are indeed none.



\begin{lemma}
Let $\ell$ be a prime. If $H\leq \Gamma$ is such that $\Gamma_0(\ell)\leq H\leq \Gamma$, then either $H=\Gamma_0(\ell)$ or $H=\Gamma$.
\end{lemma}
\begin{proof} Let $\sigma_1=\left(\begin{smallmatrix}1&1\\0&1\end{smallmatrix}\right)$, $\sigma_2=\left(\begin{smallmatrix}1&0\\1&1\end{smallmatrix}\right)$, $\sigma_3=\sigma_1\sigma_2^{-1}$, and recall that $SL_2(\Z/\ell\Z)=\langle\sigma_1,\sigma_2\rangle=\langle\sigma_1,\sigma_3\rangle$. Recall that the natural projection $\pi: \Gamma \rightarrow SL_2(\Z/\ell\Z)$ is surjective. Assume that $H\neq \Gamma_0(\ell)$. This implies that $\pi(H)=SL_2(\Z/\ell\Z)$ (we shall give a proof below). Assuming this claim for the moment let $g\in \Gamma\backslash H$. Since $\pi(\Gamma)=\pi(H) $ there exists $h\in H$ such that $\pi(g)=\pi(h)$. Therefore, $gh^{-1}\in \ker(\pi)=\Gamma(\ell)\subset H$. Therefore, $g\in H$ and $\Gamma=H$.

To see that $\pi(\Gamma)=\pi(H)$, first note that since $\Gamma_0(\ell) \subset H$ we have  all the upper triangular matrices in $ \pi(H)$. Next, let $h=\left(\begin{smallmatrix}h_1&h_2\\h_3&h_4\end{smallmatrix}\right)\in H\backslash \Gamma_0(\ell)$ such that $\pi(h)=\left(\begin{smallmatrix}\bar{h}_1&\bar{h}_2\\ \bar{h}_3&\bar{h}_4\end{smallmatrix}\right)\in \pi(H)\backslash \pi(\Gamma_0(\ell))$ (note that this difference is non-empty since otherwise $\Gamma_0(\ell)=H$). \\
We have two cases depending on $\bar{h}_1= 0$ or not. If $\bar{h}_1=0$ then $\bar{h}_3\neq0$ and $\sigma_3=\left(\begin{smallmatrix}\bar{h}_3^{-1}&\bar{h}_4\\ 0&\bar{h}_3\end{smallmatrix}\right)\bar{h}^{-1}\in \pi(H)$. On the other hand, if $\bar{h}_1\neq0$ multiplying on the right by $\left(\begin{smallmatrix}\bar{h}_1^{-1}&-\bar{h}_2\\0&\bar{h}_1\end{smallmatrix}\right) \in \pi(H)$ we see that $\left(\begin{smallmatrix}1&0\\ \bar{h}_3\bar{h}_1^{-1}&1\end{smallmatrix}\right) \in \pi(H)$ . For any integer $m$, the $m$'th power of this matrix is $\left(\begin{smallmatrix}1&0\\ m\bar{h}_3\bar{h}_1^{-1}&1\end{smallmatrix}\right) \in \pi(H)$. Taking $m\equiv \bar{h}_1\bar{h}_3^{-1}$ shows that $\sigma_2\in \pi(H)$. This shows that $\pi(H)=SL_2(\Z/\ell\Z)$.
\end{proof}

Let us also remark that for special values of $\ell$ the $(\ell,\ell^2)$-isogenous neighbor problem may be more prone to characteristic zero attacks (although we do not know of such an attack). For instance, for those $\ell$ for which $\uhp/\Gamma_0(\ell)$ (more precisely, its compactification $X_0(\ell)$) is genus $0$ or $1$, it (may) have many rational points. These points, in turn, can be used to get points over $\ZNZ$ without knowing the factorization of $N$. Nevertheless, we may just avoid these values in the system. As mentioned above, we currently do not see an attack based on this. This point is brought up just as an extra precautionary measure.

\subsection{Cryptanalysis of the candidate group with infeasible inversion}\label{sec:attack_concrete_TGII}

We now cryptanalyze the concrete candidate TGII. Recall the format of an encoding of a group element $x$ from Definition~\ref{def:represent_enc}:
\[ \enc(x) = (L_x; T_{x,1}, ..., T_{x,w_x}) = ((p_{x,1}, ..., p_{x,w_x}); (j_{x,1,1}, ..., j_{x,1,e_{x,1}}), ..., (j_{x,w_x,1}, ..., j_{x,w_x,e_{x,w_x}}) ) .\]
The ``exponent vector" $\ary{e}_x\in\Z^{w_x}$ can be read from the encoding as $ \ary{e}_x = ( |T_{x,1}|, ..., |T_{x,w_x}| )$. 

We assume polynomially many composable encodings are published in the applications of a TGII. In down-to-earth terms it means the adversary is presented with polynomially many $j$-invariants on the crater of a volcano, and the explicit isogenies between each pair of the neighboring $j$-invariants (due to Proposition~\ref{claim:kernelpoly}). 

We will be considering the following model on the adversary's attacking strategy:
\begin{definition}[The GCD attack model]\label{def:gcd_eval}
In the GCD attack model, the adversary is allowed to try to find the inverse of a target group element only by executing the unit gcd operation $\gcdop(\PP, \ell_1, \ell_2; j_1, j_2)$ given in Algorithm~\ref{alg:gcdop} for polynomially many steps, where $\ell_1, \ell_2; j_1, j_2$ are from the published encodings or obtained from the previous executions of the gcd evaluations.
\end{definition}
We do not know how to prove the construction of TGII is secure even if the adversary is restricted to attack in the GCD model. Our cryptanalysis attempts can be classified as showing (1) how to prevent the attacks that obey the GCD evaluation law\ifnum\eprint=1  ~(mainly in \S\ref{sec:attack_concrete_TGII_trivial} and \S\ref{sec:parallelogram_attack}); \else; \fi (2) how to prevent the other attack approaches.

\subsubsection{Preventing the trivial leakage of inverses}\label{sec:attack_concrete_TGII_trivial}
In applications we are often required to publish the encodings of elements that are related in some way. A typical case is the following: for $x, y\in\CL(\OOO)$, the scheme may require publishing the encodings of $x$ and $z = y\circ x^{-1}$ without revealing a valid encoding of $x^{-1}$. As a toy example, let $x = [(p_x, b_x, \cdot)]$, $y = [(p_y, b_y, \cdot)]$, where $p_x$ and $p_y$ are distinct primes. Let $j_0$, the $j$-invariant of a curve $E_0$, represent the identity element in the public parameter. Let $((p_x); (j_x))$ be a composable encoding of $x$ and $((p_y); (j_y))$ be a composable encoding of $y$.

Naively, a composable encoding of $z = y\circ x^{-1}$ could then be $((p_x, p_y); (j_{x^{-1}}), (j_{y}))$, where $j_{x^{-1}}$ is the $j$-invariant of $E_{x^{-1}} = x^{-1}E_0$. Note, however, that $((p_x); (j_{x^{-1}}))$ is a valid encoding of $x^{-1}$. In other words such an encoding of $y\circ x^{-1}$ trivially reveals the encoding of $x^{-1}$. 

One way of generating an encoding of $z = y\circ x^{-1}$ without trivially revealing $j_{x^{-1}}$ is to first pick a generator set of ideals where the norms of the ideals are coprime to $p_x$ and $p_y$, then solve the discrete-log of $z$ over these generators to compute the composable encoding. This is the approach we take in this paper.



\subsubsection{Parallelogram attack}\label{sec:parallelogram_attack}
In the applications we are often required to publish the composable encodings of group elements $a$, $b$, $c$ such that $a\circ b = c$. If the degrees of the three encodings are coprime, or in general, when $\enc(a)$, $\enc(b)$ and $\enc(c)$ are pairwise composable, then we can recover the encodings of $a^{-1}$, $b^{-1}$, and $c^{-1}$ using the following ``parallelogram attack". This is a non-trivial attack which obeys the gcd evaluation law in Definition~\ref{def:gcd_eval}.

Let us illustrate the attack via the examples in Figure~\ref{fig:parallelogram}, where the solid arrows represent the isogenies that are given as the inputs (the $j$-invariants of the target curves are written at the head of the arrows, their positions do not follow the relative positions on the volcano; the degree of the isogeny is written on the arrow); the dashed lines and the $j$-invariants on those lines are obtained from the gcd evaluation law. 

\begin{figure}
  \centering
    \begin{tikzpicture}
      \begin{scope}[xshift=1cm, yshift = 0.7cm]
        \draw node at (0.  ,0) (O) {$j_0$};
        \draw node at (-0.4,3) (A) {$j_1$};
        \draw node at (2.4 ,0) (B) {$j_2$};
        \draw node at (2   ,3) (C) {$j_3$};
        \draw node at (-2.4  , 0)      (D) {$j_{4}$};

        \draw [->] (O) -- node[auto,swap] {\scriptsize$\ell_1$} (A);
        \draw [->] (O) -- node[auto,swap] {\scriptsize$\ell_2$} (B);
        \draw [->] (O) -- node[auto,swap] {\scriptsize$\ell_3$} (C);
        \draw (B) edge[dashed] (C);        
        \draw (A) edge[dashed] (C);        
        \draw (A) edge[dashed] (D);        
        \draw (O) edge[dashed] (D);        
      \end{scope}

      \begin{scope}[xshift=8cm, yshift = 0.7cm]
        \draw node at (0.  ,0) (O) {$j_0$};
        \draw node at (-0.4,2) (A2) {$j_4$};
        \draw node at (-0.6,3) (A) {$j_1$};
        \draw node at (1.4 ,0) (B2) {$j_5$};
        \draw node at (3.6 ,0) (B) {$j_2$};
        \draw node at (1.6 ,1.6) (C2) {$j_6$};
        \draw node at (3   ,3) (C) {$j_3$};

        \draw node at (1.4 - 0.4, 2) (B2-A2) {$j_7$};
        \draw node at (1.4 - 0.6, 3) (B2-A) {$j_8$};
        \draw node at (1.6 - 2.2 , 1.6) (D1) {$j_{9}$};
        \draw node at (-3.6 + 1.4 , 0) (D2) {$j_{10}$};
        \draw node at (-3.6 + 1.6 , 1.6) (D3) {$j_{11}$};
        \draw node at (-3.6  , 0)      (D4) {$j_{12}$};

        \draw [->] (O) -- node[auto,swap] {\scriptsize$\ell_4$} (A2);
        \draw [->] (O) -- node[auto,swap] {\scriptsize$\ell_5$} (B2);
        \draw [->] (O) -- node[auto,swap] {\scriptsize$\ell_6$} (C2);
        \draw [->] (A2) -- node[auto,swap] {\scriptsize$\ell_1$} (A);
        \draw [->] (B2) -- node[auto,swap] {\scriptsize$\ell_2$} (B);
        \draw [->] (C2) -- node[auto,swap] {\scriptsize$\ell_3$} (C);
        \draw (B2) edge[dashed] (B2-A);        
        \draw (B) edge[dashed] (C);        
        \draw (A) edge[dashed] (C);        
        \draw (A) edge[dashed] (D4);        
        \draw (O) edge[dashed] (D4);        
        \draw (C2) edge[dashed] (D3);        
        \draw (B2-A) edge[dashed] (D2); 
        \draw (B2-A2) edge[dashed] (A2);        
      \end{scope}
    \end{tikzpicture}
    \caption{The parallelogram attack.  }
  \label{fig:parallelogram}
\end{figure}
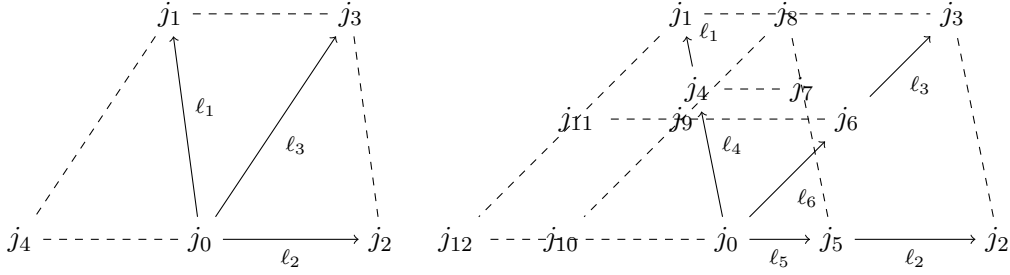

For simplicity let us first look at the example in the left figure. Let composable encodings of $a$, $b$, $c$ be simply given by $(\ell_1; (j_1))$, $(\ell_2; (j_2))$, $(\ell_3; (j_3))$, where $\ell_1$, $\ell_2$, $\ell_3$ are polynomial and pairwise coprime. 
A composable encoding of $b^{-1}$ then can be written as $(\ell_2; (j_4))$, where $j_4$ is the root of the linear equation $f(x) = \gcd( \Phi_{\ell_3}(j_1, x) , \Phi_{\ell_2}(j_0, x) )$. This is due to the relation $c\circ b^{-1} = a$, which, in particular implies that $j_1$ and $j_4$ are connected by an isogeny of degree $\ell_3$. 

The simple attack above uses the fact that the degrees of the entire encodings of $a$, $b$ and $c$ are polynomial. Let us use the example in the right figure to illustrate that even if the encodings are composed of many polynomial degree isogenies (so that the total degrees may be super-polynomial), the attack is still effective.
The idea is to view the composition law as filling the missing edges of a parallelogram given the $j$-invariants on a pair of adjacent edges. 
The final goal is to find the missing corner $j_{12}$ in the parallelogram $j_0 - j_3 - j_1 - j_{12}$. To arrive there we need the $j$-invariants on a pair of adjacent edges to begin with, so we first have to fill the $j$-invariants on, for instance, the edge $j_1-j_3$. 
Therefore, as the first step, we consider the parallelogram $j_0 - j_2 - j_3 - j_1$. To fill the $j$-invariants on the edge $j_1 - j_3$, we first compute $j_7$ as the root of $f_7(x) = \gcd( \Phi_{\ell_4}(j_5, x) , \Phi_{\ell_5}(j_4, x) )$, then compute $j_8$ as the root of $f_8(x) = \gcd( \Phi_{\ell_1}(j_7, x) , \Phi_{\ell_5}(j_1, x) )$ (the polynomials $f_7$, $f_8$ are linear since the degrees of $\enc(a)$ and $\enc(b)$ are coprime).
In the second step, we consider the parallelogram $j_0 - j_3 - j_1 - j_{12}$. To find $j_{12}$ we use the gcd evaluation law to find $j_9$, $j_{10}$, $j_{11}$, $j_{12}$ one-by-one (using the condition that the degrees of $\enc(c)$ and $\enc(b)$ are coprime). 

In the previous examples we illustrated that the attack is applicable when the degrees of $\enc(a)$, $\enc(b)$, $\enc(c)$ are coprime. Let us also remark that the attack applies as long as $\enc(a)$, $\enc(b)$, $\enc(c)$ are pairwise composable, i.e. they don't have to be of coprime degrees, which happens when the ladders in the general algorithms are applied (cf. Algorithm~\ref{alg:ladder} in \S\ref{sec:TGIIconstruction_general}).

The parallelogram attack is very powerful, in the sense that it is not preventable when application requires to publish the composable encodings of $a$, $b$, $c$ such that $a\circ b = c$, and $\enc(a)$, $\enc(b)$, $\enc(c)$ to be pairwise composable. However, the parallelogram attack does not seem to work when 2 out of the 3 pairs of $\enc(a)$, $\enc(b)$ and $\enc(c)$ are not composable. In the applications of directed transitive signature and broadcast encryption, there are encodings of $a$, $b$, $c$ such that $a\circ b = c$. Luckily, only one pair of the encodings among the three has to be composable to provide the necessary functionalities of these applications. We will explicitly mention how to choose the prime factors in the degrees of these encodings to prevent the parallelogram attack.

\subsubsection{Hiding the class group invariants: why and how}\label{sec:attack:basis}

The hardness of discrete-log problem over $\CL(D)$ is necessary for the infeasibility of inversion\footnote{The actual group with infeasible inversion is $\CL(D)_{\text{odd}}$; here we discuss over $\CL(D)$ for notational simplicity.}. Recall from Lemma~\ref{lemma:expander} that, under GRH, the graph $\mathcal{G}_{\OOO, m}(\K)$ for a sufficiently large $m$ is an expander, which means it is reasonable to assume that the closure of the composition of polynomially many $j$-invariants covers all the $h(D)$ $j$-invariants. 
So finding a composition of the published $j$-invariants that reaches a specific vertex (e.g. the vertex that represents the inverse of some encoded element) in the isogeny graph is feasible by solving the discrete-log problem over $\CL(D)$. 

However, in the basic setting of the parameters, the class number $h(D)$ is polynomially smooth, so the discrete-log problem over $\CL(D)$ can be solved in polynomial time once $h(D)$ is given, and $h(D)$ can be recovered from $D$ or any basis of a lattice $\Lattice$ of dimension $w$ such that $\Z^{w}/\Lattice \simeq \CL(D)$.  So we do need to hide the discriminant $D$, the class number $h(D)$, and any lattice $\Lattice$ defined above. 

If self-composition of an encoding is feasible, then one can efficiently guess all the polynomially smooth factors of $h(D)$. However for our construction self-composition is infeasible, due to the hardness of the $(\ell, \ell^2)$-isogenous neighbor problem.  Nevertheless, one can still attack by first guessing $D$ or $h(D)$, which takes $\secp^{O(\secp)}$ time according to the current setting of parameter.

We remark that if it is feasible to choose a super-polynomially large square-free discriminant $D$ with a trapdoor $\tau_D$ so that solving discrete-log over $\CL(D)$ given $D$ is hard but feasible given the trapdoor $\tau_D$, then we might be able to maintain the security of the system even if $D$ and $h(D)$ are public. This would avoid all the complications in hiding the class group invariants. However currently we do not know of such a method, so we need to hide the class group invariants.

\paragraph{The possibility of recovering the discriminant $D$. }
Recall that for an elliptic curve over a finite field $\F_q$, the discriminant $D$ can be obtained by first computing the trace $t$ of Frobenius, and then computing the integers $v$ and $D_0$ such that $v^2 D_0 = t^2 - 4q$, and $D_0$ is square-free. Then $D = u^2 D_0$, where $u\mid v$. When $v$ is smooth $D$ can be recovered efficiently. 

Over $\Z/N\Z$, on the other hand, since we loose the notion of the Frobenius automorphism, we do not know how to apply the previous algorithm. Also, just a set of $j$-invariants over $\ZNZ$, corresponding to curves with the same endomorphism rings over $\F_p$ and $\F_q$, does not seem to allow us to recover the discriminant.

However, we do know that the range of $D$ is bounded by $|D|\leq 4p$ and $|D|\leq 4q$, so that $|D|<\sqrt{N}$. From the encodings we also learn the set $\mathcal{P}_B = \set{ p\leq B \mid p\text{ is a prime, } \kron{D}{p} = 1 }$, where $B\in \poly(\secp)$. 
This brings us to the following problem of independent interest: 
\begin{definition}\label{def:problem_recover_disc}
Given an integer $N$ and a set of primes $\mathcal{P} = \set{ p_1, ..., p_m  }$, find a negative value $D$ such that $|D|<\sqrt{N}$ and $\kron{D}{p} = 1$ for all $p\in\mathcal{P}$.
\end{definition}

Note that each condition $\kron{D}{p}=\pm1$ cuts the possible $D$'s roughly by half. Hence, for a sequence $\epsilon_1,\epsilon_2,\cdots,\epsilon_{\omega(\log(N))}$, $\epsilon_j\in\{\pm1\}$, with high probability there is at most one $D$ with $|D|<\sqrt{N}$ that satisfies $\epsilon_j=\kron{D}{p_j}$. To find such a $D$, on the other hand, is a separate problem and the only known methods require $\sqrt{|D|}$ of the values of the sequence $\{\epsilon_j\}$ (cf. pg. 6 of \cite{hoffstein2014ahist}, also see pg. 14 of loc. cit. and \cite{goldfeld1993on} for the analogue of the same problem in the context of modular forms)

We also remark that the following similar problem, first mentioned by Damg{\aa}rd \cite{DBLP:conf/crypto/Damgard88}, is conjectured to be hard: 
\begin{definition}[Problem P1 in \cite{DBLP:conf/crypto/Damgard88}]\label{def:problem_Legrendre_symbol}
The Legendre sequence with length $\ell$ and starting point $a$ is the $\pm 1$ sequence
\[  L = \kron{a}{p},  \kron{a+1}{p}, ..., \kron{a+\ell}{p} . \]
Given $L$ (with a polynomial $\ell$) but not $a$ and $p$, the problem asks to determine $\kron{a+\ell+1}{p}$.
\end{definition}




Even if Problem~\ref{def:problem_recover_disc} can be solved efficiently, an obvious safeguard is to choose less than $\log(|D|)$ many primes for the degrees of the isogenies involved in the composable encodings. However, it is not always possible to restrict the number of distinct prime degrees in applications. In that case, we can apply another safeguard by first choosing an enlarged set of $k$ polynomially large primes $R = \set{ p_1, ..., p_k }$, and letting $R'$ be a random subset of $R$. We can then set $D = (\prod_{p\in R'} p )^2  D_0$, where $D_0$ is a polynomial size fundamental discriminant. Then in the composable encodings, we avoid the primes from the entire set $R$ as the degrees of the isogenies.

\paragraph{The possibility of revealing $\Lattice$. }

Revealing any full-rank (not necessarily short) basis $\mat{B}$ of a relation lattice $\Lattice$ of dimension $w$ such that $\Z^{w}/\Lattice \simeq \CL(D)$ implies the disclosure of the class number $h(\OOO)=| \det(\mat{B})|$. So we should prevent leaking any full-rank bases of such lattices. 

An immediate consequence is that we cannot give out many non-trivial encodings of $1_\G$ (where $\G = \CL(\OOO)$) that form a full-rank basis of a potential relation lattice $\Lattice$. 
Recall that the trivial (i.e. canonical) encoding of $1_\G$ is $j_0$. A non-trivial encoding of $1_\G$ can be obtained by composing encodings which lead to a non-zero exponent vector. As an example, suppose that the encodings of $x, y, z\in\CL(\OOO)$ share the same prime generation set of norm from $L$, and $x\circ y \circ z = 1_\G$, and denote them by $\enc(v) = (L; T_{v, 1}, ..., T_{v, |L|})$, $v\in\set{x, y, z}$. Then, the following vector is in $\Lattice$
\[ \ary{e}_{x\circ y \circ z} = ( |T_{x, 1}|+|T_{y, 1}|+|T_{z, 1}|, ..., |T_{x, |L|}|+|T_{y, |L|}|+|T_{z, |L|}| ).\] 
Note that in the example we are not required to compute the composable encoding of $x\circ y \circ z$. We only need to read off the exponents from the lengths of $T_{*,\star}$. 

In the applications we do face the situation where the general security setting does not restrict the number of non-trivial encodings of $1_\G$. A countermeasure is to enforce each distinct non-zero encoding of $1_\G$ to have an new ideal with distinct prime norm, so that the dimension of the potential basis is larger than the possible number of vectors (relations) to be collected. We will illustrate how to apply this countermeasure in the applications.

Even for the encodings of non-identity elements, sampling the vector $\ary{e}$ from an arbitrary distribution over the cosets of $\Lattice$ might leak a basis. For example, \cite{DBLP:conf/eurocrypt/NguyenR06} shows that if there are enough short vectors from the parallelepiped of a short basis, then one can find the parallelepiped and, therefore, recover the basis.

The countermeasure is to sample the vector $\ary{e}_x\in\Z^{w_x}$ from the discrete-Gaussian distribution \cite{DBLP:conf/stoc/GentryPV08}. The sampler is known of being basis-independent. 

Formally, for any $\sigma \in\R^+$, $\ary{c}\in\R^n$, define the (non-normalized) Gaussian function with center $\ary{c}$ and standard deviation $\sigma$ for any $\ary{x}\in \R^n$ as $\rho_{\sigma, \ary{c}}(\ary{x}) = e^{-\pi \|\ary{x}-\ary{c}\|^2/\sigma^2}$.
For any $n$-dimensional lattice $\Lattice$, define the discrete Gaussian distribution over $\Lattice$ as:
\[ \forall \ary{x}\in\Lattice, ~ D_{\Lattice, \sigma, \ary{c}}(\ary{x}) 
        = \frac{\rho_{\sigma, \ary{c}}(\ary{x})}{\rho_{\sigma, \ary{c}}(\Lattice)}, \]
where $\rho_{\sigma, \ary{c}}(\Lattice) := \sum_{\ary{y} \in\Lattice} \rho_{\sigma, \ary{c}}(\ary{y})$ is the normalization factor.

\begin{lemma}[\cite{MicciancioRegev07}]\label{lemma:gaussianbound}
Let $\mat{B}$ be a basis of an $n$-dimensional lattice $\Lattice$, and let $\sigma \geq \|\tilde{\mat{B}}\|\cdot\omega( \log n)$, then $\Pr_{\ary{x}\la D_{\Lattice, \sigma, \ary{0}}}[ \|\ary{x}\|\geq \sigma\cdot \sqrt{n} \lor \ary{x}=\ary{0}]\leq \negl(n)$.
\end{lemma}

\begin{lemma}[\cite{DBLP:conf/stoc/GentryPV08,brakerski2013classical}]\label{lemma:discreteGaussianSamp}
There is a p.p.t. algorithm that, given a basis $\mat{B}$ of an $n$-dimensional lattice $\Lattice(\mat{B})$, $\ary{c}\in\R^n$, and $\sigma \geq \|\tilde{\mat{B}}\| \sqrt{\ln(2n + 4)/\pi}$, outputs a sample from $D_{\Lattice,\sigma,\ary{c}}$.
\end{lemma}

Note, however, that sampling from discrete-Gaussian distribution could lead to a vector with negative entries. Having negative entries as the exponent vector might trivially leak the inverse of other encodings (as mentioned in \S~\ref{sec:attack_concrete_TGII_trivial}). There are at least two solutions to this problem. One is to sample a vector $\ary{v}_0$ with positive large entries in $\Lattice$, then add $\ary{v}_0$ on any other vectors $\ary{v}'$ sampled from discrete-Gaussian so that $\ary{v}_0 + \ary{v}'$ have all positive entries (this solution does leak one vector $\ary{v}_0$ in $\Lattice$). The other solution is to pick at least one new prime ideal in the generation set of each composable encoding, so that the inverse won't be trivially obtained from the other encodings.







\paragraph{The possibility of leaking $h(D)$ from other sources. }

We have discussed the possibilities of leaking $D$ and a basis of $\Lattice_\OOO$ from the encodings and the corresponding countermeasures. It remains to check whether there are other possibilities of leaking $h(D)$. 

Again recall that in the basic parameter setting, $h(D)$ is set to be polynomially smooth. The immediate consequence is that in the encodings, we shall not use prime degrees $\ell$ such that the order of the ideal class $[(\ell, b, \cdot)]$ is polynomially large in $\CL(D)$, to avoid the risk of unnecessarily leaking any factors of $h(D)$.

Additionally, we ask:
\begin{enumerate}
\item Given an encoding, is it feasible to recognize that it encodes an element of polynomial order in the group $\CL(\OOO)$? 
\item Given an encoding of an element that is known to be of polynomial order, is the order explicitly revealed (instead of having an $1/\poly(\secp)$ chance of being guessed correctly)?
\end{enumerate} 

If self-composition is feasible, then the answers to both questions are yes. But for our construction, self-composition of the encoded group elements is infeasible. 

Denote the degree of an encoding $\enc(x)$ by $d$ and the order of $x$ in $\CL(\OOO)$ by $r$. Let the canonical encoding of $\enc(x)$ be $j_x$. In the special case where $d^r$ is polynomially large, then we can efficiently recognize that $\enc(x)$ has order $r$ by first guessing $d^r$ and then testing whether $\Phi_d(j_0, j_x) = 0\pmod{N}$, and $\Phi_{d^{r-1}}(j_0, j_x) = 0\pmod{N}$. The presence of such an encoding then leaks $r$ as a factor of the group order. 

One way of minimizing the possibility of having an element with small order is to choose the prime factors of $h(\OOO)$ to be as large as possible. For example, we can choose $D_0$ and the odd prime factors of $f_i - 1$ (where $f_i$ is a prime factor of the conductor) to be larger than $O(\secp^3)$.

\subsection{Miscellaneous}

\paragraph{On the decisional version of the inversion problem. }
Like the typical hard problems in cryptography, a search problem usually comes with a decisional variant. For the hardness of inversion, the natural way of defining the decision problem is to say that given a group element $x$, it is hard to decide whether a string $s$ represents $x^{-1}$ or a random group element. While in the ideal interface of (T)GII the decisional variant is easy, simply due to the fact that one can compose $x$ and $s$, then check if the result is equal to $1_\G$ or not. In our concrete instantiation, the following variant of the decisional problem has a chance to be hard: 
\begin{definition}[Decisional inversion problem]\label{def:decisionalGII}
Given the public parameter $\PP$ and a (composable) encoding $\enc(x)$, decide whether a string $s$ represents the canonical encoding of $x^{-1}$ (namely $j(x^{-1} * E_0 )$) or a random value in $(\ZNZ)^\times$. 
\end{definition}
Given that the canonical encoding is not composable, there is a chance that the decisional problem is hard.
However, when the degree of $\enc(x)$ is a polynomial $d$, Problem~\ref{def:decisionalGII} is easy, since we can test whether $\Phi_d(j_0, s) = 0\pmod{N}$ or not (in contrast, we conjecture the search problem is hard even for polynomial degree encodings). If the encoding is of super-polynomial degree, then the decisional problem seems hard. 

We remark that the applications of this paper do not necessarily rely on the hardness of the decisional version of the inversion problem, so the cryptanalysis effort we spend on the decisional problem is not as much as the search problem.

\paragraph{Rational points on $X_0(\ell)$.}
Rational points on $X_0(\ell)$ give solutions to $\Phi_{\ell}(x,y)=0$ over $\Z/n\Z$ for any $n\in\N$, unless the denominators of $x$ or $y$ are not in $(\Z/n\Z)^\times$. Recall that the genus of the modular curve $X_0(\ell)$ grows linearly with $\ell$, hence for large $\ell$, by Faltings' theorem, there are only finitely many rational points on $X_0(\ell)$. So for large $\ell$, one cannot hope to find $\ell$-isogenous neighbors over $\ZNZ$ by first finding points over $\Q$ and then reducing them mod $N$. For small $\ell$, on the other hand, there are rational points that seem to be easy to find. 

Let $P = (x, y)$ be such a ($\Q$-rational) point on $X_0(\ell)$, where neither $x$ nor $y$ is equal to $0$ or $1728$. The presence of these points, on one hand, allows one to easily find one $j$-isogenous neighbor of $x$ mod $N$ (if the denominator of $x$ is invertible in $\ZNZ$), that is $y$ mod $N$.

On the other hand, for an attack towards the $(\ell,\ell^2)$-neighbor problem one would like to find two $\Q$-rational points $P,Q$ sharing a common coordinate (i.e. either $x_P=x_Q$ or $y_P=y_Q$). We remark that having such $\Q$-rational points is highly unlikely for large $\ell$. More precisely, wlog assume that $x_P=x_Q$. Then $y_P$ and $y_Q$ correspond to the $j$-invariants of $\ell^2$-isogenous curves. They, moreover, are both in $\Q$, and hence the point $(y_P,y_Q)$ defines a $\Q$-rational point on $X_0(\ell^2)$. If $\ell>7$, $X_0(\ell^2)$ has genus strictly great than $1$ and therefore has finitely many rational points. Therefore, for $\ell>7$ even if one can find a pair of $Q$-rational points $P$ and $Q$ on $X_0(\ell)$, it is highly unlikely that they would share a common coordinate. 

Although we do not see an immediate attack through rational points on $X_0(\ell)$, we can always take $\ell>7$ for extra precaution.

\subsection{Summary}

We summarize the potential attacks and the countermeasures.

The adversary is given the public parameters $N$, $j_0$, and polynomially many composable encodings, each containing several $j$-invariants connected by isogenies of polynomial degrees. Following Proposition~\ref{claim:kernelpoly} we can recover all the explicit isogenies between neighboring $j$-invariants. Given that we choose all the prime degrees $\ell$ of the isogenies to be $\geq 5$ (or $>7$ for extra precaution), the explicit isogenies do not trivially leak the $x$-coordinates of the points in the kernel of the isogeny. 

We should avoid leaking the number of points on those elliptic curves over $\F_p$ and $\F_q$ with endomorphism ring $\OOO$. Otherwise the adversary can pick a random $x$ as the $x$-coordinate of a random point on the curve, and run Lenstra's factoring algorithm to factorize $N$. 

Central to the hardness of inversion is the $(\ell, \ell^2)$-isogenous neighbor problem over a composite modulus $N$ with unknown factorization. Among the potential solutions to the $(\ell, \ell^2)$-isogenous neighbor problem, finding the one corresponding to the image of a horizontal isogeny would break our candidate group with infeasible inversion, so it is worth investigating algorithms which find isogenies with specific directions. However, the only known such algorithm, that of \cite{DBLP:journals/moc/IonicaJ13}, does not seem to work over $\ZNZ$.

In the concrete instantiation of a group with infeasible inversion, the adversary can always try to reach new $j$-invariants by performing the legal GCD operations. Such attacks are captured under the model given in Definition~\ref{def:gcd_eval}. The parallelogram attack is a powerful attack under this model. It is effective when given the composable encodings of group elements $a$, $b$, $c$ such that $a\circ b = c$, and the degrees of the three encodings are coprime. This attack can be prevented when 2 out of the 3 pairs of the degrees of $\enc(a)$, $\enc(b)$ and $\enc(c)$ are not composable. In the applications of the directed transitive signature and the broadcast encryption, we are able to set the parameters to prevent the attack.

Under GRH the isogeny graph is an expander, which means given polynomially many composable encodings, it is reasonable to assume that the closure of the composition covers all the $h(D)$ $j$-invariants. However, finding a path of composition to reach a specific point (e.g. the inverse of some encoding) may still require solving the discrete-log problem over $\CL(D)$. Under the current choice of parameters, the discrete-log problem over $\CL(D)$ can be solved efficiently once $h(D)$ is given, and $h(D)$ can be recovered efficiently given $D$ or any relation lattice $\Lattice$ of dimension $w$ such that $\Z^{w}/\Lattice \simeq \CL(D)$. So we should prevent leaking any of $\Lattice$, $D$, or $h(D)$.

The discriminant $D$ cannot be polynomially large for two reasons: First, if $D\in\poly(\secp)$, then we can guess $D$ and compute the class number in polynomial time, which enables us to efficiently solve the discrete-log problem over $\CL(D)$. Second, we can also compute the Hilbert class polynomial $H_D$ in polynomial time and therefore solve the $(\ell, \ell^2)$-isogenous neighbor problem. 

Even if $D$ is super-polynomial, it should be kept hidden. 
Since $D$ has to be hidden, we cannot give out a plaintext group element of $\CL(D)$ that allows efficient group operation, e.g. the quadratic form representation of an ideal class $[(a, b, c)]$ (cf. Section~\ref{sec:ICG}), since the discriminant $D$ can be read off from $b^2 - 4ac$.

Finally, we also remark that if we were able to choose a super-polynomially large square-free discriminant $D$ with a trapdoor $\tau_D$, so that solving discrete-log over $\CL(D)$ is hard if one is given only $D$, but feasible if one is given the trapdoor $\tau_D$, then we can construct a system with $D$ and $h(D)$ being public, which would avoid all the complications in hiding $D$ and $h(D)$. 


\section{Directed transitive signature for directed acyclic graphs}\label{sec:DTS}

The concept of transitive signature (DTS) was introduced by Rivest and Micali \cite{DBLP:conf/ctrsa/MicaliR02a}. 
In a transitive signature scheme the master signer is able to sign on the edges of a graph $G$ using the master signing key. Given the signatures on a specific set $S$ of edges, say $S = \set{ (u, v), (v, w)}$,  everyone can compute the signature on the edge $(u, w)$, and in general, any edge in the transitive closure of $S$, but not for any edge beyond the transitive closure of $S$. Transitive signatures are useful in the scenario where the graph represents some authorization relationship. The master signer has limited availability and has the intention of signing a limited number of edges ahead of time. New users can then dynamically join the graph, build edges, and obtain the composed signatures if the edges live in the transitive closure of the existing ones.

Transitive signatures for undirected graphs are constructed in \cite{DBLP:conf/ctrsa/MicaliR02a,bellare2002transitive} and many others. But for directed graphs, only the special case of directed trees was achieved by Yi \cite{DBLP:conf/ctrsa/Yi07} from the RSA assumption. However, Neven later gave a construction of DTS for directed trees from any standard digital signature, which shows that directed trees are indeed simpler to achieve \cite{neven2008simple}.


We instantiate DTS from TGII for any directed acyclic graph (DAG). The basic instantiation directly follows the construction from \cite{hohenberger2003cryptographic,molnar2003homomorphic}, where the signatures are the composable encodings of some group elements. Since the lengths of the composable encodings in our TGII keep growing during composition, the lengths of the composed signatures grow proportionally to the number of individual signatures. 
However, a DTS where the lengths of composed signatures grow proportionally to the number of individual signatures is trivially achievable -- simply use any signature scheme and interpret the concatenation of a signature on the edge $\overrightarrow{ij}$ and a signature on the edge $\overrightarrow{jk}$ as the signature on the edge $\overrightarrow{ik}$ -- as was explicitly mentioned in \cite[Page~14]{molnar2003homomorphic}.

To provide a DTS scheme with short composed signatures, we add a signature compressing step using the partial extraction algorithm (Algorithm~\ref{alg:partialconvert}) described in the general version of TGII in Section~\ref{sec:TGIIconstruction_general}. However, the compressed signature cannot be further composed with other signatures, so the scheme has not yet met the ideal definition of DTS provided by \cite{hohenberger2003cryptographic,molnar2003homomorphic}. Still, being able to compress the final composed signature is desirable for a DTS, and it does not seem to be trivially achievable.


\subsection{Definition}

We adopt the definition of a directed transitive signature from \cite{hohenberger2003cryptographic,molnar2003homomorphic}.
\begin{definition}
A directed transitive signature scheme $\DTS = (\Gen, \Cert, \Sig, \comp, \Ver)$ consists of the following tuple of efficient algorithms:
\begin{itemize}
\item $\Gen$: 
	The key generation algorithm $\Gen$ takes as input the security parameter $1^\secp$, returns the master public-key secret-key pair $(\MPK,\MSK)$.
\item $\Cert$: 
	The node certification algorithm $\Cert$ takes as input the master secret key $\MSK$ and a node $i\in\N$, returns $(\PK(i), \SK(i))$, a public value and a secret value for node $i$. 
\item $\Sig$: 
	The edge signing algorithm $\Sig$ takes as input the master secret key $\MSK$, the source node $i$ and destination node $k$, and the associated $(\PK(i), \SK(i))$, $(\PK(k), \SK(k))$, outputs a signature $\sigma_{i,k}$ of the edge $\overrightarrow{ik}$.
\item $\comp$: 
	The composition algorithm $\comp$ takes as input $\MPK$, two consecutive edges $\overrightarrow{ij}$, $\overrightarrow{jk}$, and the signatures $\sigma_{i,j}$, $\sigma_{j,k}$, outputs a signature $\sigma_{i,k}$ for the edge $\overrightarrow{ik}$. 
\item $\Ver$: 
	The verification algorithm $\Ver$ takes as input the edge $\overrightarrow{ik}$ and its public values $\PK(i), \PK(k)$, and a potential signature $\sigma'$, returns $1$ iff $\sigma'$ is a valid signature of the edge $\overrightarrow{ik}$, $0$ otherwise.
\end{itemize}

In this paper we will be considering the simplest definitions of correctness and security (we refer the readers to \cite{hohenberger2003cryptographic,molnar2003homomorphic} for the more formal definitions). 
For correctness we require that the verification algorithm outputs $1$ on all signatures obtained from compositions in the transitive closure of edges signed by the master signing key. 
For security we require that it is infeasible for any p.p.t. adversary to forge a signature beyond the transitive closure of the signed edges. 
The adversary is allowed to dynamically add nodes and edges in the graph, and is allowed to request the master signer to sign as long as the target edge for forgery does not trivially fall in the transitive closure of the signed edges.
\end{definition}

\subsection{A directed transitive signature scheme for DAGs via TGII}

Assume the DAG in the protocol has $n$ vertices $[1, 2, ..., n]$. All the directed edges $\overrightarrow{ik}$ are heading from a lower index to a higher index, i.e. $i<k$ (any DAG has at least one topological ordering).

\begin{construction}[DTS via TGII]\label{cons:DTS:ideal}
Given a trapdoor group with infeasible inversion $\TGII$, a standard digital signature scheme $\DS$, construct a directed transitive signature $\DTS$ as follows:
\begin{itemize}
\item $\Gen$: 
	The key generation algorithm $\Gen$ takes as input the security parameter $1^\secp$ and
	runs the TGII parameter generation algorithm to produce the public parameters $\TGII.\PP$ and the trapdoor $\tau$ of a group $\G = ( \circ, 1_\G )$.  
	It also generates the signing and verification keys for the regular signature scheme $\DS.\Gen(1^\secp)\to \DS.\SK, \DS.\VK$. 
	It returns the master public-key $\DTS.\MPK = (\TGII.\PP, \DS.\VK )$ and the master secret-key $\DTS.\MSK = (\tau, \DS.\SK)$.
\item $\Cert$: 
	The node certification algorithm $\Cert$ takes as input the master secret key $\DTS.\MSK = (\tau, \DS.\SK)$ and a node $i\in\N$, samples a random element $x_i\in\G$ and a composable encoding $\enc(x_i)$, and sets the encoding to be the public information for $i$: 
	\[ \PK(i) := \enc(x_i) = \TGII.\Trap\Sam(\TGII.\PP, \tau, x_i).\] 
	The secret information on node $i$ can be set as $\SK(i) = x_i^{-1}$, or simply be left as $\bot$ since the master trapdoor holder can invert $\PK(i) = \enc(x_i)$ to get $x_i^{-1}$. $\Cert$ then produces the signature $\Sigma(i) = \DS.\Sig( \DS.\SK, i||\PK(i) )$ and takes $\Sigma(i)$ as the certificate of node $i$.
\item $\Sig$: 
	The edge signing algorithm $\Sig$ takes as input the master secret key $\DTS.\MSK = (\tau, \DS.\SK)$, a source node $i$, a destination node $k$, and the associated $(\PK(i), \SK(i))$, $(\PK(k), \SK(k))$. It first verifies the node certificates, then recovers $x_i^{-1}$ from $\SK(i)$ and $x_k$ from $\PK(k)$. It then outputs 
	\[ \sigma_{i,k}: = \enc(x_i^{-1}\circ x_k) = \TGII.\Trap\Sam( \TGII.\PP, \tau, x_i^{-1}\circ x_k ). \] 
\item $\comp$: 
	The composition algorithm $\comp$ takes as input $\MPK$, two consecutive edges $\overrightarrow{ij}$, $\overrightarrow{jk}$, and the signatures $\sigma_{i,j}$, $\sigma_{j,k}$, outputs $\sigma_{i,j}\circ \sigma_{j,k}$ as the composed signature for the edge $\overrightarrow{ik}$. 
\item $\Ver$: 
	The verification algorithm $\Ver$ takes as input the edge $\overrightarrow{ik}$ and its public values $\PK(i), \PK(k)$, and a potential signature $\sigma'$. It parses the public information as $(i, \enc(i), \Sigma(i))$, $(k, \enc(k), \Sigma(k))$. If any of the certificates on the nodes is invalid, it returns 0. 
	Otherwise, it checks whether 
	\[  \convert(\TGII.\PP, \enc(i)\circ\sigma') = \convert(\TGII.\PP, \enc(k)) .\] 
	If so it returns 1, otherwise it returns 0. 
\end{itemize}
\end{construction}

\paragraph{Choosing the generation sets $S$.}
Now we provide a detailed instantiation of the DTS using our basic version of the TGII (cf. \S~\ref{sec:TGIIfromisogeny}). 
The only parameters left to be specified are the primes $\set{p_{i}}$ of the generation set $S = \set{ C_{i} = [(p_{i}, b_{i}, \cdot)] }_{i\in[ w ]}\subset \G = \CL(\OOO)_{\text{odd}}$ in the first step of the (stateful) encoding sampling algorithm (cf. Algorithm~\ref{alg:represent_enc}).
To do so we need to clarify which encodings are composable in the system and which are not (for functionality); whether there are pairwise composable encodings of $a$, $b$, $c$ that satisfy $a\circ b = c$ (to prevent the parallelogram attack detailed in \S~\ref{sec:parallelogram_attack}); and whether we are able to hide the bases of the relation lattices (to prevent the attack detailed in \S~\ref{sec:attack:basis}).

One way of choosing the generation sets is as follows: Let the following be sets of $\OOO$-ideals of distinct prime norms, each set is of size $O(\log(\secp))$: $S_{i,src}$, $S_{i,dst}$, for all $i\in[n]$, and $S_{common}$.
\begin{itemize}
	\item Let the generation set for $\PK(i) = \enc(x_i)$ be $S_{common}\cup S_{i,dst}$, plus another $\OOO$-ideal of prime norm $\ell_i$ that has never been used.	
	\item Let the generation set for $\sigma_{i,k} = \enc(x_i^{-1}\circ x_k)$ be $S_{i,src}\cup S_{k,dst}$, plus another $\OOO$-ideal of prime norm $\ell_{i,k}$ that has never been used.
\end{itemize}

Let us first check correctness of our construction: For each index $i\in[n]$, the public key on node $i$, $\PK(i) = \enc(x_i)$, should be composable with all the edges $\overrightarrow{jk}$ with $j\geq i$ (for correctness it is not important if $\PK(i)$ is composable with any other edges or vertices).
A signature $\sigma_{i,k} = \enc(x_i^{-1}\circ x_k)$ should be composable with any signature $\sigma_{i',k'}$ such that $i'\geq k$, or $k'\leq i$, and with any public value $\PK(h)$ such that $h\leq i$, but not with any other vertices or edges. 
So our assignment of the prime ideals supports the necessary composition functionalities.

We now verify that the parallelogram attack from \S~\ref{sec:parallelogram_attack} does not apply to our construction: In the DTS scheme there are two ways of obtaining composable encodings of $a, b, c$ such that $a\circ b = c$.
\begin{enumerate}
\item Through the compositions of the signatures on the edges $\overrightarrow{ij}$, $\overrightarrow{jk}$, $\overrightarrow{ik}$, which gives 
\begin{equation}\label{eqn:DTS1}
\sigma_{i,j} \circ \sigma_{j,k} = \sigma_{i,k} \Rightarrow 
\convert(\TGII.\PP, \enc(x_i^{-1}\cdot x_j) \circ \enc(x_j^{-1}\cdot x_k)) = 
\convert(\TGII.\PP, \enc(x_i^{-1}\cdot x_k)).
\end{equation}
Our choices of the generation sets make sure that $\sigma_{j,k}$ and $\sigma_{i,k}$ are not composable due to the common $S_{k,dst}$; $\sigma_{i,j}$ and $\sigma_{i,k}$ are not composable due to the common $S_{i,src}$. 
\item Through the verification algorithm on the public values $\PK(i)$, $\PK(k)$, and the signature on $\overrightarrow{ik}$, which gives 
\begin{equation}\label{eqn:DTS2}
\PK(i) \circ \sigma_{i,k} = \PK(k) \Rightarrow 
\convert(\TGII.\PP, \enc(x_i) \circ \enc(x_i^{-1}\cdot x_k)) = \convert(\TGII.\PP, \enc(x_k)). 
\end{equation}
Our choices of the generation sets make sure that $\PK(i)$ and $\PK(k)$ are not composable due to the common $S_{common}$; $\sigma_{i,k}$ and $\PK(k)$ are not composable due to the common $S_{k,dst}$. 
\end{enumerate}
Therefore, in both cases, two of the three pairs of the degrees share prime factors. Therefore the parallelogram attack does not apply.

Finally, we verify that it is unlikely to leak a full-rank basis of any relation lattice of $\CL(\OOO)$. Note that the only two ways of obtaining non-trivial encodings of the identity are mentioned in Eqn.~\eqref{eqn:DTS1} and Eqn.~\eqref{eqn:DTS2}. Due to the fresh prime ideal inserted in the generation set of each encoding, there are always more dimensions than the number of linearly independent relations in a potential relation lattice, which means it is unlikely to obtain any full-rank basis of such a lattice.

\begin{remark}[Why not supporting waiting signatures]\label{remark:whynotws}
The definition of DTS from \cite{hohenberger2003cryptographic} additionally allows the signatures to be signed, composed, or verified over non-consecutive edges. Such signatures are called \emph{waiting signatures}. 
The DTS construction from ideal TGII in \cite{hohenberger2003cryptographic,molnar2003homomorphic} naturally supports waiting signatures. 
To do so, the verification algorithm $\Ver$ takes as input two multisets of sources and destinations $P_{src} = \set{a_1, ..., a_z}$, $P_{dst} = \set{b_1, ..., b_z}$, and a potential signature $\sigma'$.  
It checks whether 
\[  \convert(\TGII.\PP, \sigma' \circ \enc(a_1) \circ  ... \circ \enc(a_z)) = \convert(\TGII.\PP, \enc(b_1) \circ  ...\circ \enc(b_z)) .\] 
If so returns 1, otherwise returns 0.   

Our TGII construction is able to support waiting signatures in terms of functionality, but not security. The reason is: to verify waiting signatures, we would have to make sure that all the public values of the vertices are composable. But then the parallelogram attack would apply through Eqn.~\eqref{eqn:DTS2}. The attacker would then be able to recover the encoding of $x_i^{-1}$, which is the secret-key on node $i$.
\end{remark}

\paragraph{Compressing the composed signature.}
The definition of \cite{hohenberger2003cryptographic,molnar2003homomorphic} additionally requires that for any consecutive edge in the graph, the composed signature should be indistinguishable from a signature produced by the master signer. We remark that the DTS from our basic version of the TGII (cf. \S\ref{sec:TGIIfromisogeny}) does not achieve this property, since the signatures grow as they are composed\footnote{The same happens in the DTS from the TGII implied by the self-bilinear maps with auxiliary inputs \cite{DBLP:conf/crypto/Yamakawa0HK14}, where the auxiliary input in the composed encoding keeps growing.}.

Instead, we will provide a signature compression technique using the partial extraction algorithm from the general version of TGII in \S\ref{sec:TGIIconstruction_general}. The compression can only be applied to a composed signature $\sigma_{i,k}$ which can be verified (given $\PK(i)$ and $\PK(k)$) but can no longer be composed with other signatures. Our construction, therefore, does not fully achieve the indistinguishability definition of \cite{hohenberger2003cryptographic,molnar2003homomorphic}, in the sense that a compressed signature cannot be composed with others later.

To obtain the compressed signature, we apply $\Partial.\convert(\TGII.\PP, \sigma_{i,k}, \PK(i))$ to obtain the $j$-invariant $j_{\overrightarrow{ik}} = \convert(\TGII.\PP, \sigma_{i,k})$, the isogeny $\phi_i: j_{\overrightarrow{ik}}\to j_{x_k}$ that represents $\PK(i)$, and a SNARG $\pi$ for the instance $(\PK(i), \phi_i)$ and statement ``there exists a string $\mathsf{str}$ such that $\phi_i$ is computed from running a circuit that instantiates Algorithm~\ref{alg:partialconvert} on inputs $\mathsf{str}, \phi_i$''. Let $(j_{\overrightarrow{ik}}, \phi_i, \pi)$ be the compressed signature. Note that its size depends only on the security parameter.  

We change the verification algorithm as follows: $\Ver$ takes as input a source node $a$, a destination node $b$, and a potential compressed signature $(j, \phi, \pi)$ for an edge from $a$ to $b$. It parses the public information as $(a, \enc(a), \Sigma(a))$, $(b, \enc(b), \Sigma(b))$ and if any of the certificates on the nodes is invalid it returns 0. Otherwise, it checks:
\begin{itemize}
\item Whether $\phi(j) = \convert(\TGII.\PP, \enc(b))$,
\item Whether $\pi$ is a valid proof for the following statement: there is a string $\mathsf{str}$ such that $\phi$ is obtained by applying Algorithm~\ref{alg:partialconvert} over the input $\mathsf{str}$, $\enc(a)$.
\end{itemize}
If both of the checks pass, returns 1, otherwise returns 0. 

The correctness of the verification algorithm follows from the correctness of the encoding algorithm, the algorithm $\Partial.\convert$, and the completeness of the SNARG. 
The soundness of SNARG guarantees that the isogeny $\phi$ is obtained honestly from applying $\Partial.\convert$ on some string $\mathsf{str}$ and $\enc(a)$, which means $\phi$ correctly represents $\PK(a)$ and therefore $j$ must be $\convert(\TGII.\PP, \enc(a^{-1}\circ b))$.

Furthermore, we can use a zero-knowledge SNARG (from e.g. \cite{groth2016size}), so that the compressed signature is indistinguishable from a compressed signature produced by the master signer.

\section{Broadcast encryption}\label{sec:BE}

A broadcast encryption scheme allows the encrypter to generate ciphertexts that are decryptable by a designated subset of the receivers. A trivial solution for broadcast encryption is to simply encrypt the message many times for each user's decryption key. Although this does provide a solution, it clearly is not effective as the number of receivers grow. A meaningful broadcast encryption scheme needs to be more efficient on at least one of: the cost of encryption, size of the public parameters, or the sizes of users' decryption keys (cf. \cite{DBLP:conf/crypto/FiatN93}). 

In this section we give a concrete instantiation of the broadcast encryption scheme of Irrer et al. \cite{ILOP04}, which was designed under the ideal interface of GII. Our instantiation supports private-key encryption and allows any numbers of users to collude. The encryption overhead and the users' secret keys are independent of the total number of the users $n$, however the public parameter size is linear in $n$. 
Let us remark that our aim here is to give a concrete application of our TGII construction and hence we did not try to optimize the public parameter overhead. 
Let us also mention that the asymptotic efficiency of our scheme has been achieved by the construction in \cite{DBLP:conf/crypto/BonehGW05} using bilinear maps. There are other schemes in literature that achieve $\log(n)$ public parameter overhead under the assumption of the security of multilinear maps \cite{DBLP:conf/crypto/BonehWZ14} or iO \cite{DBLP:journals/iacr/Zhandry14a}. We leave the challenge of achieving smaller public parameter blowup from GII or TGII to the interested reader.

\subsection{Definition}

A private-key broadcast encryption scheme consists of a tuple of efficient algorithms:
\begin{itemize}
\item $\Setup(1^\secp)$: The setup algorithm takes as input the security parameter $1^\secp$, then generates the public parameters $\PP$ and the master secret key $\MSK$. 

\item $\Gen(\PP, \MSK, u)$: The user key generation algorithm takes as input the user id, $u$, in the list of users $\mathcal{U}$. It generates a secret key $\SK_u$ and (possibly) a public key $\PK_u$ for user $u$. $\PK_u$ is included in the public parameters.

\item $\Enc(\PP, \MSK, \Gamma, m)$: The encryption algorithm takes as input a polynomial size set $\Gamma \subseteq \mathcal{U}$ of recipients. It applies a deterministic algorithm over $\PP, \MSK, \Gamma$ to derive a key $K$ of a symmetric-key encryption scheme. Denoting the encryption of the message $m$ under $K$ by $\CT_{K, m}$, it outputs $\CT = (\Gamma, \CT_{K, m})$.

\item $\Dec(\PP, u, \SK_u, \CT)$: The decryption algorithm first parses $\CT$ as $(\Gamma, \CT_{K, m})$ and derives the symmetric-key $K$ from $\PP, \SK_u$, and $\Gamma$. It then uses $K$ to decrypt $\CT_{K, m}$.
\end{itemize}

The scheme is said to be correct if for all subsets of the users $\Gamma\subseteq\mathcal{U}$ and for all messages $m$, any user $u$ from the set $\Gamma$ decrypts the correct message, i.e. $ \Dec(\PP, u, \SK_u, \Gamma, \Enc(\PP, \MSK, \Gamma, m)) = m$. 

For security we consider the simplest form of the key-recovery attack, defined by the following ``key-recovery" game between an adversary and a challenger.
\begin{enumerate}
\item The challenger runs the setup algorithm to generate a master secret key $\MSK$ and the public parameters. The challenger then picks a set of users $\mathcal{U}$ and generates the public keys and the secret keys of the users. The adversary is given $\mathcal{U}$ and all the public parameters.
\item The adversary picks a subset $\Gamma \subseteq\mathcal{U}$ of users where it wants to attack. The challenger gives the adversary all the secret keys for the users not in $\Gamma$, i.e. $\SK_u$ for $u\notin\Gamma$.
\item The adversary can make encryption queries on any message $m$ and any subset $\Gamma'\subseteq\mathcal{U}$ (including $\Gamma' = \Gamma$). The challenger runs the Encrypt algorithm to obtain $CT = \Enc(\PP, \MSK, \Gamma', m)$ and sends to the adversary. 
\item The adversary outputs a symmetric key $K^*$ and wins the game if $K^*$ is equal to the symmetric key $K$ derived from $\PP$, $\MSK$ and $\Gamma$, loses otherwise.
\end{enumerate}
A broadcast encryption scheme is said to be secure if for any polynomial time adversary the odds of winning the key-recover game is negligible.

\begin{figure}
\begin{center}
\begin{tabular}{lllllll}\hline
Construction  						& $|\text{param}|$ & $|\text{user key}|$  & $|\text{CT}|$          & Based on    & PK & CR \\\hline
\cite{DBLP:conf/crypto/FiatN93}     & $O(t^2 n \log n )$  & $O(t \log^2 t \log n)$ & $O(t^2 \log^2 t \log n)$ & RSA assumption           & No  & $\leq t$ users \\
\cite{ILOP04} 						& $O(n)$           & $O(1)$               & $O(1)$                 & Ideal GII   & No & Arbitrary \\
\cite{DBLP:conf/crypto/BonehGW05} I & $O(n)$           & $O(1)$               & $O(1)$                 & Bilinear maps & Yes & Arbitrary \\
\cite{DBLP:conf/crypto/BonehGW05} II& $O( \sqrt{n} )$  & $O(\sqrt{n})$        & $O(1)$                 & Bilinear maps & Yes & Arbitrary \\
\cite{DBLP:conf/crypto/BonehWZ14} 	& $O(\log n)$      & $O(\log n)$          & $O(\log n)$            & Mmaps 	& Yes & Arbitrary    \\
\cite{DBLP:journals/iacr/Zhandry14a}& $O(\log n)$      & $O(\log n)$          & $O(\log n)$            & iO 	& Yes & Arbitrary    \\
This work     						& $O(n)$           & $O(1)$               & $O(1)$                 & A concrete TGII & No  & Arbitrary  \\
\end{tabular}
\caption{A brief summary of the existing collusion resistant broadcast encryption schemes. $n$ represents the number of users; ``PK" stands for supporting public key encryption; ``CR" stands for being collusion resistant. All the parameters ignore the (possibly multiplicative) $\poly(\secp)$ factors. }
\end{center}
\end{figure}

\subsection{A private-key broadcast encryption scheme from TGII}

We first present the construction from an abstract TGII, then specify the parameters in detail.

\begin{construction}[Broadcast encryption under an abstract TGII]\label{cons:ideal}
Given a trapdoor group with infeasible inversion $\TGII$ and a symmetric-key encryption scheme $\Sym$, construct a broadcast encryption scheme $\BE$ as follows:
\begin{itemize}
\item $\BE.\Setup(1^\secp)$: The setup algorithm takes as input the security parameter $1^\secp$, runs the TGII parameter generation algorithm to produce the public parameters $\TGII.\PP$ and the trapdoor $\tau$ of a group $\G = ( \circ, 1_\G )$. Then sample a random element $s\in\G$. 
The public parameter $\BE.\PP$ is set to be $\TGII.\PP$. 
The master secret key $\BE.\MSK$ includes $\tau$ and $s$.

\item $\BE.\Gen(\BE.\PP, \BE.\MSK, u)$: The user secret key generation algorithm parses $\tau, s$ from $\BE.\MSK$. It samples a random element $x\in\G$, computes $\enc(x)\la \TGII.\Trap\Sam(\TGII.\PP, \tau, x)$ as the user's public key $\PK_u$; computes $\enc(s\circ x)\la \TGII.\Trap\Sam(\TGII.\PP, \tau, s\circ x)$ and treats it as the user's secret key $\SK_u$.

\item $\BE.\Enc(\BE.\PP, \BE.\MSK, \Gamma,m)$: The encryption algorithm takes as input a polynomial size set $\Gamma \subseteq \mathcal{U}$ of recipients and a message $m$. It first computes the message encryption key $K = \left( \prod_{i \in \Gamma} x_i \right)\circ s $, then computes $\Sym.\CT_{K, m}:= \Sym.\Enc(K, m)$ and outputs $\BE.\CT = (\Gamma, \Sym.\CT_{K, m})$. 

\item $\BE.\Dec(\BE.\PP, u, \SK_u, \CT)$: The decryption algorithm extracts the set $\Gamma$ from $\BE.\CT$, takes the public keys $\PK_i$ for users $i\in\Gamma\setminus{u}$, the secret key $\SK_u$ for user $u$, and computes $K' = \convert(\TGII.\PP, \prod_{ i\in\Gamma\setminus{u} }(\PK_i)\circ \SK_u) $. It then decrypts $\Sym.\CT_{K, m}$ using $K'$.
\end{itemize}
\end{construction}
To be more cautious, we can also apply a randomness extractor on $K$ to derive a key that is statistically close to the uniform in distribution.

\paragraph{Choosing the generation sets.}
We now provide the detailed instantiation of the broadcast encryption scheme using our basic version of the TGII (cf. Section~\ref{sec:TGIIfromisogeny}). 
Similar to the situation of the DTS, the only parameters left to be specified are the primes $\set{p_{i}}$ in the generation set $S = \set{ C_{i} = [(p_{i}, b_{i}, \cdot)] }_{i\in[ w ]}\subset \G = \CL(\OOO)_{\text{odd}}$ in the first step of the (stateful) encoding sampling algorithm (cf. Algorithm~\ref{alg:represent_enc}).

One way of choosing the primes for the generation sets is as follows. Choose sets of $\OOO$-idelas $S_{i}$ for $i\in \mathcal{U}$ and $S_{msk}$ such that they consist of ideals of distinct prime norms and each one of them is of size $O(\log(\secp))$.	
\begin{itemize}
	\item Let the generation set of $\PK_i = \enc(x_i)$ be $S_{i}$, plus another $\OOO$-ideal of prime norm $\ell_i$ that has never been used.	
	\item Let the generation set of $\SK_i = \enc(x_i \circ s)$ be $S_{i}\cup S_{msk}$, plus another $\OOO$-ideal of prime norm $\ell'_i$ that has never been used.	
\end{itemize}

Let us first check the correctness. The secret key $\SK_i$ does not have to be composable with any other secret keys, it has to be composable with all the public keys except $\PK_i$. 
The $\PK_i$ has to be composable with everything other than $\SK_i$. 
So our assignment of the prime ideals supports the necessary composition functionalities.

We will now consider the parallelogram attack from Section~\ref{sec:parallelogram_attack}. Since the only parameters in the encryption scheme are $\SK_u$ and $\PK_v$ for $u,v\in \mathcal{U}$, the only way we can obtain a identity of the type $a\circ b = c$ is through the relations 
\begin{equation}\label{eqn:BE}
\PK_i \circ \SK_k = \PK_k \circ \SK_i,
\end{equation}
where the public-key secret-key pairs are $\PK_i = \enc(x_i)$, $\SK_i = \enc(x_i \circ s)$, $\PK_k = \enc(x_k)$, $\SK_k = \enc(x_k \circ s)$. The attack scenario in Eqn.~\eqref{eqn:BE} may happen when Users $i$ and $j$ collude to find the secret keys of other users.

Our choices of the generation sets make sure that $\SK_i$ and $\SK_k$ are not composable for any $i,k$ due to the common $S_{msk}$. Similarly, $\SK_i$ and $\PK_i$ are not composable for any $i$ due to the common $S_{i}$. So there are two out of three pairs of the degrees that share prime factors. Therefore, our instantiation is secure to the parallelogram attack.

Finally we check whether there is a chance of leaking a full-rank basis of any relation lattice of $\CL(\OOO)$. Note that the only two ways of obtaining non-trivial encodings of the identity are mentioned in Eqn.~\eqref{eqn:BE}. Due to the new prime ideals inserted in the generation set of each encoding, there are always more dimensions than the number of linearly independent relations in a potential relation lattice, which means it is unlikely to obtain any full-rank basis of such a lattice.

\begin{remark}
The scheme we presented above makes a slight change over the original scheme of Irrer et al. \cite{ILOP04}. 
In \cite{ILOP04} the user's secret key is $x$ and the user's public key is $s\circ x$. The encryption key with respect to the set $\Gamma$ is $K = \left( \prod_{i \in \Gamma} x_i \right)\circ s^{|\Gamma|-1} $. 
Whereas in our scheme we flip the public key and secret key, and change the encryption key accordingly. 
The purpose of the change is to provide a candidate instantiation where only the short-basis of the relation lattice for $S_{msk}$ has to be computed.

Note that this change does not affect the security analysis from \cite{ILOP04}, which shows that a key recovery attack to the broadcast encryption scheme implies the ability of computing inverses.
\end{remark}

\begin{remark}
The CPA-style security definition in \cite{DBLP:conf/crypto/BonehGW05} requires that it is hard to distinguish a correct decryption key from a random key. Our scheme is also a candidate that satisfies the CPA definition when the decisional inversion problem is hard, which is plausible under the current setting of parameters.
\end{remark}

\section{Future directions}

We conclude the paper with several future directions.

\paragraph{Further investigation of the $(\ell, \ell^2)$-isogenous neighbors problem.} 
The hardness of the $(\ell, \ell^2)$-isogenous neighbor problem over $\ZNZ$ is necessary for the security of our candidate trapdoor group with infeasible inversion. Let us remark that once the adversary is given two $j$-invariants that are $\ell$-isogenous, she can recover the rational polynomial of the isogeny, which is not available in the $\ell$-isogenous neighbor problem where the adversary is given a single $j$-invariant as input. Although it is not clear how to use the explicit rational polynomial of the isogeny to mount an attack, it should serve as a warning sign that the $(\ell, \ell^2)$-isogenous neighbors problem might be easier than the $\ell$-isogenous neighbors problem.

\paragraph{Proving security in the GCD evaluation model.}
The GCD evaluation model (cf. Definition~\ref{def:gcd_eval}) provides a simplified interface for the attacker. It will be interesting to prove (or disprove) that the TGII is secure if the adversary is restricted to follow the GCD evaluation model. 

\paragraph{Looking for alternative constructions of GII or TGII.}
Given the complications and the limitations of our construction of TGII, one might want to look for a simpler or a different construction of GII or TGII. Some concrete directions to study further are:
\begin{enumerate}
\item A construction, where the encoding is stateless. 
\item A construction, where the size of the encoding does not grow with composition. 
\item A (T)GII candidate for a non-commutative group. 
\end{enumerate}

\section*{Acknowledgments}
We would like to thank Andrew V. Sutherland, and the anonymous reviewers for their helpful comments.

\bibliography{ref}

\end{document}